\documentclass[12pt,oneside,reqno]{article}

\usepackage{amsfonts}
\usepackage[all]{xy}
\usepackage{geometry}
\usepackage{amsmath}
\usepackage[latin5]{inputenc}
\usepackage{amssymb}
\usepackage{ntheorem}
\usepackage{geometry}
\usepackage{setspace}
\usepackage[all]{xy}
 \usepackage{graphics} 
  \usepackage{epsfig} 
 \usepackage[colorlinks=true]{hyperref}

\newtheorem*{proposition}{Proposition}

\newtheorem*{remark}{Remark}
\newtheorem*{proof}{Proof}

\begin{document}
\title{Reductions of Dynamics on Second Iterated Bundles of Lie Groups}

\author{O\u{g}ul Esen$^{1}$ and Hasan G\"{u}mral$^{2}$\footnote{On leave of absence
from Department of Mathematics, Yeditepe University}}
\maketitle
\begin{center}
$^{1}$Department of Mathematics, Yeditepe University \\
oesen@yeditepe.edu.tr \\
$^{2}$Department of Mathematics, Australian College of Kuwait\\
h.gumral@ack.edu.kw
\end{center}

\begin{abstract}
We consider trivializations of second iterated bundles of a Lie group that
preserve lifted group structures. With such a trivialization, we elaborate
Hamiltonian dynamics on cotangent, Lagrangian dynamics on tangent bundles
and, both Hamiltonian and Lagrangian dynamics on Tulczyjew's symplectic
space which is tangent of cotangent bundle of Lie group. We present all
possible Poisson, symplectic and Lagrangian reductions of spaces and
corresponding dynamics on them. In particular, reduction of Lagrangian
dynamics on second iterated tangent bundle includes reduction of dynamics on
second order tangent bundle.
 MSC2000; Primary: 70H50, 70G65, 53D20; Secondary: 53D17, 70H30.
\end{abstract}

\section{Introduction}

One observes that, the form of equations governing dynamics on Lie
groups depends on the kind of trivializations adapted on iterated
bundles \cite{CoDi11, CoDi13, GaHoMeRaVi10, MaRaRa91}. Additional terms in these equations may
or may not appear depending on whether trivialization preserves semidirect product and
group structures or not. If one preserves the group structures, canonical
embeddings of factors involving trivialization defines subgroups of iterated
bundles and reduction of dynamics with these subgroups becomes possible.

Based on exhaustive investigation of trivializations in our previous work
\cite{EsGu14a}, we shall present all reductions of dynamics on iterated
bundles of a Lie group with the convenient trivialization of the first kind.
In trivialization of the first kind, we identify tangent $TG$ and cotangent $%
T^{\ast }G$ bundles with their semidirect product trivializations $%
G\circledS \mathfrak{g}$ and $G\circledS \mathfrak{g}^{\ast }$,
respectively. Then, we trivialize the iterated bundles $T\left( G\circledS \mathfrak{g%
}\right) ,$ $T\left( G\circledS \mathfrak{g}^{\ast }\right) ,$ $T^{\ast
}\left( G\circledS \mathfrak{g}\right) $ and $T^{\ast }\left( G\circledS
\mathfrak{g}^{\ast }\right) $ by considering them as tangent and cotangent
groups of semidirect products and express them as semidirect products of
base group with its Lie algebra and dual of Lie algebra, respectively. As an
example, we obtain\
\begin{equation}
\text{ }^{1}TT^{\ast }G\simeq T\left( G\circledS \mathfrak{g}^{\ast }\right)
\simeq \left( G\circledS \mathfrak{g}^{\ast }\right) \circledS Lie\left(
G\circledS \mathfrak{g}^{\ast }\right) \simeq \left( G\circledS \mathfrak{g}%
^{\ast }\right) \circledS \left( \mathfrak{g}\circledS \mathfrak{g}^{\ast
}\right)  \label{1st}
\end{equation}%
for which, the trivialization maps preserve lifted group structures thereby
making possible various reductions of dynamics. On the other hand, in
trivialization of the second kind, one distributes functors $T$ and $T^{\ast
}$ to $G\circledS \mathfrak{g}$ and $G\circledS \mathfrak{g}^{\ast }$,
obtains products of first order bundles and then, trivializes each factor
involving the products. This results in, for example,
\begin{equation}
^{2}TT^{\ast }G\simeq T\left( G\circledS \mathfrak{g}^{\ast }\right)
\rightarrow TG\circledS T\mathfrak{g}^{\ast }\simeq \left( G\circledS
\mathfrak{g}\right) \circledS \left( \mathfrak{g}^{\ast }\times \mathfrak{g}%
^{\ast }\right)  \label{2nd}
\end{equation}%
for which distributions of functors mix up orders of fibrations, and do not
preserve group structures \cite{EsGu14a}.

\subsection{Content of the work}

In this work, Hamiltonian and Lagrangian dynamics on iterated bundles and their reductions will
be studied under trivializations of the first kind. We shall present
Hamiltonian dynamics on trivialized Tulczyjew's symplectic space $%
^{1}TT^{\ast }G$, trivialized cotangent spaces $^{1}T^{\ast }TG$ and $%
^{1}T^{\ast }T^{\ast }G$ and perform all possible Poisson and symplectic
reductions. Lagrangian dynamics on $^{1}TTG$ and $^{1}TT^{\ast }G$ will be
presented with all possible Lagrangian reductions including reduction to
second order dynamics on $T^{2}G$.

Next section will start with trivialized dynamics on the first order bundles
$G\circledS \mathfrak{g}$ and $G\circledS \mathfrak{g}^{\ast }$. Trivialized
Euler-Lagrange equations for a Lagrangian density $\bar{L}=\bar{L}\left(
g,\xi \right) $ on $G\circledS \mathfrak{g}$ will be derived. When the
Lagrangian $\bar{L}$ does not depend on the base variable $g,$ that is $\bar{%
L}=l\left( \xi \right) $, trivialized Euler-Lagrange equations reduce to
Euler-Poincaré equations on the Lie algebra $\mathfrak{g}$. A Hamiltonian
function(al) $\bar{H}$ on $G\circledS \mathfrak{g}^{\ast }$ will determine
trivialized Hamilton's equations. When the Hamiltonian $\bar{H}$ depends
only on fiber variables, that is $\bar{H}=h\left( \mu \right) $, trivialized
Hamilton's equations reduce to a set of equations equivalent to Lie-Poisson
equations on the dual space $\mathfrak{g}^{\ast }$.

In the third section, we shall derive trivialized Euler-Lagrange equations
on $\ ^{1}TTG$ and Euler-Poincaré equations on $\mathfrak{g}\circledS
\mathfrak{g}$. The immersion $T^{2}G\rightarrow TTG$ will lead us to define
trivialized second order Euler-Lagrange equations on $G\circledS \left(
\mathfrak{g}\times \mathfrak{g}\right) $ and second order Euler-Poincaré
equations on $2\mathfrak{g}$.

In section four, Hamiltonian dynamics on $\ ^{1}T^{\ast }TG$ will be
studied. The group multiplication, symplectic two-form, and Hamilton's
equations on $\ ^{1}T^{\ast }TG$ will be written. The Poisson reduced space $%
\mathfrak{g}\circledS \left( \mathfrak{g^{\ast }}\times \mathfrak{g}^{\ast
}\right) $ by action of $G$ on $^{1}T^{\ast }TG$ and dynamics on its
symplectic leaves $\mathcal{O}_{\lambda }\times \mathfrak{g\times g}^{\ast }$
will be presented. Lie-Poisson structure, which is different from product
Poisson structure on $\mathfrak{g}^{\ast }\times \mathfrak{g}^{\ast }$ and
its symplectic leaves resulted from via symplectic reduction by action of $%
G\circledS \mathfrak{g}$ on $^{1}T^{\ast }TG$ will be written.

Section five describes dynamics on $^{1}T^{\ast }T^{\ast }G$ and its
reductions. Hamilton's equations on $\mathfrak{g}^{\ast }\circledS \left(
\mathfrak{g^{\ast }\times g}\right) $ and its symplectic leaves $\mathcal{O}%
_{\lambda }\times \mathfrak{g\times g}^{\ast }$ will be derived. Lie-Poisson
structures on $\mathfrak{g}\times \mathfrak{g}^{\ast }$, and its symplectic
leaves will be obtained via symplectic reduction by the action of $%
G\circledS \mathfrak{g}^{\ast }$ on $^{1}T^{\ast }T^{\ast }G$.

In the sixth section, we shall study dynamics on tangent bundle $%
^{1}TT^{\ast }G$ endowed with Tulczyjew's symplectic structure. $%
^{1}TT^{\ast }G$ carries an exact symplectic two-form $\Omega _{\
^{1}TT^{\ast }G}$ with two potential one-forms $\theta _{1}$ and $\theta
_{2} $ which enable us to describe Hamiltonian dynamics on $^{1}TT^{\ast }G$%
. We shall define two embeddings of $G\circledS \mathfrak{g}^{\ast }$ into $%
^{1}TT^{\ast }G$, one of which is symplectic and the other is Lagrangian.
This will lead us to present embedding of Lie-Poisson dynamics on $\mathfrak{%
g}^{\ast }\times \mathfrak{g}$ into Hamiltonian dynamics on $\ ^{1}TT^{\ast
}G$. As it is a tangent bundle, we shall derive trivialized Euler-Lagrange
equations on $^{1}TT^{\ast }G$. Euler-Poincaré equations on the Lie algebra $%
\mathfrak{g}\times \mathfrak{g}^{\ast }$ of the cotangent group $G\circledS
\mathfrak{g}^{\ast }$ will be computed and embedding of this dynamics into
trivialized Euler-Lagrange equations on $\ ^{1}TT^{\ast }G$ will also be
established.

\subsection{Notations}

$G$ is a Lie group and, $\mathfrak{g}=Lie\left( G\right) \simeq T_{e}G$ is
its Lie algebra. The dual of $\mathfrak{g}$ is $\mathfrak{g}^{\ast
}=Lie^{\ast }\left( G\right) $. Throughout the work, we shall adapt the
letters%
\begin{equation}
g,h\in G,\text{ \ \ }\xi ,\eta ,\zeta \in \mathfrak{g},\text{ \ \ }\mu ,\nu
,\lambda \in \mathfrak{g}^{\ast }  \label{G}
\end{equation}%
as elements of the spaces shown. For a tensor field which is either right or
left invariant, we shall use $V_{g}\in T_{g}G$, $\alpha _{g}\in T_{g}^{\ast
}G$, etc... We shall denote left and right multiplications on $G$ by $L_{g}$
and $R_{g}$, respectively. The right inner automorphism
\begin{equation}
I_{g}=L_{g^{-1}}\circ R_{g}  \label{InnerR}
\end{equation}%
is a right representation of $G$ on $G$ satisfying $I_{g}\circ I_{h}=I_{hg}.$
The right adjoint action $Ad_{g}=T_{e}I_{g}$ of $G$ on $\mathfrak{g}$ is
defined as the tangent map of $I_{g}$ at the identity $e\in G$. The
infinitesimal right adjoint representation $ad_{\xi }\eta $ is $\left[ \xi
,\eta \right] $ and is defined as derivative of $Ad_{g}$ over the identity.
A right invariant vector field $X_{\xi }^{G}$ generated by $\xi \in
\mathfrak{g}$ is
\begin{equation}
X_{\xi }^{G}\left( g\right) =T_{e}R_{g}\xi .  \label{riG}
\end{equation}%
The identity
\begin{equation}
\left[ \xi ,\eta \right] =\left[ X_{\xi }^{G},X_{\eta }^{G}\right] _{JL}
\label{rivf}
\end{equation}%
defines the isomorphism between $\mathfrak{g}$ and the space $\mathfrak{X}%
^{R}\left( G\right) $ of right invariant vector fields endowed with the
Jacobi-Lie bracket. The coadjoint action $Ad_{g}^{\ast }$ of $G$ on the dual
$\mathfrak{g}^{\ast }$ of the Lie algebra $\mathfrak{g}$ is a right
representation and is the linear algebraic dual of $Ad_{g^{-1}}$, namely,
\begin{equation}
\left\langle Ad_{g}^{\ast }\mu ,\xi \right\rangle =\left\langle \mu
,Ad_{g^{-1}}\xi \right\rangle  \label{dist*}
\end{equation}%
holds for all $\xi \in \mathfrak{g}$ and $\mu \in \mathfrak{g}^{\ast }$. The
infinitesimal coadjoint action $ad_{\xi }^{\ast }$ of $\mathfrak{g}$ on $%
\mathfrak{g}^{\ast }$ is the linear algebraic dual of $ad_{\xi }$. Note
that, the infinitesimal generator of the coadjoint action $Ad_{g}^{\ast }$
is minus the infinitesimal coadjoint action $ad_{\xi }^{\ast }$, that is, if
$g^{t}\subset G$ is a curve passing through the identity in the direction of
$\xi \in \mathfrak{g},$ then
\begin{equation}
\left. \frac{d}{dt}\right\vert _{t=0}Ad_{g^{t}}^{\ast }\mu =-ad_{\xi }^{\ast
}\mu .  \label{Adtoad}
\end{equation}

In the diagrams of this work, EL and EP will abbriviate Euler-Lagrange and
Euler-Poincaré equations, respectively, whereas P.R., S.R., L.R. and EP.R.
will denote Poisson, symplectic, Lagrangian and Euler-Poincaré reductions.
\newpage
\section{Dynamics on the First Order Bundles}

\subsection{Lagrangian dynamics on tangent group}

The group structure on a Lie group $G$ can be lifted to its tangent bundle $%
TG$. The right trivialization
\begin{equation}
tr_{TG}^{R}:TG\rightarrow G\circledS \mathfrak{g}:V_{g}\rightarrow \left(
g,T_{g}R_{g^{-1}}V_{g}\right) ,  \label{trTG}
\end{equation}%
is both a diffeomorphism and a group isomorphism from $TG$ to the semidirect
product group $G\circledS \mathfrak{g}$ with multiplication%
\begin{equation}
\left( g,\xi \right) \left( h,\eta \right) =\left( gh,\xi +Ad_{g^{-1}}\eta
\right) ,  \label{tgtri}
\end{equation}%
where $\left( g,\xi \right) ,\left( h,\eta \right) \in G\circledS \mathfrak{g%
}$ \cite{EsGu14a, Hi06, KoMiSl93, MaRaRa91, Mi08, Ra80}. For a Lagrangian
density $L$ on $TG,$ there is a unique function $\bar{L}$ on $G\circledS
\mathfrak{g}$ defined by the identity $\bar{L}\circ tr_{TG}^{R}=L.$ We will
compute the variation of the action integral
\begin{equation}
\delta \int_{a}^{b}\bar{L}\left( g,\xi \right) dt=\int_{a}^{b}\left(
\left\langle \frac{\delta \bar{L}}{\delta g},\delta g\right\rangle
_{g}+\left\langle \frac{\delta \bar{L}}{\delta \xi },\delta \xi
\right\rangle _{e}\right) dt  \label{act}
\end{equation}%
by applying Hamilton's principle to the variations of the base (group)
component and the reduced variational principle
\begin{equation}
\delta \xi =\dot{\eta}+\left[ \xi ,\eta \right]  \label{rvp}
\end{equation}%
to variations of the fiber (Lie algebra) component of $\bar{L}$. For the
reduced variational principle we refer to \cite{CeMaPeRa02, EsGu14b, MaRa99}
and for the Lagrangian dynamics on semidirect products to \cite{CeHoMaRa98,
CeMaRa01, HoMaRa97, MaWeRa84}. For the following result see also \cite{BoMa09, CoDi11, CoDi13, En00, EsGu14b}.

\begin{proposition}
A Lagrangian density $\bar{L}$ on the product bundle\ $G\circledS \mathfrak{g%
}$ for the extremal values of action integral in Eq.(\ref{act}) defines the
trivialized Euler-Lagrange dynamics
\begin{equation}
\frac{d}{dt}\frac{\delta \bar{L}}{\delta \xi }=T_{e}^{\ast }R_{g}\frac{%
\delta \bar{L}}{\delta g}+ad_{\xi }^{\ast }\frac{\delta \bar{L}}{\delta \xi }%
.  \label{preeulerlagrange}
\end{equation}
\end{proposition}

\begin{proof}
Using the reduced variational principle in Eq.(\ref{rvp}) we compute
\begin{eqnarray*}
&&\delta \int_{b}^{a}\bar{L}\left( g,\xi \right) dt=\int_{b}^{a}\left(
\left\langle \frac{\delta \bar{L}}{\delta g},\delta g\right\rangle
_{g}+\left\langle \frac{\delta \bar{L}}{\delta \xi },\delta \xi
\right\rangle _{e}\right) dt \\
&=&\int_{b}^{a}\left( \left\langle \frac{\delta \bar{L}}{\delta g},\delta
g\right\rangle _{g}+\left\langle \frac{\delta \bar{L}}{\delta \xi },\dot{\eta%
}+\left[ \xi ,\eta \right] _{\mathfrak{g}}\right\rangle _{e}\right) dt \\
&=&-\left. \left\langle \frac{\delta \bar{L}}{\delta \xi },\eta
\right\rangle _{e}\right\vert _{b}^{a}+\int_{b}^{a}\left( \left\langle \frac{%
\delta \bar{L}}{\delta g},\delta g\right\rangle _{g}+\left\langle -\frac{d}{%
dt}\frac{\delta \bar{L}}{\delta \xi }+ad_{\xi }^{\ast }\frac{\delta \bar{L}}{%
\delta \xi },\eta \right\rangle _{e}\right) dt \\
&=&-\left. \left\langle \frac{\delta \bar{L}}{\delta \xi }%
,T_{g}R_{g^{-1}}\delta g\right\rangle _{e}\right\vert
_{b}^{a}+\int_{b}^{a}\left\langle \frac{\delta \bar{L}}{\delta g},\delta
g\right\rangle _{g}+\left\langle ad_{\xi }^{\ast }\frac{\delta \bar{L}}{%
\delta \xi }-\frac{d}{dt}\frac{\delta \bar{L}}{\delta \xi }%
,T_{g}R_{g^{-1}}\delta g\right\rangle _{e}dt \\
&=&-\left. \left\langle T_{g}^{\ast }R_{g^{-1}}\frac{\delta \bar{L}}{\delta
\xi },\delta g\right\rangle _{g}\right\vert
_{a}^{b}+\int_{b}^{a}\left\langle \frac{\delta \bar{L}}{\delta g}%
+T_{g}^{\ast }R_{g^{-1}}\left( ad_{\xi }^{\ast }\frac{\delta \bar{L}}{\delta
\xi }-\frac{d}{dt}\frac{\delta \bar{L}}{\delta \xi }\right) ,\delta
g\right\rangle _{g}dt
\end{eqnarray*}%
and the conclusion follows from requirement that arbitrary variation $\delta
g$ vanishes on the boundaries.
\end{proof}

When $\bar{L}\left( g,\xi \right) =l\left( \xi \right) $ is independent of
the group variable $g,$ that is, when the Lagrangian is right invariant $L=%
\bar{L}\circ tr_{TG}^{R}$, the trivialized Euler-Lagrange equations (\ref%
{preeulerlagrange}) reduce to the Euler-Poincaré equations
\begin{equation}
\frac{d}{dt}\frac{\delta \bar{L}}{\delta \xi }=ad_{\xi }^{\ast }\frac{\delta
\bar{L}}{\delta \xi }  \label{EPEq}
\end{equation}%
on $\mathfrak{g}$.

\subsection{Hamiltonian dynamics on cotangent group}

The group structure on $G$ can be lifted to its cotangent bundle $T^{\ast }G$
with the multiplication
\begin{equation}
\left( g,\mu \right) \left( h,\nu \right) =\left( gh,\mu +Ad_{g^{-1}}^{\ast
}\nu \right) ,  \label{rgc}
\end{equation}%
by requiring that the right trivialization%
\begin{equation}
tr_{T^{\ast }G}^{R}:T^{\ast }G\rightarrow G\circledS \mathfrak{g}^{\ast
}:\alpha _{g}\rightarrow \left( g,T_{e}^{\ast }R_{g}\alpha _{g}\right)
\label{trT*G}
\end{equation}%
is a group isomorphism \cite{AbMa78}. By pulling back the canonical one-form $\theta
_{T^{\ast }G}$ and the symplectic two-form $\Omega _{T^{\ast }G}$ on $%
T^{\ast }G$ with the trivialization map $tr_{T^{\ast }G}^{R}$, $G\circledS
\mathfrak{g}^{\ast }$ becomes a symplectic manifold carrying an exact
symplectic two-form $\Omega _{G\circledS \mathfrak{g}^{\ast }}=d\theta
_{G\circledS \mathfrak{g}^{\ast }}$. The values of the canonical one-form $%
\theta _{G\circledS \mathfrak{g}^{\ast }}$ and the symplectic two-form $%
\Omega _{G\circledS \mathfrak{g}^{\ast }}$ on right invariant vector fields
are%
\begin{eqnarray}
\left\langle \theta _{G\circledS \mathfrak{g}^{\ast }},X_{\left( \xi ,\nu
\right) }^{G\circledS \mathfrak{g}^{\ast }}\right\rangle \left( g,\mu
\right) &=&\left\langle \mu ,\xi \right\rangle  \label{OhmT*G} \\
\left\langle \Omega _{G\circledS \mathfrak{g}^{\ast }};\left( X_{\left( \xi
,\nu \right) }^{G\circledS \mathfrak{g}^{\ast }},X_{\left( \eta ,\lambda
\right) }^{G\circledS \mathfrak{g}^{\ast }}\right) \right\rangle \left(
g,\mu \right) &=&\left\langle \nu ,\eta \right\rangle -\left\langle \lambda
,\xi \right\rangle +\left\langle \mu ,\left[ \xi ,\eta \right] \right\rangle
,  \label{Ohm2T*G}
\end{eqnarray}%
where a right invariant vector field generated by the Lie algebra element $%
\left( \xi ,\nu \right) \in Lie\left( G\circledS \mathfrak{g}^{\ast }\right)
\simeq \mathfrak{g}\circledS \mathfrak{g}^{\ast }$ takes the value%
\begin{equation*}
X_{\left( \xi ,\nu \right) }^{G\circledS \mathfrak{g}^{\ast }}\left( g,\mu
\right) =\left( T_{e}R_{g}\xi ,\nu +ad_{\xi }^{R\ast }\mu \right)
\end{equation*}%
at the point $\left( g,\mu \right) $. Let $H$ be a Hamiltonian function on $%
T^{\ast }G$ and define $\bar{H}$ on $G\circledS \mathfrak{g}^{\ast }$ by
requiring $\bar{H}\circ tr_{T^{\ast }G}^{R}=H$. For $\bar{H}$, the
Hamilton's equations are given by%
\begin{equation*}
i_{X_{\bar{H}}^{G\circledS \mathfrak{g}^{\ast }}}\Omega _{G\circledS
\mathfrak{g}^{\ast }}=-d\bar{H},
\end{equation*}%
where the Hamiltonian vector field%
\begin{equation}
X_{\bar{H}}^{G\circledS \mathfrak{g}^{\ast }}\left( g,\mu \right) =\left(
T_{e}R_{g}\frac{\delta \bar{H}}{\delta \mu },ad_{\frac{\delta \bar{H}}{%
\delta \mu }}^{\ast }\mu -T_{e}R_{g}^{\ast }\frac{\delta \bar{H}}{\delta g}%
\right)  \label{HamVF}
\end{equation}%
is a right invariant vector field generated by the Lie algebra element
\begin{equation*}
(\delta \bar{H}/\delta \mu ,-T_{e}^{\ast }R_{g}\left( \delta \bar{H}/\delta
g\right) ) \in \mathfrak{g}\circledS \mathfrak{g}^{\ast }.
\end{equation*}
Some related
works on the Hamiltonian dynamics on semidirect products are \cite{CoDi11,
EsGu14b, HoMaRa86, MaMiOrPeRa07, MaRaWe84, MaRaWe84b, MaRaRa91, Ra80}.

\begin{proposition}
For a Hamiltonian function $\bar{H}$ on the symplectic manifold $
G\circledS\mathfrak{g}^{\ast }$, the trivialized Hamilton's equations are%
\begin{equation}
\frac{dg}{dt}=T_{e}R_{g}\left( \frac{\delta \bar{H}}{\delta \mu }\right) ,\
\text{\ }\ \frac{d\mu }{dt}=ad_{\frac{\delta \bar{H}}{\delta \mu }}^{\ast
}\mu -T_{e}^{\ast }R_{g}\frac{\delta \bar{H}}{\delta g}.  \label{ULP}
\end{equation}
\end{proposition}

We note that the second term on the right hand side of the second equation
in Eq.(\ref{ULP}) is due to the semidirect product structure on $G\circledS
\mathfrak{g}^{\ast }$. For two function(al)s $\bar{F}$ and $\bar{K}$ on $%
G\circledS \mathfrak{g}^{\ast }$, the canonical Poisson bracket on $%
G\circledS \mathfrak{g}^{\ast }$ can be expressed as
\begin{equation}
\left\{ \bar{F},\bar{K}\right\} _{G\circledS \mathfrak{g}^{\ast }}\left(
g,\mu \right) =\left\langle T_{e}^{\ast }R_{g}\frac{\delta \bar{K}}{\delta g}%
,\frac{\delta \bar{F}}{\delta \mu }\right\rangle -\left\langle T_{e}^{\ast
}R_{g}\frac{\delta \bar{F}}{\delta g},\frac{\delta \bar{K}}{\delta \mu }%
\right\rangle +\left\langle \mu ,\left[ \frac{\delta \bar{F}}{\delta \mu },%
\frac{\delta \bar{K}}{\delta \mu }\right] _{\mathfrak{g}}^{R}\right\rangle ,
\label{PoissonGg*}
\end{equation}%
and is non-degenerate.

\begin{remark}
The symplectic two-form $\Omega _{G\circledS \mathfrak{g}^{\ast }}$ does not
conserved under the group operation in Eqs.(\ref{rgc}), hence $G\circledS
\mathfrak{g}^{\ast }$ is not a symplectic Lie group as defined in \cite%
{LiMe88}.
\end{remark}

\subsubsection{Reduction by $G$}

The left action of $G$ on $G\circledS \mathfrak{g}^{\ast }$ is
\begin{equation}
G\times \left( G\circledS \mathfrak{g}^{\ast }\right) \rightarrow G\circledS
\mathfrak{g}^{\ast }:\left( h;\left( g,\mu \right) \right) \rightarrow
\left( hg,Ad_{h^{-1}}^{\ast }\mu \right)  \label{GonGxg*}
\end{equation}%
with the infinitesimal generator $X_{\left( \xi ,0\right) }^{G\circledS _{R}%
\mathfrak{g}^{\ast }}$. If $\bar{H},$ defined on $G\circledS \mathfrak{g}%
^{\ast }$, is independent of $g$, it is right invariant under $G$. In this
case, dropping the terms involving $\delta \bar{H}/\delta g$ in Poisson
bracket (\ref{PoissonGg*}) is the Poisson reduction $G\circledS \mathfrak{g}%
^{\ast }\rightarrow G\backslash \left( G\circledS \mathfrak{g}^{\ast
}\right) \simeq \mathfrak{g}^{\ast }$. When $\bar{F}$ and $\bar{K}$ are
independent of the group variable $g\in G$, that is, $\bar{F}=f\left( \mu
\right) $ and $\bar{K}=k\left( \mu \right) $, we have the Lie-Poisson bracket%
\begin{equation}
\left\{ f,k\right\} _{\mathfrak{g}^{\ast }}\left( \mu \right) =\left\langle
\mu ,\left[ \frac{\delta f}{\delta \mu },\frac{\delta k}{\delta \mu }\right]
_{\mathfrak{g}}\right\rangle  \label{LPbracket}
\end{equation}%
on the dual space $G\backslash \left( G\circledS \mathfrak{g}^{\ast }\right)
\simeq \mathfrak{g}^{\ast }$. This is a manifestation of the fact that the
projection $G\circledS \mathfrak{g}^{\ast }\rightarrow \mathfrak{g}^{\ast }$
is the momentum mapping for the cotangent lifted left action of $G$ on $%
G\circledS \mathfrak{g}^{\ast }$. The Lie-Poisson bracket given in Eq.(\ref%
{LPbracket}) can also be obtained by pulling back the non-degenerate Poisson
bracket in Eq.(\ref{PoissonGg*}) with the embedding $\mathfrak{g}^{\ast
}\rightarrow G\circledS \mathfrak{g}^{\ast }$. The dynamics on $\mathfrak{g}%
^{\ast }$ is driven by the Hamiltonian vector field $X_{h}^{\mathfrak{g}%
^{\ast }}$ satisfying
\begin{equation*}
\left\{ f,h\right\} _{\mathfrak{g}^{\ast }}=-\left\langle df,X_{h}^{%
\mathfrak{g}^{\ast }}\right\rangle
\end{equation*}%
for a Hamiltonian function(al) $h$ on $\mathfrak{g}^{\ast }$. Using Eq.(\ref%
{LPbracket}), this gives the explicit form of the Lie-Poisson equations
\begin{equation}
\dot{\mu}=ad_{\frac{\delta h}{\delta \mu }}^{\ast }\mu .  \label{LP}
\end{equation}

For the symplectic leaves of this Poisson structure \cite{We83}, we apply Marsden-Weinstein symplectic redcution theorem \cite{MaWe74} to $G\circledS\mathfrak{g}^{\ast }$ with the action of $G$.
The action in Eq.(\ref{GonGxg*}) is symplectic
inducing the momentum mapping%
\begin{equation}
\mathbf{J}_{G\circledS\mathfrak{g}^{\ast }}:G\circledS \mathfrak{g}%
^{\ast }\longrightarrow \mathfrak{g}^{\ast }:\left( g,\mu \right)
\rightarrow \mu .
\end{equation}%
The inverse image $\mathbf{J}_{G\circledS \mathfrak{g}^{\ast }}^{-1}\left(
\mu \right) \subset G\circledS \mathfrak{g}^{\ast }$ of a regular value $\mu
\in \mathfrak{g}^{\ast }$ consists of two-tuples $\left( g,\mu \right) $ for
$g\in G$ and fixed $\mu \in \mathfrak{g}^{\ast }$. Hence, we may identify $%
\mathbf{J}_{G\circledS \mathfrak{g}^{\ast }}^{-1}\left( \mu \right) $ with
the group $G$. If $G_{\mu }$ is the isotropy group of coadjoint action, then
the quotient space%
\begin{equation}
\left. \mathbf{J}_{G\circledS \mathfrak{g}^{\ast }}^{-1}\left( \mu \right)
\right/ G_{\mu }=\mathcal{O}_{\mu }  \label{coadorb}
\end{equation}%
is isomorphic to the coadjoint orbit through the point $\mu \in \mathfrak{g}%
^{\ast }$. We denote the reduced symplectic two-form on $\mathcal{O}_{\mu }$
by $\Omega _{G\circledS \mathfrak{g}^{\ast }}^{G\backslash }\left( \mu
\right) $ which is the Kostant-Kirillov-Souriou two-form \cite{MaWe83, mwrss83}. The value of $%
\Omega _{G\circledS _{R}\mathfrak{g}^{\ast }}^{G\backslash }\left( \mu
\right) $ on two vector fields $\xi _{\mathfrak{g}^{\ast }}\left( \mu
\right) ,\eta _{\mathfrak{g}^{\ast }}\left( \mu \right) \in T_{\mu }\mathcal{%
O}_{\mu }$\ is
\begin{equation}
\left\langle \Omega _{G\circledS \mathfrak{g}^{\ast }}^{G\backslash };\left(
\xi _{\mathfrak{g}^{\ast }},\eta _{\mathfrak{g}^{\ast }}\right)
\right\rangle \left( \mu \right) =-\left\langle \mu ,\left[ \xi ,\eta \right]
_{\mathfrak{g}}^{R}\right\rangle .  \label{KKS}
\end{equation}%
Note that a Hamiltonian vector field $X_{h}^{\mathfrak{g}^{\ast }}$ for a
reduced Hamiltonian $h$ on $\mathcal{O}_{\mu }$ is the one generating the
Lie-Poisson equations (\ref{LP}).

\subsubsection{Reduction by $G_{\protect\mu }$}

The isotropy subgroup $G_{\mu }$ acts on $G\circledS \mathfrak{g}^{\ast }$
as described by Eq.(\ref{GonGxg*}). Then, a Poisson and a symplectic
reductions of dynamics are possible. We shall study these reductions in
detail in section four. Referring to the section 4.6 and in particular, to
the diagram (\ref{HRBS}), we recapitulate some results. The Poisson reduction
of the symplectic manifold $G\circledS \mathfrak{g}^{\ast }$ under the
action of the isotropy group $G_{\mu }$ results in
\begin{equation*}
G_{\mu }\backslash \left( G\circledS \mathfrak{g}^{\ast }\right) \simeq
\mathcal{O}_{\mu }\times \mathfrak{g}^{\ast },
\end{equation*}%
where $\mathcal{O}_{\mu }$ is the coadjoint orbit through $\mu $ with
Poisson bracket
\begin{equation*}
\left\{ H,K\right\} _{\mathcal{O}_{\mu }\times \mathfrak{g}^{\ast }}\left(
\mu ,\nu \right) =\left\langle \mu ,\left[ \frac{\delta H}{\delta \mu },%
\frac{\delta K}{\delta \mu }\right] \right\rangle +\left\langle \nu ,\left[
\frac{\delta H}{\delta \mu },\frac{\delta K}{\delta \nu }\right] -\left[
\frac{\delta K}{\delta \mu },\frac{\delta H}{\delta \nu }\right]
\right\rangle .
\end{equation*}%
Note that it is not the direct product of Lie-Poisson structures on $%
\mathcal{O}_{\mu }$ and $\mathfrak{g}^{\ast }$.

The coadjoint action of $G\circledS \mathfrak{g}$ on the dual $\mathfrak{g}%
^{\ast }\times \mathfrak{g}^{\ast }$ of its Lie algebra is
\begin{equation}
Ad_{\left( g,\xi \right) }^{\ast }:\mathfrak{g}^{\ast }\times \mathfrak{g}%
^{\ast }\rightarrow \mathfrak{g}^{\ast }\times \mathfrak{g}^{\ast }:\left(
\mu ,\nu \right) \rightarrow \left( Ad_{g}^{\ast }\left( \mu -ad_{\xi
}^{\ast }\nu \right) ,Ad_{g}^{\ast }\nu \right) .  \label{coad}
\end{equation}%
The symplectic reduction of $G\circledS \mathfrak{g}^{\ast }$ under the
action of the isotropy subgroup $G_{\mu }$ results in the coadjoint orbit $%
\mathcal{O}_{\left( \mu ,\nu \right) }$ in $\mathfrak{g}^{\ast }\times
\mathfrak{g}^{\ast }$ through the point $\left( \mu ,\nu \right) $ under the
action in Eq.(\ref{coad}). The reduced symplectic two-form $\Omega _{%
\mathcal{O}_{\left( \mu ,\nu \right) }}$ takes the value%
\begin{equation}
\left\langle \Omega _{\mathcal{O}_{\left( \mu ,\nu \right) }};\left( \eta
,\zeta \right) ,\left( \bar{\eta},\bar{\zeta}\right) \right\rangle \left(
\mu ,\nu \right) =\left\langle \mu ,\left[ \bar{\eta},\eta \right]
\right\rangle +\left\langle \nu ,\left[ \bar{\eta},\zeta \right] -\left[
\eta ,\bar{\zeta}\right] \right\rangle
\end{equation}%
on two vectors $\left( \eta ,\zeta \right) $ and $\left( \bar{\eta},\bar{%
\zeta}\right) $ in $T_{\left( \mu ,\nu \right) }\mathcal{O}_{\left( \mu ,\nu
\right) }$.

We summarize reductions of the symplectic space $G\circledS \mathfrak{g}%
^{\ast }$ in the following diagram.%
\begin{equation}
\xymatrix{\mathfrak{g}^{\ast } \ar[dd]_{\txt{Poisson \\ embedding}} &&&&
\mathcal{O}_{\mu } \ar@{_{(}->}[llll]_{\txt{symplectic leaf}}
\ar[dd]^{\txt{symplectic \\ embedding}} \\ && G\circledS\mathfrak{g}^{\ast }
\ar[ull]|-{\text{P.R. by G}} \ar[urr]|-{\text{S.R. by G}}
\ar[dll]|-{\text{P.R. by } G_{\mu}} \ar[drr]|-{\text{S.R. by }G_{\mu}} \\
\mathcal{O}_{\mu }\circledS\mathfrak{g}^{\ast }&&&&\mathcal{O}_{(\mu,\nu)}
\ar@{^{(}->}[llll]^{\txt{symplectic leaf}}}  \label{DiagramGg*}
\end{equation}
\newpage
\section{Tangent Bundle of Tangent Group}

\subsection{Trivialization and group structure}

The iterated tangent bundle $TTG\simeq T\left( G\circledS \mathfrak{g}%
\right) \mathfrak{\ }$of the tangent group $TG\simeq G\circledS \mathfrak{g}$
can be globally trivialized as semidirect product of $G$ and three copies of
its Lie algebra $\mathfrak{g}$
\begin{eqnarray}
tr_{T\left( G\circledS \mathfrak{g}\right) } &:&T\left( G\circledS \mathfrak{%
g}_{1}\right) \rightarrow \left( G\circledS \mathfrak{g}_{1}\right)
\circledS \left( \mathfrak{g}_{2}\circledS \mathfrak{g}_{3}\right) =\text{ }%
^{1}TTG  \notag \\
&:&\left( V_{g},V_{\xi _{1}}\right) \rightarrow \left( g,\xi
_{1},TR_{g^{-1}}V_{g},V_{\xi _{1}}-\left[ \xi ,TR_{g^{-1}}V_{g}\right]
\right)  \label{trTTG}
\end{eqnarray}%
for $\left( V_{g},V_{\xi _{1}}\right) \in T_{\left( g,\xi _{1}\right)
}\left( G\circledS \mathfrak{g}_{1}\right) $. The trivialization is
performed by considering the tangent bundle $T\left( G\circledS \mathfrak{g}%
_{1}\right) $ as semidirect product of the group $G\circledS \mathfrak{g}%
_{1} $ and its Lie algebra $\mathfrak{g}_{2}\circledS \mathfrak{g}_{3}$.
Being a tangent bundle, $^{1}TTG$ is a Lie group with multiplication
\begin{eqnarray}
&&\left( g,\xi _{1},\xi _{2},\xi _{3}\right) \left( h,\eta _{1},\eta
_{2},\eta _{3}\right)  \notag \\
&=&\left( gh,\xi _{1}+Ad_{g^{-1}}\eta _{1},\xi _{2}+Ad_{g^{-1}}\eta _{2},\xi
_{3}+Ad_{g^{-1}}\eta _{3}+[Ad_{g^{-1}}\eta _{2},\xi _{1}]\right) .
\label{GrTTG}
\end{eqnarray}%
See \cite{Vi13} on group structure for jet bundles over Lie groups.

\begin{proposition}
The canonical immersions of the following submanifolds
\begin{eqnarray}
&&G,\text{ }\mathfrak{g}_{1},\text{ }\mathfrak{g}_{2}\text{, }\mathfrak{g}%
_{3}\text{, }G\circledS \mathfrak{g}_{1},\text{ }G\circledS \mathfrak{g}_{2},%
\text{ }G\circledS \mathfrak{g}_{3},\text{ }\mathfrak{g}_{1}\times \mathfrak{%
g}_{2},\text{ }\mathfrak{g}_{1}\times \mathfrak{g}_{3}\text{, }\mathfrak{g}%
_{1}\times \mathfrak{g}_{4},\text{ }  \label{SbGrTTG} \\
&&G\circledS \left( \mathfrak{g}_{1}\times \mathfrak{g}_{2}\right) ,\text{ }%
G\circledS \left( \mathfrak{g}_{1}\times \mathfrak{g}_{3}\right) \text{, }%
G\circledS \left( \mathfrak{g}_{2}\times \mathfrak{g}_{3}\right) \text{, }%
\left( \mathfrak{g}_{1}\times \mathfrak{g}_{2}\right) \circledS \mathfrak{g}%
_{3}  \notag
\end{eqnarray}%
define subgroups of $^{1}TTG$ and hence, they act on $^{1}TTG$ by actions
induced by the multiplication in Eq.(\ref{GrTTG}).
\end{proposition}

Here, the group multiplications on the vector spaces listed in (\ref{SbGrTTG}%
) are vector additions. The group on the semidirect products $G\circledS
\mathfrak{g}$ is the one given in Eq.(\ref{tgtri}). The group structures on $%
G\circledS \left( \mathfrak{g\times g}\right) $ are in form%
\begin{equation}
\left( g,\xi ,\bar{\xi}\right) \left( h,\eta ,\bar{\eta}\right) =\left(
gh,\xi _{1}+Ad_{g^{-1}}\eta _{1},\bar{\xi}+Ad_{g^{-1}}\bar{\eta}\right) ,
\end{equation}%
and the group structure on $\left( \mathfrak{g}_{1}\times \mathfrak{g}%
_{2}\right) \circledS \mathfrak{g}_{3}$ is
\begin{equation}
\left( \xi _{1},\xi _{2},\xi _{3}\right) \left( \eta _{1},\eta _{2},\eta
_{3}\right) =\left( \xi _{1}+\eta _{1},\xi _{2}+\eta _{2},\xi _{3}+\eta
_{3}+[\eta _{2},\xi _{1}]\right) .
\end{equation}

\subsection{Lagrangian Dynamics}

We consider a functional $L=L\left( g,\xi _{1},\xi _{2},\xi _{3}\right) $ on
$\ ^{1}TTG$ and the variation%
\begin{eqnarray}
\delta \int Ldt &=&\int \left\langle \frac{\delta L}{\delta \left( g,\xi
_{1},\xi _{2},\xi _{3}\right) },\left( g,\xi _{1},\xi _{2},\xi _{3}\right)
\right\rangle dt  \notag \\
&=&\int \left\langle \frac{\delta L}{\delta g},\delta g\right\rangle
+\left\langle \frac{\delta L}{\delta \xi _{1}},\delta \xi _{1}\right\rangle
+\left\langle \frac{\delta L}{\delta \xi _{2}},\delta \xi _{2}\right\rangle
+\left\langle \frac{\delta L}{\delta \xi _{3}},\delta \xi _{3}\right\rangle
dt  \label{aTTG}
\end{eqnarray}%
of the associated action integral to obtain trivialized Euler-Lagrange
equations on $\ ^{1}TTG$. To formulate variational principle on $\ ^{1}TTG$
we proceed as follows. For the variation $\left( \delta g,\delta \xi
_{1}\right) $ in the first and second integrals in Eq.(\ref{aTTG}), we pull $%
\delta \left( g,\xi _{1}\right) \simeq \left( \delta g,\delta \xi
_{1}\right) $ back from right to the identity $\left( e,0\right) $ by
\begin{equation*}
\left( \zeta ,\zeta _{1}\right) =T_{\left( e,0\right) }R_{\left( g,\xi
\right) ^{-1}}\delta \left( g,\xi _{1}\right) =\left( TR_{g^{-1}}\delta
g,\delta \xi _{1}-\left[ \xi _{1},TR_{g^{-1}}\delta g\right] _{\mathfrak{g}%
}\right) ,
\end{equation*}%
where $\left( \zeta ,\zeta _{1}\right) \in \mathfrak{g}\circledS \mathfrak{g}
$, and obtain
\begin{equation*}
\delta g=TR_{g}\zeta ,\text{ \ \ }\delta \xi _{1}=\zeta _{1}+\left[ \xi
_{1},\zeta \right] _{\mathfrak{g}},
\end{equation*}%
for arbitrary choices of $\left( \zeta ,\zeta _{1}\right) $. To obtain the
variation $\left( \delta \xi _{2},\delta \xi _{3}\right) =\delta \left( \xi
_{2},\xi _{3}\right) $ in the third and fourth integral in Eq.(\ref{aTTG})
we consider the reduced variational principle
\begin{equation}
\delta \left( \xi _{2},\xi _{3}\right) =\frac{d}{dt}\left( \eta _{2},\eta
_{3}\right) +\left[ \left( \xi _{2},\xi _{3}\right) ,\left( \eta _{2},\eta
_{3}\right) \right] _{\mathfrak{g}\circledS \mathfrak{g}},  \label{rvpgxg}
\end{equation}%
on the Lie algebra $\mathfrak{g}\circledS \mathfrak{g}$ of $G\circledS
\mathfrak{g}$ for arbitrary choice of $\left( \eta _{2},\eta _{3}\right) $
in $\mathfrak{g}\circledS \mathfrak{g}$. Here, the Lie algebra bracket on $%
\mathfrak{g}\circledS \mathfrak{g}$ is
\begin{equation}
\lbrack \left( \xi _{2},\xi _{3}\right) ,\left( \eta _{2},\eta _{3}\right)
]_{\mathfrak{g}\circledS \mathfrak{g}}=\left( ad_{\xi _{2}}\eta _{2},ad_{\xi
_{2}}\eta _{3}-ad_{\eta _{2}}\xi _{3}\right) .  \label{LABgg}
\end{equation}%
Note that, the variational principle in Eq.(\ref{rvpgxg}) is the same with
the one in Eq.(\ref{rvp}) but this time Lie algebra elements are two-tuples
and the Lie algebra bracket is (\ref{LABgg}). So, we have
\begin{equation}
\delta \xi _{2}=\dot{\eta}_{2}+\left[ \xi _{2},\eta _{2}\right] _{\mathfrak{g%
}}\text{, \ \ }\delta \xi _{3}=\dot{\eta}_{3}+\left[ \xi _{2},\eta _{3}%
\right] _{\mathfrak{g}}-\left[ \eta _{2},\xi _{3}\right] _{\mathfrak{g}}.
\label{hovp}
\end{equation}%
Now, we are ready to prove the following proposition by direct calculations.

\begin{proposition}
A Lagrangian $L=L\left( g,\xi _{1},\xi _{2},\xi _{3}\right) $ on $^{1}TTG$
generates the trivialized Euler-Lagrange equations
\begin{eqnarray}
\frac{d}{dt}\frac{\delta L}{\delta \xi _{2}} &=&T^{\ast }R_{g}\frac{\delta L%
}{\delta g}+ad_{\xi _{1}}^{\ast }\frac{\delta L}{\delta \xi _{1}}+ad_{\xi
_{2}}^{\ast }\frac{\delta L}{\delta \xi _{2}}+ad_{\xi _{3}}^{\ast }\frac{%
\delta L}{\delta \xi _{3}}  \label{UnEPTTG1} \\
\frac{d}{dt}\frac{\delta L}{\delta \xi _{3}} &=&\frac{\delta L}{\delta \xi
_{1}}+ad_{\xi _{2}}^{\ast }\frac{\delta L}{\delta \xi _{3}}.  \label{UnEPTTG}
\end{eqnarray}
\end{proposition}

It is possible to write the trivialized Euler-Lagrange equations (\ref%
{UnEPTTG1}) and (\ref{UnEPTTG}) as a single equation by substituting the
second equation to the first one. This gives%
\begin{equation}
\frac{d}{dt}\frac{\delta L}{\delta \xi _{2}}=T^{\ast }R_{g}\frac{\delta L}{%
\delta g}+ad_{\xi _{1}}^{\ast }\left( \frac{d}{dt}\frac{\delta L}{\delta \xi
_{3}}-ad_{\xi _{2}}^{\ast }\frac{\delta L}{\delta \xi _{3}}\right) +ad_{\xi
_{2}}^{\ast }\frac{\delta L}{\delta \xi _{2}}+ad_{\xi _{3}}^{\ast }\frac{%
\delta L}{\delta \xi _{3}}.  \label{UnEPTTG2}
\end{equation}

\subsection{Reductions}

We shall derive equations governing dynamics reduced by subgroups given in
Eq.(\ref{SbGrTTG}). When the Lagrangian is independent of the group variable
$L=L\left( \xi _{1},\xi _{2},\xi _{3}\right) $, that is, right invariant
under the action of $G$, we arrive at the equation
\begin{equation}
\frac{d}{dt}\frac{\delta L}{\delta \xi _{2}}=ad_{\xi _{1}}^{\ast }\left(
\frac{d}{dt}\frac{\delta L}{\delta \xi _{3}}-ad_{\xi _{2}}^{\ast }\frac{%
\delta L}{\delta \xi _{3}}\right) +ad_{\xi _{2}}^{\ast }\frac{\delta L}{%
\delta \xi _{2}}+ad_{\xi _{3}}^{\ast }\frac{\delta L}{\delta \xi _{3}}
\label{Laggxgxg}
\end{equation}%
on $\mathfrak{g}_{1}\circledS \mathfrak{g}_{2}\circledS \mathfrak{g}_{3}$.
When the Lagrangian $L=L\left( g,\xi _{2},\xi _{3}\right) $ is independent
of the Lie algebra variable $\xi _{1}$, that is, $L$ is right invariant
under the action of $\mathfrak{g}_{1}$, we obtain the equations
\begin{equation}
\frac{d}{dt}\frac{\delta L}{\delta \xi _{2}}=T^{\ast }R_{g}\frac{\delta L}{%
\delta g}+ad_{\xi _{2}}^{\ast }\frac{\delta L}{\delta \xi _{2}}+ad_{\xi
_{3}}^{\ast }\frac{\delta L}{\delta \xi _{3}},\text{ \ \ }\frac{d}{dt}\frac{%
\delta L}{\delta \xi _{3}}=ad_{\xi _{2}}^{\ast }\frac{\delta L}{\delta \xi
_{3}}  \label{LagGxgxg}
\end{equation}%
on $G\circledS \mathfrak{g}_{2}\circledS \mathfrak{g}_{3}$. In case that $%
L=L\left( g,\xi _{2}\right) $, the trivialized Euler-Lagrange equations (\ref%
{UnEPTTG}) on $^{1}TTG$ reduces to the trivialized Euler-Lagrange equations (%
\ref{preeulerlagrange}) on $G\circledS \mathfrak{g}$. Further reduction of $%
G\circledS \mathfrak{g}_{2}\circledS \mathfrak{g}_{3}$ by $G$ reduces Eq.(%
\ref{LagGxgxg}) to Euler-Poincaré equations on $\mathfrak{g}\circledS
\mathfrak{g}$.

\begin{proposition}
The Euler-Poincaré equations on the Lie algebra $\mathfrak{g}\circledS
\mathfrak{g}$ of the group $G\circledS \mathfrak{g}$ are
\begin{equation}
\frac{d}{dt}\frac{\delta L}{\delta \xi _{2}}=ad_{\xi _{2}}^{\ast }\frac{%
\delta L}{\delta \xi _{2}}+ad_{\xi _{3}}^{\ast }\frac{\delta L}{\delta \xi
_{3}},\text{ \ \ }\frac{d}{dt}\frac{\delta L}{\delta \xi _{3}}=ad_{\xi
_{2}}^{\ast }\frac{\delta L}{\delta \xi _{3}},  \label{EPgg}
\end{equation}%
for $\left( \xi _{2},\xi _{3}\right) \in \mathfrak{g}\circledS \mathfrak{g}$.
\end{proposition}

There are several ways to arrive at this Euler-Poincaré equations. The first
is to consider $^{1}TTG$ as a semidirect product of the group $G\circledS
\mathfrak{g}_{1}$ and the Lie algebra $\mathfrak{g}_{2}\circledS \mathfrak{g}%
_{3}$, and use the reduced variational principle, as given in Eq.(\ref{hovp}%
), on $\mathfrak{g}\circledS \mathfrak{g}$ while extremizing the action
integral with density $L=L\left( \xi _{2},\xi _{3}\right) $. The second way
is to reduce the Euler-Lagrange dynamics on $G\circledS \mathfrak{g}%
_{2}\circledS \mathfrak{g}_{3}$ given by Eq.(\ref{LagGxgxg}) under the
action of $G.$ This is called reduction by stages in \cite{CeMaRa01,
HoMaRa97}. Note finally, that, the dependence $L=l\left( \xi _{2}\right) $
reduces all Lagrangian dynamics of this subsection to the Euler-Poincaré
equation (\ref{EPEq}) on $\mathfrak{g}$.

\subsection{Trivialized second order equations}

Under the right trivialization, the second order tangent bundle $T^{2}G$ can
be written as a semidirect product of $G$ and two copies of $\mathfrak{g}$,
we denote the right trivialization of $T^{2}G$ by $^{1}T^{2}G=G\circledS
\left( \mathfrak{g}\times \mathfrak{g}\right) $. For the trivialized total
spaces, we have the following canonical immersion%
\begin{equation}
^{1}T^{2}G\rightarrow G\circledS \mathfrak{g}_{2}\circledS \mathfrak{g}%
_{3}:\left( g,\xi ,\dot{\xi}\right) \rightarrow \left( g,\xi ,\dot{\xi}%
\right) .  \label{imm}
\end{equation}%
We refer to recent works \cite{AbCaCl13, CoDi11,
CoDi13, CoPr14, GaHoMeRaVi10} on the second order Lagrangian dynamics on Lie groups.

\begin{proposition}
On the image space of canonical immersion $^{1}T^{2}G\rightarrow G\circledS
\mathfrak{g}_{2}\circledS \mathfrak{g}_{3}$, the trivialized Euler-Lagrange
equations (\ref{LagGxgxg}) reduce to the trivialized second order
Euler-Lagrange equations%
\begin{equation}
\left( \frac{d}{dt}-ad_{\xi }^{\ast }\right) \left( \frac{\delta L}{\delta
\xi }-\frac{d}{dt}\left( \frac{\delta L}{\delta \dot{\xi}}\right) \right)
=T^{\ast }R_{g}\frac{\delta L}{\delta g}.  \label{soEL}
\end{equation}%
When the Lagrangian $L=L\left( \xi ,\dot{\xi}\right) $ does not depend on
the group variable, we arrive at the second order Euler-Poincaré equations%
\begin{equation}
\left( \frac{d}{dt}-ad_{\xi }^{\ast }\right) \left( \frac{\delta L}{\delta
\xi }-\frac{d}{dt}\left( \frac{\delta L}{\delta \dot{\xi}}\right) \right) =0,
\label{soep}
\end{equation}%
on the reduced space $\mathfrak{g\times g}$. Here, the Cartesian product $%
\mathfrak{g\times g}$ can be understood as a quotient space of the group $%
^{1}T^{2}G$ under the acion of $G$.
\end{proposition}

\begin{proof}
The choices $\xi _{1}=\xi _{2}=\xi $ and $\xi _{3}=\dot{\xi}$ reduce the
trivialized Euler-Lagrange equations (\ref{UnEPTTG}) to the set of equations
\begin{equation}
\frac{d}{dt}\frac{\delta L}{\delta \xi }=T^{\ast }R_{g}\frac{\delta L}{%
\delta g}+ad_{\xi }^{\ast }\frac{\delta L}{\delta \xi }+ad_{\dot{\xi}}^{\ast
}\frac{\delta L}{\delta \dot{\xi}}\text{, \ \ }\frac{d}{dt}\frac{\delta L}{%
\delta \dot{\xi}}=ad_{\xi }^{\ast }\frac{\delta L}{\delta \dot{\xi}}
\label{qwer}
\end{equation}%
on the image space of the immersion in $^{1}TTG$. To replace the last term
on the right hand side of the first equation we proceed to compute%
\begin{eqnarray}
ad_{\dot{\xi}}^{\ast }\frac{\delta L}{\delta \dot{\xi}} &=&\frac{d}{dt}%
\left( ad_{\xi }^{\ast }\frac{\delta L}{\delta \dot{\xi}}\right) -ad_{\xi
}^{\ast }\frac{d}{dt}\left( \frac{\delta L}{\delta \dot{\xi}}\right)  \notag
\\
&=&\frac{d^{2}}{dt^{2}}\left( \frac{\delta L}{\delta \dot{\xi}}\right)
-ad_{\xi }^{\ast }\frac{d}{dt}\left( \frac{\delta L}{\delta \dot{\xi}}\right)
\notag \\
&=&\left( \frac{d}{dt}-ad_{\xi }^{\ast }\right) \frac{d}{dt}\left( \frac{%
\delta L}{\delta \dot{\xi}}\right)  \label{idi}
\end{eqnarray}%
where we used the second equation. After arranging terms and substitution of
the identity in Eq.(\ref{idi}), the first equation in (\ref{qwer}) reads
\begin{eqnarray*}
\frac{d}{dt}\frac{\delta L}{\delta \xi }-ad_{\xi }^{\ast }\frac{\delta L}{%
\delta \xi } &=&T^{\ast }R_{g}\frac{\delta L}{\delta g}+\left( \frac{d}{dt}%
-ad_{\xi }^{\ast }\right) \frac{d}{dt}\left( \frac{\delta L}{\delta \dot{\xi}%
}\right) \\
\left( \frac{d}{dt}-ad_{\xi }^{\ast }\right) \frac{\delta L}{\delta \xi }
&=&T^{\ast }R_{g}\frac{\delta L}{\delta g}+\left( \frac{d}{dt}-ad_{\xi
}^{\ast }\right) \frac{d}{dt}\left( \frac{\delta L}{\delta \dot{\xi}}\right)
\\
\left( \frac{d}{dt}-ad_{\xi }^{\ast }\right) \left( \frac{\delta L}{\delta
\xi }-\frac{d}{dt}\left( \frac{\delta L}{\delta \dot{\xi}}\right) \right)
&=&T^{\ast }R_{g}\frac{\delta L}{\delta g}.
\end{eqnarray*}
\end{proof}
We summarize the trivializations and reductions of this section in
diagram (\ref{TTGd}) accompanied by Eqs.(\ref{TTG})-(\ref{g}).
\newpage

\begin{eqnarray}
\xymatrix{&{\begin{array}{c} \left( G\circledS \mathfrak{g}_{1}\right)\circledS \left( \mathfrak{g}_{2}\circledS \mathfrak{g}_{3}\right)\\\text{EL in (\ref{TTG})}\end{array}} \ar[dd]|-{\text{L.R. by
}\mathfrak{g}_{1}} \ar[ddr]|-{\text{L.R. by }G} \ar[ddddr]|-{\text{EP.R.
by }G\circledS\mathfrak{g}_{1}} \\\\ {\begin{array}{c} G\circledS
(\mathfrak{g}_{2}\times \overset{\cdot }{\mathfrak{g}}_{2}) \\ \text{2nd
order EL in (\ref{Ggggdot})}\end{array}} \ar@{_{(}->}[r]^{\txt{immersion \\
in Eq.(\ref{imm})}} \ar[dd]|-{\text{EP.R. by }G} &
{\begin{array}{c}G\circledS \left( \mathfrak{g}_{2}\times
\mathfrak{g}_{3}\right) \\ \text{EL in (\ref{Ggg})}\end{array}}
\ar[ddr]|-{\text{L.R. by }G} & {\begin{array}{c}(\mathfrak{g}_{1}\times
\mathfrak{g}_{2})\circledS \mathfrak{g}_{3} \\\text{EL in
(\ref{ggg})}\end{array}} \ar[dd]|-{\text{L.R. by }\mathfrak{g}_{1}} \\\\
{\begin{array}{c}\mathfrak{g}_{2}\circledS \overset{\cdot
}{\mathfrak{g}}_{2}\\\text{2nd order EP in (\ref{ggdot})}\end{array}} &
{\begin{array}{c}G\circledS \mathfrak{g}_{2} \\\text{EL in
(\ref{Gg})}\end{array}} \ar[dd]|-{\text{EP.R. by }G} \ar@{^{(}->}[uul]
|-{\txt{canonical \\ immersion}} \ar@{^{(}->}[uu]|-{\txt{canonical \\
immersion}} & {\begin{array}{c}\mathfrak{g}_{2}\circledS \mathfrak{g}_{3}
\\\text{EP in ( \ref{gg})}\end{array}} \\\\ &
{\begin{array}{c}\mathfrak{g}_{2} \\\text{EP in (\ref{g})}\end{array}}
\ar@{^{(}->}[uul]|-{\txt{canonical \\ immersion}}
\ar@{^{(}->}[uur]|-{\txt{canonical \\ immersion}} }  \label{TTGd} \\
------------------------  \notag \\
\frac{d}{dt}\frac{\delta L}{\delta \xi _{2}}=T^{\ast }R_{g}\frac{\delta L}{%
\delta g}+ad_{\xi _{1}}^{\ast }\left( \frac{d}{dt}\frac{\delta L}{\delta \xi
_{3}}-ad_{\xi _{2}}^{\ast }\frac{\delta L}{\delta \xi _{3}}\right) +ad_{\xi
_{2}}^{\ast }\frac{\delta L}{\delta \xi _{2}}+ad_{\xi _{3}}^{\ast }\frac{%
\delta L}{\delta \xi _{3}}  \label{TTG} \\
\frac{d}{dt}\frac{\delta L}{\delta \xi _{2}}=ad_{\xi _{1}}^{\ast }\left(
\frac{d}{dt}\frac{\delta L}{\delta \xi _{3}}-ad_{\xi _{2}}^{\ast }\frac{%
\delta L}{\delta \xi _{3}}\right) +ad_{\xi _{2}}^{\ast }\frac{\delta L}{%
\delta \xi _{2}}+ad_{\xi _{3}}^{\ast }\frac{\delta L}{\delta \xi _{3}}
\label{ggg} \\
\frac{d}{dt}\frac{\delta L}{\delta \xi _{2}}=T^{\ast }R_{g}\frac{\delta L}{%
\delta g}+ad_{\xi _{2}}^{\ast }\frac{\delta L}{\delta \xi _{2}}+ad_{\xi
_{3}}^{\ast }\frac{\delta L}{\delta \xi _{3}},\text{ \ \ }\frac{d}{dt}\frac{%
\delta L}{\delta \xi _{3}}=ad_{\xi _{2}}^{\ast }\frac{\delta L}{\delta \xi
_{3}}  \label{Ggg} \\
T^{\ast }R_{g}\frac{\delta L}{\delta g}=\left( \frac{d}{dt}-ad_{\xi }^{\ast
}\right) \left( \frac{\delta L}{\delta \xi }-\frac{d}{dt}\left( \frac{\delta
L}{\delta \dot{\xi}}\right) \right)  \label{Ggggdot} \\
0=\left( \frac{d}{dt}-ad_{\xi }^{\ast }\right) \left( \frac{\delta L}{\delta
\xi }-\frac{d}{dt}\left( \frac{\delta L}{\delta \dot{\xi}}\right) \right)
\label{ggdot} \\
\frac{d}{dt}\frac{\delta L}{\delta \xi _{2}}=ad_{\xi _{2}}^{\ast }\frac{%
\delta L}{\delta \xi _{2}}+ad_{\xi _{3}}^{\ast }\frac{\delta L}{\delta \xi
_{3}},\text{ \ \ \ }\frac{d}{dt}\frac{\delta L}{\delta \xi _{3}}=ad_{\xi
_{2}}^{\ast }\frac{\delta L}{\delta \xi _{3}}  \label{gg} \\
\frac{d}{dt}\frac{\delta \bar{L}}{\delta \xi _{2}}=T_{e}^{\ast }R_{g}\frac{%
\delta \bar{L}}{\delta g}+ad_{\xi }^{\ast }\frac{\delta \bar{L}}{\delta \xi
_{2}}  \label{Gg} \\
\frac{d}{dt}\frac{\delta \bar{L}}{\delta \xi }=ad_{\xi }^{\ast }\frac{\delta
\bar{L}}{\delta \xi }  \label{g}
\end{eqnarray}
\newpage

\section{Cotangent Bundle of Tangent Group}

\subsection{Trivialization}

The cotangent bundle $T^{\ast }TG$ of tangent group can be identified with
the cotangent bundle $T^{\ast }\left( G\circledS \mathfrak{g}\right) $ of
the semidirect product group $G\circledS \mathfrak{g}_{1}$ using the
trivialization in Eq.(\ref{trTG}). This way, the global trivialization of $%
T^{\ast }TG\simeq T^{\ast }\left( G\circledS \mathfrak{g}\right) $ can be
achieved by trivializing $T^{\ast }\left( G\circledS \mathfrak{g}\right) $
into the semidirect product group $G\circledS \mathfrak{g}_{1}$ and the dual
$\mathfrak{g}_{2}^{\ast }\times \mathfrak{g}_{3}^{\ast }$ of its Lie algebra
$\mathfrak{g}_{2}\circledS \mathfrak{g}_{3}$
\begin{eqnarray}
tr_{T^{\ast }\left( G\circledS \mathfrak{g}\right) }^{1} &:&T^{\ast }\left(
G\circledS \mathfrak{g}_{1}\right) \rightarrow \left( G\circledS \mathfrak{g}%
_{1}\right) \circledS \left( \mathfrak{g}_{2}^{\ast }\times \mathfrak{g}%
_{3}^{\ast }\right) =:\text{ }^{1}T^{\ast }TG  \notag \\
&:&\left( \alpha _{g},\alpha _{\xi }\right) \rightarrow \left( g,\xi
,T_{e}^{\ast }R_{g}\left( \alpha _{g}\right) +ad_{\xi }^{\ast }\alpha _{\xi
},\alpha _{\xi }\right) .  \label{trT*TG}
\end{eqnarray}%
This trivialization preserves the group multiplication
\begin{eqnarray}
&&\left( g,\xi ,\mu _{1},\mu _{2}\right) \left( h,\eta ,\nu _{1},\nu
_{2}\right)  \label{GrT*TG} \\
&=&\left( gh,\xi +Ad_{g^{-1}}\eta ,\mu _{1}+Ad_{g^{-1}}^{\ast }\left( \nu
_{1}+ad_{Ad_{g}\xi }^{\ast }\nu _{2}\right) ,\mu _{2}+Ad_{g^{-1}}^{\ast }\nu
_{2}\right)  \notag
\end{eqnarray}%
on $^{1}T^{\ast }TG$ from which we obtain the following subgroups.

\begin{proposition}
The canonical immersions of the following submanifolds
\begin{eqnarray}
G\text{, }\mathfrak{g}_{1}\text{, }\mathfrak{g}_{2}^{\ast }\text{, }%
\mathfrak{g}_{3}^{\ast }\text{, }G\circledS \mathfrak{g}_{1}\text{, }%
G\circledS \mathfrak{g}_{2}^{\ast }\text{, }G\circledS \mathfrak{g}%
_{3}^{\ast }\text{, }\mathfrak{g}_{2}^{\ast }\times \mathfrak{g}_{3}^{\ast }%
\text{, }  \notag \\
\left( \mathfrak{g}_{1}\times \mathfrak{g}_{3}^{\ast }\right) \circledS
\mathfrak{g}_{2}^{\ast },\text{ }G\circledS \left( \mathfrak{g}_{1}\times
\mathfrak{g}_{2}^{\ast }\right) \text{, }G\circledS \left( \mathfrak{g}%
_{2}^{\ast }\times \mathfrak{g}_{3}^{\ast }\right)  \label{immg}
\end{eqnarray}%
define subgroups of $^{1}T^{\ast }TG$ and hence they act on $^{1}T^{\ast }TG$
by actions induced from the multiplication in Eq.(\ref{GrT*TG}).
\end{proposition}

Here, the group structure on $G\circledS \mathfrak{g}_{1}$ is the one given
by Eq.(\ref{tgtri}), that of $G\circledS \mathfrak{g}_{2}^{\ast }$ and $%
G\circledS \mathfrak{g}_{3}^{\ast }$ are in Eq.(\ref{rgc}) and, we obtain
the multiplications%
\begin{eqnarray}
\left( g,\xi ,\mu \right) \left( h,\eta ,\nu \right) &=&( gh,\xi
+Ad_{g^{-1}}\eta ,\mu +Ad_{g^{-1}}^{\ast }\nu)  \label{GrP}
\\
\left( g,\mu _{1},\mu _{2}\right) \left( h,\nu _{1},\nu _{2}\right) &=&(
gh,\mu _{1}+Ad_{g^{-1}}^{\ast }\nu _{1},\mu _{2}+Ad_{g^{-1}}^{\ast }\nu
_{2})  \label{GrGg*g*}
\\
\left( \xi ,\mu _{1},\mu _{2}\right) \left( \eta ,\nu _{1},\nu _{2}\right)
&=&( \xi +\eta ,\mu _{1}+\nu _{1}+ad_{\xi }^{\ast }\nu _{2},\mu _{2}+\nu
_{2})  \label{Grgg*g*}
\end{eqnarray}
defining the group structures on $G\circledS \left( \mathfrak{g}_{1}\times
\mathfrak{g}_{2}^{\ast }\right) $, $G\circledS \left( \mathfrak{g}_{2}^{\ast
}\times \mathfrak{g}_{3}^{\ast }\right) $ and $\left( \mathfrak{g}_{1}\times
\mathfrak{g}_{3}^{\ast }\right) \circledS \mathfrak{g}_{2}^{\ast }$,
respectively.

\subsection{Symplectic Structure}

By requiring the trivialization $tr_{T^{\ast }\left( G\circledS \mathfrak{g}%
\right) }^{1}$ be a symplectic map, we define a canonical one-form $\theta
_{\ ^{1}T^{\ast }TG}$\ and a symplectic two-form $\Omega _{\ ^{1}T^{\ast
}TG} $ on the trivialized cotangent bundle$\ ^{1}T^{\ast }TG$. To this end,
we recall that a right invariant vector field $X_{\left( \eta ,\zeta
,\lambda _{1},\lambda _{2}\right) }^{\text{ }^{1}T^{\ast }TG}$ on $%
^{1}T^{\ast }TG$ is generated by an element $\left( \eta ,\zeta ,\lambda
_{1},\lambda _{2}\right) $ in the Lie algebra $\left( \mathfrak{g}\circledS
\mathfrak{g}\right) \circledS \left( \mathfrak{g}^{\ast }\times \mathfrak{g}%
^{\ast }\right) $ of $^{1}T^{\ast }TG$ by means of the tangent lift of right
translation on $^{1}T^{\ast }TG$. At a point $\left( g,\xi ,\mu ,\nu \right)
$ in $^{1}T^{\ast }TG$, the value of such a right invariant vector field
reads
\begin{equation}
X_{\left( \eta ,\zeta ,\lambda _{1},\lambda _{2}\right) }^{\text{ }%
^{1}T^{\ast }TG}\left( g,\xi ,\mu ,\nu \right) =\left( T_{e}R_{g}\eta ,\zeta
+[\xi ,\eta ],\lambda _{1}+ad_{\eta }^{\ast }\mu +ad_{\zeta }^{\ast }\nu
,\lambda _{2}+ad_{\eta }^{\ast }\nu \right)  \label{rivfT*TG}
\end{equation}%
which is an element of the fiber $T_{\left( g,\xi ,\mu ,\nu \right) }\left(
^{1}T^{\ast }TG\right) $. The values of canonical forms $\theta _{\
^{1}T^{\ast }TG}$ and $\Omega _{\ ^{1}T^{\ast }TG}$ on right invariant
vector fields can then be computed to be%
\begin{eqnarray}
\langle \theta _{\ ^{1}T^{\ast }TG};X_{\left( \eta ,\zeta ,\lambda
_{1},\lambda _{2}\right) }^{\ ^{1}T^{\ast }TG}\rangle &=&\left\langle \mu
,\eta \right\rangle +\left\langle \nu ,\zeta \right\rangle  \label{thet1T*TG}
\\
\left\langle \Omega _{\ ^{1}T^{\ast }TG};\left( X_{\left( \eta ,\zeta
,\lambda _{1},\lambda _{2}\right) }^{\ ^{1}T^{\ast }TG},X_{\left( \bar{\eta},%
\bar{\zeta},\bar{\lambda}_{1},\bar{\lambda}_{2}\right) }^{\ ^{1}T^{\ast
}TG}\right) \right\rangle &=&\left\langle \lambda _{1},\bar{\eta}%
\right\rangle +\left\langle \lambda _{2},\bar{\zeta}\right\rangle
-\left\langle \bar{\lambda}_{1},\eta \right\rangle -\left\langle \bar{\lambda%
}_{2},\zeta \right\rangle  \notag \\
&&+\left\langle \mu ,\left[ \eta ,\bar{\eta}\right] \right\rangle
+\left\langle \nu ,\left[ \eta ,\bar{\zeta}\right] -\left[ \bar{\eta},\zeta %
\right] \right\rangle .  \notag
\end{eqnarray}%
The musical isomorphism $\Omega _{\ ^{1}T^{\ast }TG}^{\flat }$, induced from
the symplectic two-form $\Omega _{\ ^{1}T^{\ast }TG},$ maps the tangent
bundle $T(^{1}T^{\ast }TG)$ to the cotangent bundle $T^{\ast }(^{1}T^{\ast
}TG)$. It takes the right invariant vector field in Eq.(\ref{rivfT*TG}) to
an element of the cotangent bundle $T_{\left( g,\xi ,\mu ,\nu \right)
}^{\ast }(^{1}T^{\ast }TG)$ with coordinates
\begin{equation*}
\Omega _{\ ^{1}T^{\ast }TG}^{\flat }\left( X_{\left( \eta ,\zeta ,\lambda
_{1},\lambda _{2}\right) }^{\text{ }^{1}T^{\ast }TG}\left( g,\xi ,\mu ,\nu
\right) \right) =\left( T_{g}^{\ast }R_{g^{-1}}\left( \lambda _{1}-ad_{\xi
}^{\ast }\lambda _{2}\right) ,\lambda _{2},-\eta ,-\zeta \right) .
\end{equation*}

For a Hamiltonian function(al) $H$ on the symplectic manifold $\left( \
^{1}T^{\ast }TG,\Omega _{\ ^{1}T^{\ast }TG}\right) ,$ the Hamilton's
equations read%
\begin{equation}
i_{X_{H}^{\ ^{1}T^{\ast }TG}}\Omega _{\ ^{1}T^{\ast }TG}=-dH,
\label{HamEqT*TG}
\end{equation}%
where the Hamiltonian vector field $X_{H}^{\ ^{1}T^{\ast }TG\text{ }}$is a
right invariant vector field generated by the element
\begin{equation*}
\left( \frac{\delta H}{\delta \mu },\frac{\delta H}{\delta \nu }%
,-T_{e}^{\ast }R_{g}\left( \frac{\delta H}{\delta g}\right) -ad_{\xi }^{\ast
}\left( \frac{\delta H}{\delta \xi }\right) ,-\frac{\delta H}{\delta \xi }%
\right)
\end{equation*}%
of the Lie algebra $\left( \mathfrak{g}\circledS \mathfrak{g}\right)
\circledS \left( \mathfrak{g}^{\ast }\times \mathfrak{g}^{\ast }\right) $
\cite{AbCaCl13}.

\begin{proposition}
Trivialized Hamilton's equations on $\left( ^{1}T^{\ast }TG,\Omega _{\
^{1}T^{\ast }TG}\right) $ are
\begin{eqnarray}
\frac{dg}{dt} &=&T_{e}R_{g}\frac{\delta H}{\delta \mu },\text{ \ \ }
\label{HamT*TG1} \\
\frac{d\xi }{dt} &=&\frac{\delta H}{\delta \nu }+ad_{\xi }\frac{\delta H}{%
\delta \mu },  \label{HamT*TG2} \\
\frac{d\mu }{dt} &=&-T_{e}^{\ast }R_{g}\frac{\delta H}{\delta g}-ad_{\xi
}^{\ast }\frac{\delta H}{\delta \xi }+ad_{\frac{\delta H}{\delta \mu }%
}^{\ast }\mu +ad_{\frac{\delta H}{\delta \nu }}^{\ast }\nu ,
\label{HamT*TG3} \\
\frac{d\nu }{dt} &=&-\frac{\delta H}{\delta \xi }+ad_{\frac{\delta H}{\delta
\mu }}^{\ast }\nu .  \label{HamT*TG4}
\end{eqnarray}
\end{proposition}

\begin{remark}
The Hamilton's equations (\ref{HamT*TG1})-(\ref{HamT*TG4}) have extra terms,
compared to ones in, for example, \cite{CoDi11, GaHoMeRaVi10}, coming from
the semidirect product structures. Literally, it manifests properties of
adapted trivialization. The trivialization of \cite{CoDi11} is of the second
kind given by Eq.(\ref{2nd}) whereas Eq.(\ref{HamT*TG1})-(\ref{HamT*TG4})
results from trivializations of the first kind.
\end{remark}

From the equations (\ref{HamT*TG2}) and (\ref{HamT*TG4}), we single out $%
\delta H/\delta \nu $ and $\delta H/\delta \xi $, respectively. By
substituting these into Eq.(\ref{HamT*TG3}), we obtain the system%
\begin{equation*}
\left( \frac{d}{dt}-ad_{\frac{\delta H}{\delta \mu }}^{\ast }\right) \left(
ad_{\xi }^{\ast }\nu -\mu \right) =T_{e}^{\ast }R_{g}\frac{\delta H}{\delta g%
}
\end{equation*}%
equivalent to Eq.(\ref{HamT*TG2})-(\ref{HamT*TG4}).

\subsection{Reduction by $G$}

We shall first perform Poisson reduction of Hamiltonian system on $%
^{1}T^{\ast }TG$ under the action of $G$ given by%
\begin{equation}
\left( g;\left( h,\eta ,\nu _{1},\nu _{2}\right) \right) \rightarrow \left(
gh,Ad_{g^{-1}}\eta ,Ad_{g^{-1}}^{\ast }\nu _{1},Ad_{g^{-1}}^{\ast }\nu
_{2}\right) .  \notag
\end{equation}%
In this case, the Hamiltonian is independent of group variable $H=H\left(
\xi ,\mu ,\nu \right) $, hence, is right invariant under $G$. The total
space is the group $\mathfrak{g}_{1}\circledS \left( \mathfrak{g}_{2}^{\ast
}\times \mathfrak{g}_{3}^{\ast }\right) $.

\begin{proposition}
The Poisson reduced manifold $\mathfrak{g}_{1}\circledS \left( \mathfrak{g}%
_{2}^{\ast }\times \mathfrak{g}_{3}^{\ast }\right) $ carries the Poisson
bracket
\begin{eqnarray}
&&\left\{ H,K\right\} _{\mathfrak{g}_{1}\circledS \left( \mathfrak{g}%
_{2}^{\ast }\times \mathfrak{g}_{3}^{\ast }\right) }\left( \xi ,\mu ,\nu
\right) =\left\langle \frac{\delta K}{\delta \xi },\frac{\delta H}{\delta
\nu }\right\rangle -\left\langle \frac{\delta H}{\delta \xi },\frac{\delta K%
}{\delta \nu }\right\rangle +\left\langle \mu ,\left[ \frac{\delta H}{\delta
\mu },\frac{\delta K}{\delta \mu }\right] \right\rangle   \notag \\
&&+\left\langle ad_{\xi }^{\ast }\frac{\delta K}{\delta \xi },\frac{\delta H%
}{\delta \mu }\right\rangle -\left\langle ad_{\xi }^{\ast }\frac{\delta H}{%
\delta \xi },\frac{\delta K}{\delta \mu }\right\rangle +\left\langle \nu ,%
\left[ \frac{\delta H}{\delta \mu },\frac{\delta K}{\delta \nu }\right] -%
\left[ \frac{\delta K}{\delta \mu },\frac{\delta H}{\delta \nu }\right]
\right\rangle ,  \label{Poigxg*xg*}
\end{eqnarray}%
for two right invariant functionals $H$ and $K$.
\end{proposition}

\begin{remark}
$T^{\ast }\mathfrak{g}_{1}=\mathfrak{g}_{1}\times \mathfrak{g}_{3}^{\ast }$
carries\ a canonical Poisson bracket and $\mathfrak{g}_{2}^{\ast }$ carries
Lie-Poisson bracket. The immersions $\mathfrak{g}_{1}\times \mathfrak{g}%
_{3}^{\ast }\rightarrow \mathfrak{g}_{1}\circledS \left( \mathfrak{g}%
_{2}^{\ast }\times \mathfrak{g}_{3}^{\ast }\right) $ and $\mathfrak{g}%
_{2}^{\ast }\rightarrow \mathfrak{g}_{1}\circledS \left( \mathfrak{g}%
_{2}^{\ast }\times \mathfrak{g}_{3}^{\ast }\right) $ are Poisson maps.
However, the Poisson structure described by Eq.(\ref{Poigxg*xg*}) on the space $%
\mathfrak{g}_{1}\circledS \left( \mathfrak{g}_{2}^{\ast }\times \mathfrak{g}%
_{3}^{\ast }\right) $ is not a direct product of canonical Poisson bracket
on $T^{\ast }\mathfrak{g}_{1}=\mathfrak{g}_{1}\times \mathfrak{g}_{3}^{\ast
} $ and Lie-Poisson bracket on $\mathfrak{g}_{2}^{\ast }$. In fact, this
Poisson structure reduces to a direct product structure on $\mathfrak{g}%
_{1}\times \left( \mathfrak{g}_{2}^{\ast }\times \mathfrak{g}_{3}^{\ast
}\right) $ if one uses trivialization of the second kind as in, for example,
\cite{GaHoMeRaVi10} and \cite{CoDi11}. In this case, the terms in the third
line of the Hamilton's equations (\ref{Poigxg*xg*}) are lost.
\end{remark}

\begin{proposition}
The Marsden-Weinstein symplectic reduction by the action of $G$ on $\
^{1}T^{\ast }TG$ with the momentum mapping
\begin{equation*}
\mathbf{J}_{\text{ }^{1}T^{\ast }TG}^{G}:\text{ }^{1}T^{\ast }TG\rightarrow
\mathfrak{g}_{2}^{\ast }:\left( g,\xi ,\mu ,\nu \right) \rightarrow \mu
\end{equation*}%
results in the reduced symplectic two-form $\Omega _{\ ^{1}T^{\ast
}TG}^{\left. G\right\backslash }$ on the reduced space $\mathcal{O}_{\mu
}\times \mathfrak{g}_{1}\times \mathfrak{g}_{3}^{\ast }$. The value of $%
\Omega _{\ ^{1}T^{\ast }TG}^{\left. G\right\backslash }$ on two vectors $%
\left( \eta _{\mathfrak{g}^{\ast }}\left( \mu \right) ,\zeta ,\lambda
\right) $ and $\left( \bar{\eta}_{\mathfrak{g}^{\ast }}\left( \mu \right) ,%
\bar{\zeta},\bar{\lambda}\right) $ is%
\begin{equation}
\Omega _{\ ^{1}T^{\ast }TG}^{\left. G\right\backslash }\left( \left( \eta _{%
\mathfrak{g}^{\ast }}\left( \mu \right) ,\zeta ,\lambda \right) ,\left( \bar{%
\eta}_{\mathfrak{g}^{\ast }}\left( \mu \right) ,\bar{\zeta},\bar{\lambda}%
\right) \right) =\left\langle \lambda ,\bar{\zeta}\right\rangle
-\left\langle \bar{\lambda},\zeta \right\rangle -\left\langle \mu ,[\eta ,%
\bar{\eta}]\right\rangle  \label{RedOhmT*TG}
\end{equation}%
and the reduced Hamilton's equations for a right invariant Hamiltonian $H$
are
\begin{equation*}
\frac{d\zeta }{dt}=\frac{\delta H}{\delta \lambda },\text{ \ \ }\frac{%
d\lambda }{dt}=-\frac{\delta H}{\delta \zeta },\text{\ \ \ }\frac{d\mu }{dt}%
=ad_{\frac{\delta H}{\delta \mu }}^{\ast }\mu .
\end{equation*}
\end{proposition}

\begin{remark}
The reduced space $\mathcal{O}_{\mu }\times \mathfrak{g}_{1}\times \mathfrak{%
g}_{3}^{\ast }$ is a symplectic leaf (\cite{We83}) of the Poisson manifold $%
\mathfrak{g}_{1}\circledS \left( \mathfrak{g}_{2}^{\ast }\times \mathfrak{g}%
_{3}^{\ast }\right) $ as well as of $\mathfrak{g}_{1}\times \mathfrak{g}%
_{2}^{\ast }\times \mathfrak{g}_{3}^{\ast }$ with direct product Poisson
structure
\begin{equation}
\left\{ H,K\right\} _{\mathfrak{g}_{1}\times \mathfrak{g}_{2}^{\ast }\times
\mathfrak{g}_{3}^{\ast }}\left( \xi ,\mu ,\nu \right) =\left\langle \frac{%
\delta K}{\delta \xi },\frac{\delta H}{\delta \nu }\right\rangle
-\left\langle \frac{\delta H}{\delta \xi },\frac{\delta K}{\delta \nu }%
\right\rangle +\left\langle \mu ,\left[ \frac{\delta H}{\delta \mu },\frac{%
\delta K}{\delta \mu }\right] \right\rangle .
\end{equation}%
The latter result was obtained in \cite{CoDi11, GaHoMeRaVi10}.
\end{remark}

\begin{remark}
For a Lagrangian $L=L\left( g,\xi ,\dot{\xi}\right) $ defined on the second
order bundle $T^{2}G$ and whose dynamics is described by Eq.(\ref{soEL}), we
define the energy function
\begin{equation}
H\left( g,\xi ,\dot{\xi},\mu ,\nu \right) =\left\langle \mu ,\xi
\right\rangle +\left\langle \nu ,\dot{\xi}\right\rangle -L\left( g,\xi ,\dot{%
\xi}\right)  \label{LegHamT2G}
\end{equation}%
on the Whitney product $^{1}T^{2}G\times _{G\circledS \mathfrak{g}}\
^{1}T^{\ast }TG$. When the fiber derivative (the Legendre map) $\delta
L/\delta \dot{\xi}=\nu $ is invertible for the Lagrangian $L=L\left( g,\xi ,%
\dot{\xi}\right) $, the energy in Eq.(\ref{LegHamT2G}) becomes a Hamiltonian
function on $^{1}T^{\ast }TG$. In this case, substitution of $H$ into
Hamilton's equations (\ref{HamT*TG1})-(\ref{HamT*TG4}) results in the
Euler-Lagrange equations (\ref{soEL}).\ If $l=l\left( \xi ,\dot{\xi}\right) $
and the fiber derivative $\delta l/\delta \dot{\xi}=\nu $ is invertible,
then the Hamiltonian
\begin{equation}
h\left( \xi ,\mu ,\nu \right) =\left\langle \mu ,\xi \right\rangle
+\left\langle \nu ,\dot{\xi}\right\rangle -l\left( \xi ,\dot{\xi}\right) ,%
\text{ \ \ }\frac{\delta l}{\delta \dot{\xi}}=\nu  \label{LegHamg*xg*xg}
\end{equation}%
generates the second order Euler-Poincaré equation (\ref{soep}).
\end{remark}

The symplectic two-form $\Omega _{\ ^{1}T^{\ast }TG}^{\left.
G\right\backslash }$ given in Eq.(\ref{RedOhmT*TG})\ on $\mathcal{O}_{\mu
}\times \mathfrak{g}_{1}\times \mathfrak{g}_{3}^{\ast }$ is in a direct
product form. Hence a reduction is possible by the additive action of $%
\mathfrak{g}$ the second in $\mathcal{O}_{\mu }\times \mathfrak{g}_{1}\times
\mathfrak{g}_{3}^{\ast }$.

\begin{proposition}
The momentum mapping of additive action of $\mathfrak{g}$ on the symplectic
manifold $(\mathcal{O}_{\mu }\times \mathfrak{g}_{1}\times \mathfrak{g}%
_{3}^{\ast },\Omega _{\ ^{1}T^{\ast }TG}^{\left. G\right\backslash })$ is
\begin{equation*}
\mathbf{J}_{\mathcal{O}_{\mu }\times \mathfrak{g}_{1}\times \mathfrak{g}%
_{3}^{\ast }}^{\mathfrak{g}}:\mathcal{O}_{\mu }\times \mathfrak{g}_{1}\times
\mathfrak{g}_{3}^{\ast }\rightarrow \mathfrak{g}_{3}^{\ast }:\left( \mu ,\xi
,\nu \right) \rightarrow \nu
\end{equation*}%
and the symplectic reduction results in the total space $\mathcal{O}_{\mu }$
with Kostant-Kirillov-Souriou two-form (\ref{KKS}).
\end{proposition}

\subsection{Reduction by $\mathfrak{g}$}

The vector space structure of $\mathfrak{g}$ makes it an Abelian group, and
according to the immersion in Eq.(\ref{immg}), $\mathfrak{g}$ is an Abelian
subgroup of $^{1}T^{\ast }TG$. It acts on the total space $^{1}T^{\ast }TG$
by
\begin{equation}
\left( \xi ;\left( h,\eta ,\mu ,\nu \right) \right) \rightarrow \left( h,\xi
+\eta ,\mu +ad_{\xi }^{\ast }\nu ,\nu \right) .  \label{gonT*TG}
\end{equation}%
Since the action of $G\circledS \mathfrak{g}$ on its cotangent bundle $%
^{1}T^{\ast }TG$ is symplectic, the subgroup $\mathfrak{g}$ of $G\circledS
\mathfrak{g}$ also acts on $^{1}T^{\ast }TG$ symplectically. Following
results describes Poisson and symplectic reductions of $^{1}T^{\ast }TG$ by $%
\mathfrak{g}$ assuming that functions $K,$ $H=H\left( g,\mu ,\nu \right) $
defined on $G\circledS \left( \mathfrak{g}_{2}^{\ast }\times \mathfrak{g}%
_{3}^{\ast }\right) $ are right invariant under the above action of $%
\mathfrak{g}$.

\begin{proposition}
Poisson reduction of $^{1}T^{\ast }TG$ by Abelian subgroup $\mathfrak{g}$
gives the Poisson manifold $G\circledS \left( \mathfrak{g}_{2}^{\ast }\times
\mathfrak{g}_{3}^{\ast }\right) $ endowed with the Poisson bracket
\begin{eqnarray}
\left\{ H,K\right\} _{G\circledS \left( \mathfrak{g}_{2}^{\ast }\times
\mathfrak{g}_{3}^{\ast }\right) } &=&\left\langle T_{e}^{\ast }R_{g}\frac{%
\delta K}{\delta g},\frac{\delta H}{\delta \mu }\right\rangle -\left\langle
T_{e}^{\ast }R_{g}\frac{\delta H}{\delta g},\frac{\delta K}{\delta \mu }%
\right\rangle +\left\langle \mu ,\left[ \frac{\delta H}{\delta \mu },\frac{%
\delta K}{\delta \mu }\right] \right\rangle  \notag \\
&&+\left\langle \nu ,\left[ \frac{\delta H}{\delta \mu },\frac{\delta K}{%
\delta \nu }\right] -\left[ \frac{\delta K}{\delta \mu },\frac{\delta H}{%
\delta \nu }\right] \right\rangle  \label{PoGxg*xg*}
\end{eqnarray}%
evaluated at $\left( g,\mu ,\nu \right) $.
\end{proposition}

\begin{remark}
Recall that $G\times \mathfrak{g}_{2}^{\ast }$ is canonically symplectic
with the Poisson bracket in Eq.(\ref{PoissonGg*}) and the immersion $G\times
\mathfrak{g}_{2}^{\ast }\rightarrow G\circledS \left( \mathfrak{g}_{2}^{\ast
}\times \mathfrak{g}_{3}^{\ast }\right) $ is a Poisson map. On the other
hand, $\mathfrak{g}_{3}^{\ast }$ is naturally Lie-Poisson and $\mathfrak{g}%
_{3}^{\ast }\rightarrow \mathfrak{g}_{1}\circledS \left( \mathfrak{g}%
_{2}^{\ast }\times \mathfrak{g}_{3}^{\ast }\right) $ is a also Poisson map.
The Poisson bracket in Eq.(\ref{PoGxg*xg*}) is, however, not a direct
product of these structures.
\end{remark}

\begin{proposition}
The application of the Marsden-Weinstein symplectic reduction by the action
of $\mathfrak{g}$ on $\ ^{1}T^{\ast }TG$ with the momentum mapping
\begin{equation*}
\mathbf{J}_{\text{ }^{1}T^{\ast }TG}^{\mathfrak{g}}:\text{ }^{1}T^{\ast
}TG\rightarrow \mathfrak{g}_{3}^{\ast }:\left( g,\xi ,\mu ,\nu \right)
\rightarrow \nu
\end{equation*}%
results in the reduced symplectic space $\left( \mathbf{J}_{\text{ }%
^{1}T^{\ast }TG}^{\mathfrak{g}}\right) ^{-1}/\mathfrak{g}$ isomorphic to $%
G\circledS \mathfrak{g}_{2}^{\ast }$ and with the canonical symplectic
two-from $\Omega _{G\circledS \mathfrak{g}_{2}^{\ast }}$ in Eq.(\ref{Ohm2T*G}%
).
\end{proposition}

It follows that the immersion $G\circledS \mathfrak{g}_{2}^{\ast
}\rightarrow G\circledS \left( \mathfrak{g}_{2}^{\ast }\times \mathfrak{g}%
_{3}^{\ast }\right) $ defines symplectic leaves of the Poisson manifold $%
G\circledS \left( \mathfrak{g}_{2}^{\ast }\times \mathfrak{g}_{3}^{\ast
}\right) $.

The symplectic reduction of $G\circledS \mathfrak{g}_{2}^{\ast }$ under the
action of $G$ results with the total space $\mathcal{O}_{\mu }$ with
Kostant-Kirillov-Souriou two-form (\ref{KKS}). We arrive the following
proposition.

\begin{proposition}
Reductions by actions of $\mathfrak{g}$\ and $G$ make the following diagram
commutative
\begin{equation}
\xymatrix{& G\circledS\mathfrak{g}_{1})\circledS(\mathfrak{g}_{2}^{\ast
}\times\mathfrak{g}_{3}^{\ast }) \ar[dl]|-{\text{S.R. by } \mathfrak{g}_{1}}
\ar[dr]|-{\text{S.R. by } G \text{ at } \mu} \\
G\circledS\mathfrak{g}_{2}^{\ast} \ar[dr]|-{\text{S.R. by } G \text{ at }
\mu} && \mathcal{O}_{\mu}\times\mathfrak{g}_{1}\times\mathfrak{g}_{3}^{\ast
} \ar[dl]|-{\text{S.R. by } \mathfrak{g}_{1}} \\ &\mathcal{O}_{\mu} }
\end{equation}
\end{proposition}

Note that, the symplectic reduction of $^{1}T^{\ast }TG$ by the total action
of the group $G\circledS \mathfrak{g}_{1}$ does not result in $\mathcal{O}%
_{\mu }$ as the reduced space. This is a matter of Hamiltonian reduction by
stages theorem \cite{MaMiOrPeRa07}. In the following subsection, we will
discuss the reduction of $^{1}T^{\ast }TG$ under the action of $G\circledS
\mathfrak{g}_{1}$ as well as the implications of the Hamiltonian reduction
by stages theorem for this case.

\begin{remark}
The actions of the subgroups $\mathfrak{g}_{2}^{\ast }$ and $\mathfrak{g}%
_{3}^{\ast }$ are not symplectic, nor are any subgroup in the list (\ref%
{immg}) containing $\mathfrak{g}_{2}^{\ast }$ and $\mathfrak{g}_{3}^{\ast }$%
. There remains only the action of the group $G\circledS \mathfrak{g}$.
\end{remark}

\subsection{Reduction by $G\circledS \mathfrak{g}$}

The Lie algebra of $G\circledS \mathfrak{g}\ $is the space $\mathfrak{g}%
\circledS \mathfrak{g}$ endowed with the semidirect product Lie algebra
bracket
\begin{equation}
\left[ \left( \xi _{1},\xi _{2}\right) ,\left( \eta _{1},\eta _{2}\right) %
\right] _{\mathfrak{g}\circledS \mathfrak{g}}=\left( \left[ \xi _{1},\eta
_{1}\right] ,\left[ \xi _{1},\eta _{2}\right] -\left[ \eta _{1},\xi _{2}%
\right] \right)
\end{equation}%
for $\left( \xi _{1},\xi _{2}\right) $ and $\left( \eta _{1},\eta
_{2}\right) $ in $\mathfrak{g}\circledS \mathfrak{g}$. Accordingly, the dual
space $\mathfrak{g}_{2}^{\ast }\times \mathfrak{g}_{3}^{\ast }$ has the
Lie-Poisson bracket
\begin{eqnarray}
\left\{ F,K\right\} _{\mathfrak{g}_{2}^{\ast }\times \mathfrak{g}_{3}^{\ast
}}\left( \mu ,\nu \right) &=&\left\langle \left( \mu ,\nu \right) ,\left[
\left( \frac{\delta F}{\delta \mu },\frac{\delta F}{\delta \nu }\right)
,\left( \frac{\delta E}{\delta \mu },\frac{\delta E}{\delta \nu }\right) %
\right] _{\mathfrak{g}\circledS \mathfrak{g}}\right\rangle  \notag \\
&=&\left\langle \left( \mu ,\nu \right) ,\left( \left[ \frac{\delta F}{%
\delta \mu },\frac{\delta E}{\delta \mu }\right] ,\left[ \frac{\delta F}{%
\delta \mu },\frac{\delta E}{\delta \nu }\right] -\left[ \frac{\delta E}{%
\delta \mu },\frac{\delta F}{\delta \nu }\right] \right) \right\rangle
\notag \\
&=&\left\langle \mu ,\left[ \frac{\delta F}{\delta \mu },\frac{\delta E}{%
\delta \mu }\right] \right\rangle +\left\langle \nu ,\left[ \frac{\delta F}{%
\delta \mu },\frac{\delta E}{\delta \nu }\right] -\left[ \frac{\delta E}{%
\delta \mu },\frac{\delta F}{\delta \nu }\right] \right\rangle ,
\label{LPBg*xg*}
\end{eqnarray}%
for two functionals $F$ and $K$ on $\mathfrak{g}_{2}^{\ast }\times \mathfrak{%
g}_{3}^{\ast }$.

\begin{proposition}
The Lie-Poisson structure on $\mathfrak{g}_{2}^{\ast }\times \mathfrak{g}%
_{3}^{\ast }$ is given by the bracket in Eq.(\ref{LPBg*xg*}) and the
Hamilton's equations for $H\left( \mu ,\nu \right) $ read%
\begin{equation}
\frac{d\mu }{dt}=ad_{\frac{\delta H}{\delta \mu }}^{\ast }\mu +ad_{\frac{%
\delta H}{\delta \nu }}^{\ast }\nu ,\text{ \ \ }\frac{d\nu }{dt}=ad_{\frac{%
\delta H}{\delta \mu }}^{\ast }\nu .  \label{LPg*g*}
\end{equation}
\end{proposition}

Alternatively, the Lie-Poisson equations (\ref{LPg*g*}) can be obtained by
Poisson reduction of $^{1}T^{\ast }TG$ with the action of $G\circledS
\mathfrak{g}$ given by%
\begin{equation}
\left( \left( g,\xi \right) ;\left( h,\eta ,\mu ,\nu \right) \right)
\rightarrow \left( gh,\xi +Ad_{g^{-1}}\eta ,Ad_{g^{-1}}^{\ast }\left( \mu
+ad_{Ad_{g}\xi }^{\ast }\nu \right) ,Ad_{g^{-1}}^{\ast }\nu \right)
\label{GxgonT*TG}
\end{equation}%
and restricting the Hamiltonian function $H$ to the fiber variables $\left(
\mu ,\nu \right) $. In this case, the Lie-Poisson dynamics of $g$ and $\xi $
remains the same but the dynamics governing $\mu $ and $\nu $ have the
reduced form given by Eq.(\ref{LPg*g*}). This is a manifestation of the fact
that the projections to the last two factors in the trivialization (\ref%
{trT*TG})\ is a momentum map under the left Hamiltonian action of the group $%
G\circledS \mathfrak{g}_{1}$ to its trivialized cotangent bundle $\
^{1}T^{\ast }TG$. Yet another way is to reduce the bracket (\ref{Poigxg*xg*}%
) on $\mathfrak{g}_{1}\circledS \left( \mathfrak{g}_{2}^{\ast }\times
\mathfrak{g}_{3}^{\ast }\right) $ by assuming that functionals depends on
elements of the dual spaces. That is, to consider the Abelian group action
of $\mathfrak{g}_{1}$ on $\mathfrak{g}_{1}\circledS \left( \mathfrak{g}%
_{2}^{\ast }\times \mathfrak{g}_{3}^{\ast }\right) $ given by%
\begin{equation}
\left( \xi ;\left( \eta ,\mu ,\nu \right) \right) \rightarrow \left( \xi
+\eta ,\mu +ad_{\xi }^{\ast }\nu ,\nu \right)  \notag
\end{equation}%
and then, apply Poisson reduction. Note finally that, the immersion $%
\mathfrak{g}_{2}^{\ast }\times \mathfrak{g}_{3}^{\ast }\rightarrow \mathfrak{%
g}_{1}\circledS \left( \mathfrak{g}_{2}^{\ast }\times \mathfrak{g}_{3}^{\ast
}\right) $ is a Poisson map.

Application of the Marsden-Weinstein reduction to the symplectic manifold $%
^{1}T^{\ast }TG$ results in the symplectic leaves of the Poisson structure
on $\mathfrak{g}_{2}^{\ast }\times \mathfrak{g}_{3}^{\ast }$. The action in
Eq.(\ref{GxgonT*TG}) has the momentum mapping
\begin{equation*}
\mathbf{J}_{\text{ }^{1}T^{\ast }TG}^{G\circledS \mathfrak{g}}:\text{ }%
^{1}T^{\ast }TG\rightarrow \mathfrak{g}_{2}^{\ast }\times \mathfrak{g}%
_{3}^{\ast }:\left( g,\xi ,\mu ,\nu \right) \rightarrow \left( \mu ,\nu
\right) .
\end{equation*}%
The pre-image $\left( \mathbf{J}_{\text{ }^{1}T^{\ast }TG}^{G\circledS
\mathfrak{g}}\right) ^{-1}\left( \mu ,\nu \right) $ of an element $\left(
\mu ,\nu \right) \in \mathfrak{g}_{2}^{\ast }\times \mathfrak{g}_{3}^{\ast }$
is diffeomorphic to $G\circledS \mathfrak{g}$. The isotropy group $\left(
G\circledS \mathfrak{g}\right) _{\left( \mu ,\nu \right) }$ of $\left( \mu
,\nu \right) $ consists of pairs $\left( g,\xi \right) $ in $G\circledS
\mathfrak{g}$ satisfying
\begin{equation}
Ad_{\left( g,\xi \right) }^{\ast }\left( \mu ,\nu \right) =\left(
Ad_{g}^{\ast }\left( \mu -ad_{\xi }^{\ast }\nu \right) ,Ad_{g}^{\ast }\nu
\right) =\left( \mu ,\nu \right)  \label{coad2}
\end{equation}%
which means that $g\in G_{\nu }\cap G_{\mu }$ and the representation of $%
Ad_{g}\xi $ on $\nu $ is null, that is $ad_{Ad_{g}\xi }^{\ast }\nu =0.$ From
the general theory, we deduce that the quotient space
\begin{equation*}
\left. \left( \mathbf{J}_{\text{ }^{1}T^{\ast }TG}^{G\circledS \mathfrak{g}%
}\right) ^{-1}\left( \mu ,\nu \right) \right/ \left( G\circledS \mathfrak{g}%
\right) _{\left( \mu ,\nu \right) }\simeq \mathcal{O}_{\left( \mu ,\nu
\right) }
\end{equation*}%
is diffeomorphic to the coadjoint orbit $\mathcal{O}_{\left( \mu ,\nu
\right) }$ in $\mathfrak{g}_{2}^{\ast }\times \mathfrak{g}_{3}^{\ast }$
through the point $\left( \mu ,\nu \right) $ under the action $Ad_{\left(
g,\xi \right) }^{\ast }$ in Eq.(\ref{coad2}), that is,
\begin{equation}
\mathcal{O}_{\left( \mu ,\nu \right) }=\left\{ \left( \mu ,\nu \right) \in
\mathfrak{g}_{2}^{\ast }\times \mathfrak{g}_{3}^{\ast }:Ad_{\left( g,\xi
\right) }^{\ast }\left( \mu ,\nu \right) =\left( \mu ,\nu \right) \right\} .
\label{Omunu}
\end{equation}

\begin{proposition}
The symplectic reduction of $^{1}T^{\ast }TG$ results in the coadjoint orbit
$\mathcal{O}_{\left( \mu ,\nu \right) }$ in $\mathfrak{g}_{2}^{\ast }\times
\mathfrak{g}_{3}^{\ast }$ through the point $\left( \mu ,\nu \right) $. The
reduced symplectic two-form $\Omega _{\ ^{1}T^{\ast }TG}^{G\circledS
\mathfrak{g}\backslash }$ (denoted simply by $\Omega _{\mathcal{O}_{\left(
\mu ,\nu \right) }}$) takes the value%
\begin{equation}
\left\langle \Omega _{\mathcal{O}_{\left( \mu ,\nu \right) }};\left( \eta
,\zeta \right) ,\left( \bar{\eta},\bar{\zeta}\right) \right\rangle \left(
\mu ,\nu \right) =\left\langle \mu ,\left[ \bar{\eta},\eta \right]
\right\rangle +\left\langle \nu ,\left[ \bar{\eta},\zeta \right] -\left[
\eta ,\bar{\zeta}\right] \right\rangle  \label{Symp/Gxg}
\end{equation}%
on two vectors $\left( \eta ,\zeta \right) $ and $\left( \bar{\eta},\bar{%
\zeta}\right) $ in $T_{\left( \mu ,\nu \right) }\mathcal{O}_{\left( \mu ,\nu
\right) }$.
\end{proposition}

This reduction can also be achieved by stages as described in \cite%
{HoMaRa97, MaMiOrPeRa07}. That is, first trivialize dynamics by the action
of Lie algebra $\mathfrak{g}$ on $^{1}T^{\ast }TG$ which results in the
Poisson structure on the product $G\circledS \left( \mathfrak{g}_{2}^{\ast
}\times \mathfrak{g}_{3}^{\ast }\right) $ given by Eq.(\ref{PoGxg*xg*}).
Symplectic leaves of this Poisson structure are spaces diffeomorphic to $%
G\circledS \mathfrak{g}_{2}^{\ast }$ with symplectic two-form given in Eq.(%
\ref{Ohm2T*G}). The isotropy group $G_{\mu }$ of an element $\mu \in
\mathfrak{g}^{\ast }$ acts on $G\circledS \mathfrak{g}_{2}^{\ast }$ by the
same way as assigned in Eq.(\ref{GonGxg*}), that is,
\begin{equation}
G_{\mu }\times \left( G\circledS \mathfrak{g}_{2}^{\ast }\right) \rightarrow
G\circledS \mathfrak{g}_{2}^{\ast }:\left( g;\left( h,\lambda \right)
\right) \rightarrow \left( gh,Ad_{g^{-1}}^{\ast }\lambda \right) .
\label{Gmu}
\end{equation}%
Then, the Hamiltonian reduction by stages theorem states that, the
symplectic reduction of $G\circledS \mathfrak{g}_{2}^{\ast }$ under the
action of $G_{\mu }$ will result in $\mathcal{O}_{\left( \mu ,\nu \right) }$
as the reduced space endowed with the symplectic two-form $\Omega _{\mathcal{%
O}_{\left( \mu ,\nu \right) }}$ in Eq.(\ref{Symp/Gxg}). Following diagram
summarizes the Hamiltonian reduction by stages theorem for the case of $%
^{1}T^{\ast }TG$ under consideration
\begin{equation}
\xymatrix{&& G\circledS\mathfrak{g}_{1})\circledS(\mathfrak{g}_{2}^{\ast
}\times\mathfrak{g}_{3}^{\ast }) \ar[dll]|-{\text{S.R. by } \mathfrak{g}_{1}
\text{ at }\mu } \ar[dd]|-{\text{S.R. by } G\times\mathfrak{g}_{1} \text{ at
}(\mu,\nu)} \\ G\circledS\mathfrak{g}_{2}^{\ast} \ar[drr]|-{\text{S.R. by }
G_{\mu} \text{ at } \nu} \\ &&\mathcal{O}_{(\mu,\nu)} }  \label{HRBS}
\end{equation}%
There exists a momentum mapping $\mathbf{J}_{G\circledS \mathfrak{g}%
_{2}^{\ast }}^{G_{\mu }}$ from $G\circledS \mathfrak{g}_{2}^{\ast }$ to the
dual space $\mathfrak{g}_{\mu }^{\ast }$ of the isotropy subalgebra $%
\mathfrak{g}_{\mu }$ of $G_{\mu }$. Isotropy subgroup $G_{\mu ,\nu }$ of the
coadjoint action is
\begin{equation*}
G_{\mu ,\nu }=\left\{ g\in G_{\mu }:Ad_{g^{-1}}^{\ast }\nu =\nu \right\} .
\end{equation*}%
The quotient symplectic space
\begin{equation*}
\left. \left( \mathbf{J}_{G\circledS \mathfrak{g}_{2}^{\ast }}^{G_{\mu
}}\right) ^{-1}\left( \nu \right) \right/ G_{\mu ,\nu }\simeq \mathcal{O}%
_{\left( \mu ,\nu \right) }
\end{equation*}%
is diffeomorphic to the coadjoint orbit $\mathcal{O}_{\left( \mu ,\nu
\right) }$ defined in (\ref{Omunu}).

It is also possible to establish the Poisson reduction of the symplectic
manifold $G\circledS \mathfrak{g}_{2}^{\ast }$ under the action of the
isotropy group $G_{\mu }$. This results in
\begin{equation*}
G_{\mu }\backslash \left( G\circledS \mathfrak{g}_{2}^{\ast }\right) \simeq
\mathcal{O}_{\mu }\times \mathfrak{g}_{2}^{\ast }
\end{equation*}%
with the Poisson bracket
\begin{equation*}
\left\{ H,K\right\} _{\mathcal{O}_{\mu }\times \mathfrak{g}_{2}^{\ast
}}\left( \mu ,\nu \right) =\left\langle \mu ,\left[ \frac{\delta H}{\delta
\mu },\frac{\delta K}{\delta \mu }\right] \right\rangle +\left\langle \nu ,%
\left[ \frac{\delta H}{\delta \mu },\frac{\delta K}{\delta \nu }\right] -%
\left[ \frac{\delta K}{\delta \mu },\frac{\delta H}{\delta \nu }\right]
\right\rangle .
\end{equation*}%
We note again, that, the Poisson structure on $\mathcal{O}_{\mu }\times
\mathfrak{g}^{\ast }$ is not a direct product of the Lie-Poisson structures
on $\mathcal{O}_{\mu }$ and $\mathfrak{g}_{2}^{\ast }$. Following diagram
illustrates various reductions of $^{1}T^{\ast }TG$ under the actions of $G$%
, $\mathfrak{g}$ and $G\circledS \mathfrak{g}$. Diagram (\ref{DiagramGg*}),
describing reductions of $G\circledS \mathfrak{g}^{\ast }$, can be attached
to the lower right corner of this to have a complete picture of reductions.
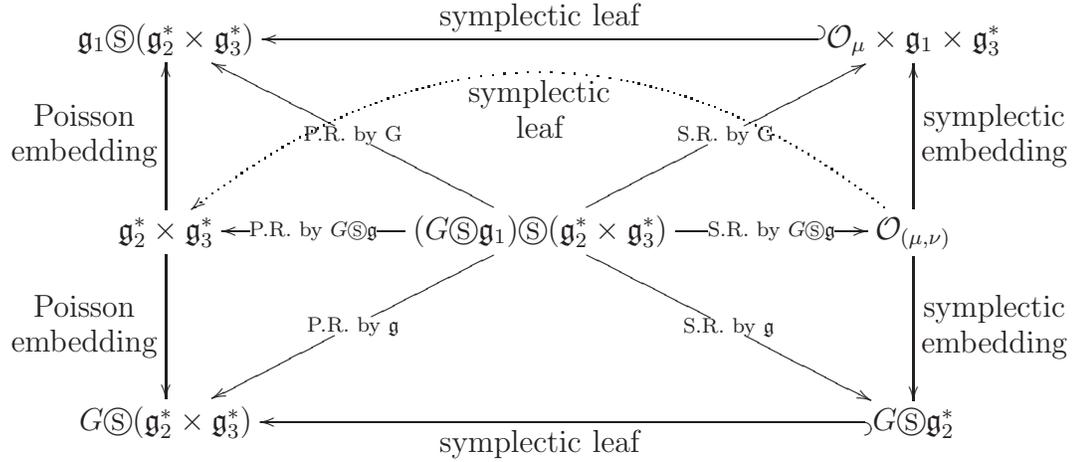
\begin{equation}
\xymatrix{ \mathfrak{g}_{1}\circledS(\mathfrak{g}_{2}^{\ast
}\times\mathfrak{g}_{3}^{\ast }) &&&& \mathcal{O}_{\mu }\times
\mathfrak{g}_{1}\times\mathfrak{g}_{3}^{\ast }
\ar@{_{(}->}[llll]_{\txt{symplectic leaf}} \\\\ \mathfrak{g}_{2}^{\ast
}\times\mathfrak{g}_{3}^{\ast } \ar[dd]_{\txt{Poisson \\ embedding}}
\ar[uu]^{\txt{Poisson \\ embedding}} &&(
G\circledS\mathfrak{g}_{1})\circledS(\mathfrak{g}_{2}^{\ast
}\times\mathfrak{g}_{3}^{\ast }) \ar[uull]|-{\text{P.R. by G}}
\ar[uurr]|-{\text{S.R. by G}} \ar[ddll]|-{\text{P.R. by } \mathfrak{g}}
\ar[ddrr]|-{\text{S.R. by }\mathfrak{g}} \ar[rr]|-{\text{S.R. by
}G\circledS\mathfrak{g}} \ar[ll]|-{\text{P.R. by }G\circledS\mathfrak{g}} &&
\mathcal{O}_{(\mu,\nu)} \ar[dd]^{\txt{symplectic \\ embedding}}
\ar[uu]_{\txt{symplectic \\ embedding}}
\ar@/_{5pc}/@{.>}[llll]^{\txt{symplectic\\leaf}} \\\\
G\circledS(\mathfrak{g}_{2}^{\ast }\times\mathfrak{g}_{3}^{\ast
})&&&&G\circledS\mathfrak{g}_{2}^{\ast } \ar@{^{(}->}[llll]^{\txt{symplectic
leaf}} }  \label{T*TG}
\end{equation}
\newpage
\section{Cotangent Bundle of Cotangent Group}

Global trivialization of the cotangent bundle $T^{\ast }G$ to the semidirect
product $G\circledS \mathfrak{g}_{1}^{\ast }$ in Eq.(\ref{trT*G})\ makes the
identification of its cotangent bundle $T^{\ast }T^{\ast }G\simeq T^{\ast
}\left( G\circledS \mathfrak{g}_{1}^{\ast }\right) $ possible. Hence, the
global trivialization of the iterated cotangent bundle can be achieved by
the semidirect product of the group $G\circledS \mathfrak{g}_{1}^{\ast }$
and dual of its Lie algebra $\mathfrak{g}_{2}^{\ast }\times \mathfrak{g}_{3}$
\cite{EsGu14a}. The trivialization map is
\begin{eqnarray}
tr_{T^{\ast }T^{\ast }G}^{1} &:&T^{\ast }\left( G\circledS \mathfrak{g}%
^{\ast }\right) \rightarrow \left( G\circledS \mathfrak{g}_{1}^{\ast
}\right) \circledS \left( \mathfrak{g}_{2}^{\ast }\times \mathfrak{g}%
_{3}\right) :=\ ^{1}T^{\ast }T^{\ast }G  \notag \\
&:&\left( \alpha _{g},\alpha _{\mu }\right) \rightarrow \left( g,\mu
,T_{e}^{\ast }R_{g}\left( \alpha _{g}\right) -ad_{\alpha _{\mu }}^{\ast }\mu
,\alpha _{\mu }\right)  \label{trT*T*G}
\end{eqnarray}%
and, on $^{1}T^{\ast }T^{\ast }G$, the group multiplication is given by
\begin{eqnarray}
&&\left( g,\mu ,\lambda _{1},\xi _{1}\right) \left( h,\nu ,\lambda _{2},\xi
_{2}\right)  \label{Gr1T*T*G} \\
&=&\left( gh,\mu +Ad_{g^{-1}}^{\ast }\nu ,\lambda _{1}+Ad_{g^{-1}}^{\ast
}\lambda _{2}-ad_{Ad_{g^{-1}}\xi _{2}}^{\ast }\mu ,\xi _{1}+Ad_{g^{-1}}\xi
_{2}\right) .  \notag
\end{eqnarray}

\begin{proposition}
Embeddings of following subspaces
\begin{eqnarray}
&&G,\text{ }\mathfrak{g}_{1}^{\ast }\text{, }\mathfrak{g}_{2}^{\ast }\text{,
}\mathfrak{g}_{3}\text{, }G\circledS \mathfrak{g}_{1}^{\ast }\text{, }%
G\circledS \mathfrak{g}_{2}^{\ast }\text{, }G\circledS \mathfrak{g}_{3}\text{%
, }\mathfrak{g}_{1}^{\ast }\circledS \mathfrak{g}_{2}^{\ast }\text{, }%
\mathfrak{g}_{2}^{\ast }\times \mathfrak{g}_{3},\text{ }  \label{immg2} \\
&&G\circledS \mathfrak{g}_{1}^{\ast }\circledS \mathfrak{g}_{2}^{\ast }\text{%
, }G\circledS \left( \mathfrak{g}_{2}^{\ast }\times \mathfrak{g}_{3}\right)
\text{, }\left( \mathfrak{g}_{1}^{\ast }\times \mathfrak{g}_{3}\right)
\circledS \mathfrak{g}_{2}^{\ast }  \notag
\end{eqnarray}%
define subgroups of $^{1}T^{\ast }T^{\ast }G$ and hence they act on $%
^{1}T^{\ast }T^{\ast }G$ by actions induced by the multiplication in Eq.(\ref%
{Gr1T*T*G}).
\end{proposition}

The group structures on $G\circledS \left( \mathfrak{g}_{2}^{\ast }\times
\mathfrak{g}_{3}\right) $, $G\circledS \mathfrak{g}_{1}^{\ast }\circledS
\mathfrak{g}_{2}^{\ast }$, $\left( \mathfrak{g}_{1}^{\ast }\times \mathfrak{g%
}_{3}\right) \circledS \mathfrak{g}_{2}^{\ast }$ are (up to some reordering)
given by Eqs.(\ref{GrP}), (\ref{GrGg*g*}) and (\ref{Grgg*g*}), respectively.

\subsection{Symplectic Structure}

The canonical one-form and the symplectic two form on $T^{\ast }T^{\ast }G$
can be mapped by $tr_{T^{\ast }T^{\ast }G}^{1}$ to $^{1}T^{\ast }T^{\ast }G$
based on the fact that the trivialization map is a symplectic
diffeomorphism. To define canonical forms, consider a right invariant vector
field $X_{\left( \eta ,\lambda _{1},\lambda _{2},\zeta \right) }^{\text{\ }%
^{1}T^{\ast }T^{\ast }G}$ generated by an element $\left( \eta ,\lambda
_{1},\lambda _{2},\zeta \right) $ in the Lie algebra $\left( \mathfrak{g}%
\circledS \mathfrak{g}^{\ast }\right) \circledS \left( \mathfrak{g}^{\ast
}\times \mathfrak{g}\right) $ of $^{1}T^{\ast }T^{\ast }G$. At the point $%
\left( g,\mu ,\nu ,\xi \right) ,$ the right invariant vector
\begin{equation}
X_{\left( \eta ,\lambda _{1},\lambda _{2},\zeta \right) }^{\text{\ }%
^{1}T^{\ast }T^{\ast }G}=\left( TR_{g}\eta ,\lambda _{1}+ad_{\eta }^{\ast
}\mu ,\lambda _{2}+ad_{\eta }^{\ast }\nu -ad_{\xi }^{\ast }\lambda
_{1},\zeta +[\xi ,\eta ]\right)  \label{RiT*T*G}
\end{equation}%
is an element of $T_{\left( g,\mu ,\nu ,\xi \right) }\left( ^{1}T^{\ast
}T^{\ast }G\right) .$ The values of canonical forms $\theta _{\ ^{1}T^{\ast
}T^{\ast }G}$ and $\Omega _{\text{ }^{1}T^{\ast }T^{\ast }G}$ at right
invariant vector fields can now be evaluated to be
\begin{eqnarray}
\langle \theta _{\ ^{1}T^{\ast }T^{\ast }G},X_{\left( \eta ,\lambda
_{1},\lambda _{2},\zeta \right) }^{\ ^{1}T^{\ast }T^{\ast }G}\rangle
&=&\left\langle \eta ,\nu \right\rangle +\left\langle \lambda _{1},\xi
\right\rangle  \label{thet1T*T*G} \\
\langle \Omega _{\text{ }^{1}T^{\ast }T^{\ast }G};\left( X_{\left( \eta
,\lambda _{1},\lambda _{2},\zeta \right) }^{\ ^{1}T^{\ast }T^{\ast
}G},X_{\left( \bar{\eta},\bar{\lambda}_{1},\bar{\lambda}_{2},\bar{\zeta}%
\right) }^{\ ^{1}T^{\ast }T^{\ast }G}\right) \rangle &=&\left\langle \lambda
_{2},\bar{\eta}\right\rangle +\left\langle \zeta -\left[ \eta ,\xi \right] ,%
\bar{\lambda}_{1}\right\rangle -\left\langle \bar{\lambda}_{2},\eta
\right\rangle  \notag \\
&&+\left\langle \left[ \bar{\eta},\xi \right] -\bar{\zeta},\lambda
_{1}\right\rangle +\left\langle \nu ,[\eta ,\bar{\eta}]\right\rangle .
\label{OhmT*T*G}
\end{eqnarray}%
The musical isomorphism $\Omega _{\text{ }^{1}T^{\ast }T^{\ast }G}^{\flat }$%
, induced from the symplectic two-form $\Omega _{\text{ }^{1}T^{\ast
}T^{\ast }G}$ in Eq.(\ref{OhmT*T*G}), maps $T\left( ^{1}T^{\ast }T^{\ast
}G\right) $ to $T^{\ast }\left( ^{1}T^{\ast }T^{\ast }G\right) $. At the
point $\left( g,\mu ,\nu ,\xi \right) $, $\Omega _{\text{ }^{1}T^{\ast
}T^{\ast }G}^{\flat }$ takes the vector in Eq.(\ref{RiT*T*G}) to the element%
\begin{equation*}
\Omega _{\text{ }^{1}T^{\ast }T^{\ast }G}^{\flat }\left( X_{\left( \eta
,\lambda _{1},\lambda _{2},\zeta \right) }^{\text{\ }^{1}T^{\ast }T^{\ast
}G}\right) =\left( T_{g}^{\ast }R_{g^{-1}}\left( \lambda _{2}+ad_{\zeta
}^{\ast }\mu \right) ,\zeta ,-\eta ,-\lambda _{1}\right)
\end{equation*}%
in $T_{\left( g,\mu ,\nu ,\xi \right) }^{\ast }\left( \ ^{1}T^{\ast }T^{\ast
}G\right) .$

\begin{proposition}
A Hamiltonian function $H$ on$\ ^{1}T^{\ast }T^{\ast }G$ determines the
Hamilton's equations
\begin{equation*}
i_{X_{H}^{^{1}T^{\ast }T^{\ast }G}}\Omega _{^{1}T^{\ast }T^{\ast }G}=-DH
\end{equation*}%
by uniquely defining Hamiltonian vector field $X_{H}^{^{1}T^{\ast }T^{\ast
}G}$. The Hamiltonian vector field is a right invariant vector field
generated by four-tuple Lie algebra element
\begin{equation*}
\left( \frac{\delta H}{\delta \nu },\frac{\delta H}{\delta \xi },ad_{\frac{%
\delta H}{\delta \mu }}^{\ast }\mu -T_{e}^{\ast }R_{g}\left( \frac{\delta H}{%
\delta g}\right) ,-\frac{\delta H}{\delta \mu }\right)
\end{equation*}%
in $\left( \mathfrak{g}\circledS \mathfrak{g}^{\ast }\right) \circledS
\left( \mathfrak{g}^{\ast }\times \mathfrak{g}\right) $. At the point $%
\left( g,\mu ,\nu ,\xi \right) ,$ the Hamilton's equations are%
\begin{eqnarray}
\frac{dg}{dt} &=&T_{e}R_{g}\left( \frac{\delta H}{\delta \nu }\right) ,\text{
\ \ }  \label{HamT*T*G1} \\
\frac{d\mu }{dt} &=&\frac{\delta H}{\delta \xi }+ad_{\frac{\delta H}{\delta
\nu }}^{\ast }\mu \\
\frac{d\nu }{dt} &=&ad_{\frac{\delta H}{\delta \mu }}^{\ast }\mu +ad_{\frac{%
\delta H}{\delta \nu }}^{\ast }\nu -T_{e}^{\ast }R_{g}\left( \frac{\delta H}{%
\delta g}\right) -ad_{\xi }^{\ast }\frac{\delta H}{\delta \xi } \\
\frac{d\xi }{dt} &=&-\frac{\delta H}{\delta \mu }+[\xi ,\frac{\delta H}{%
\delta \nu }].  \label{HamT*T*G}
\end{eqnarray}
\end{proposition}

\subsection{Reduction by $G$}

It follows from Eq.(\ref{Gr1T*T*G}) that the left action of $G$ on $%
^{1}T^{\ast }T^{\ast }G$ is%
\begin{equation}
\left( g,\left( h,\nu ,\lambda _{2},\xi _{2}\right) \right) \rightarrow
\left( gh,Ad_{g^{-1}}^{\ast }\nu ,Ad_{g^{-1}}^{\ast }\lambda
_{2},Ad_{g^{-1}}\xi _{2}\right)  \label{GonT*T*G}
\end{equation}%
with the infinitesimal generator $X_{\left( \eta ,0,0,0\right) }^{\
^{1}T^{\ast }T^{\ast }G}$ being a right invariant vector field as in Eq.(\ref%
{RiT*T*G}) generated by $\left( \eta ,0,0,0\right) $ for $\eta \in \mathfrak{%
g}$.

\begin{proposition}
Poisson reduction of $^{1}T^{\ast }T^{\ast }G$ under the action $G$ results
in $\mathfrak{g}_{1}^{\ast }\circledS \left( \mathfrak{g}_{2}^{\ast }\times
\mathfrak{g}_{3}\right) $ endowed with the Poisson bracket%
\begin{eqnarray}
&&\left\{ H,K\right\} _{\mathfrak{g}_{1}^{\ast }\circledS \left( \mathfrak{g}%
_{2}^{\ast }\times \mathfrak{g}_{3}\right) }\left( \mu ,\nu ,\xi \right)
=\left\langle \frac{\delta K}{\delta \mu },\frac{\delta H}{\delta \xi }%
\right\rangle -\left\langle \frac{\delta H}{\delta \mu },\frac{\delta K}{%
\delta \xi }\right\rangle +\left\langle \nu ,\left[ \frac{\delta H}{\delta
\nu },\frac{\delta K}{\delta \nu }\right] \right\rangle  \notag \\
&&+\left\langle \xi ,ad_{\frac{\delta K}{\delta \nu }}^{\ast }\frac{\delta H%
}{\delta \xi }-ad_{\frac{\delta H}{\delta \nu }}^{\ast }\frac{\delta K}{%
\delta \xi }\right\rangle +\left\langle \mu ,\left[ \frac{\delta H}{\delta
\mu },\frac{\delta K}{\delta \nu }\right] -\left[ \frac{\delta K}{\delta \mu
},\frac{\delta H}{\delta \nu }\right] \right\rangle ,  \label{Poig*g*g}
\end{eqnarray}%
and symplectic reduction gives $\mathcal{O}_{\mu }\times \mathfrak{g}\times
\mathfrak{g}^{\ast }$ with the symplectic two-form defined by
\begin{equation}
\Omega _{\ ^{1}T^{\ast }T^{\ast }G}^{\left. G\right\backslash }\left( \left(
\eta _{\mathfrak{g}^{\ast }}\left( \mu \right) ,\lambda ,\zeta \right)
,\left( \bar{\eta}_{\mathfrak{g}^{\ast }}\left( \mu \right) ,\bar{\lambda},%
\bar{\zeta}\right) \right) =\left\langle \zeta ,\bar{\lambda}\right\rangle
-\left\langle \bar{\zeta},\lambda \right\rangle -\left\langle \mu ,[\eta ,%
\bar{\eta}]\right\rangle  \label{RedOhmT*T*G}
\end{equation}%
on two elements $\left( \eta _{\mathfrak{g}^{\ast }}\left( \mu \right)
,\lambda ,\zeta \right) $ and $\left( \bar{\eta}_{\mathfrak{g}^{\ast
}}\left( \mu \right) ,\bar{\lambda},\bar{\zeta}\right) $ of $T_{\mu }%
\mathcal{O}_{\mu }\times \mathfrak{g}\times \mathfrak{g}^{\ast }$.
\end{proposition}

Recall that, in the previous section, Poisson and symplectic reductions of $%
^{1}T^{\ast }TG$ resulted in reduced spaces $\mathfrak{g}\circledS \left(
\mathfrak{g}^{\ast }\times \mathfrak{g}^{\ast }\right) $ and $\mathcal{O}%
_{\mu }\times \mathfrak{g}\times \mathfrak{g}^{\ast }$, respectively. The
reduced Poisson bracket $\left\{ \text{ },\text{ }\right\} _{\mathfrak{g}%
\circledS \left( \mathfrak{g}^{\ast }\times \mathfrak{g}^{\ast }\right) }$
on $\mathfrak{g}\circledS \left( \mathfrak{g}^{\ast }\times \mathfrak{g}%
^{\ast }\right) $ was given by Eq.(\ref{Poigxg*xg*}) whereas the reduced
symplectic two-form $\Omega _{\ ^{1}T^{\ast }TG}^{\left. G\right\backslash }$
on $\mathcal{O}_{\mu }\times \mathfrak{g}\times \mathfrak{g}$ was given in
Eq.(\ref{RedOhmT*TG}). We have the following theorem \cite{EsGu14a} relating
the reductions of cotangent bundles $^{1}T^{\ast }T^{\ast }G$ and $%
^{1}T^{\ast }TG$. We refer \cite{RaKu83} for a detailed study on the canonical maps between semidirect products.

\begin{proposition}
For the trivialized symplectic spaces $^{1}T^{\ast }T^{\ast }G$ and $%
^{1}T^{\ast }TG$, and their reductions, we have the following commutative
diagram
\begin{equation}
\xymatrix{^{1}T^{\ast }T^{\ast }G \ar[dd]|-{\text{P.R. by }G}
\ar[rrr]_{\Gamma} \ar@{->}@/_{5pc}/[dddd]_{\text{S.R. by }G} &&& ^{1}T^{\ast
}TG \ar[dd]|-{\text{P.R. by }G} \ar@{->}@/^{5pc}/[dddd]^{\text{S.R. by }G}
\\\\ \mathfrak{g}^{\ast }\circledS(\mathfrak{g}^{\ast }\times\mathfrak{g})
\ar[rrr]_{\gamma^{P}} &&& \mathfrak{g}\circledS(\mathfrak{g}^{\ast
}\times\mathfrak{g}^{\ast }) \\\\ \mathcal{O}_{\mu }\times\mathfrak{g}^{\ast
}\times\mathfrak{g} \ar[rrr]_{\gamma^{S}}
\ar@{^{(}->}[uu]_{\txt{symplectic\\leaf}} &&& \mathcal{O}_{\mu
}\times\mathfrak{g}\times\mathfrak{g}^{\ast }
\ar@{^{(}->}[uu]^{\txt{symplectic\\leaf}} }
\end{equation}%
where the diffeomorphisms are given by
\begin{eqnarray*}
\Gamma &:&\text{ }^{1}T^{\ast }T^{\ast }G\rightarrow \text{ }^{1}T^{\ast
}TG:\left( g,\mu ,\nu ,\xi \right) \rightarrow \left( g,-\xi ,\nu ,\mu
\right) \\
\gamma ^{P} &:&\mathfrak{g}^{\ast }\circledS \left( \mathfrak{g}^{\ast
}\times \mathfrak{g}\right) \rightarrow \mathfrak{g}\circledS \left(
\mathfrak{g}^{\ast }\times \mathfrak{g}^{\ast }\right) :\left( \mu ,\nu ,\xi
\right) \rightarrow \left( -\xi ,\nu ,\mu \right) \\
\gamma ^{S} &:&\mathcal{O}_{\mu }\times \mathfrak{g}\times \mathfrak{g}%
^{\ast }\rightarrow \mathcal{O}_{\mu }\times \mathfrak{g}\times \mathfrak{g}%
^{\ast }:\left( Ad_{g^{-1}}^{\ast }\mu ,\xi ,\nu \right) \rightarrow \left(
Ad_{g^{-1}}^{\ast }\mu ,-\xi ,\nu \right) .
\end{eqnarray*}%
$\Gamma $ and $\gamma ^{S}$ are symplectic diffeomorphisms while $\gamma
^{P} $ is a Poisson mapping.
\end{proposition}

\subsection{Reduction by $\mathfrak{g}^{\ast }$}

An action of $\mathfrak{g}^{\ast }$ on $^{1}T^{\ast }T^{\ast }G$, given by%
\begin{equation}
\left( \lambda ;\left( g,\mu ,\nu ,\xi \right) \right) \rightarrow \left(
g,\lambda +\mu ,\nu -ad_{\xi }^{\ast }\lambda ,\xi \right)  \label{g*onT*T*G}
\end{equation}%
with infinitesimal generator $X_{\left( 0,\lambda _{1},0,0\right) }^{\text{\
}^{1}T^{\ast }T^{\ast }G}=\left( 0,\lambda _{1},-ad_{\xi }^{\ast }\lambda
_{1},0\right) $, is symplectic, because the action of $G\circledS \mathfrak{g%
}^{\ast }$ on its cotangent bundle $^{1}T^{\ast }T^{\ast }G$ is symplectic,
and $\mathfrak{g}^{\ast }$ is a subgroup. Hence we can perform a Poisson and
a symplectic reductions of $^{1}T^{\ast }T^{\ast }G$ and arrive at the
following proposition.

\begin{proposition}
The Poisson reduction of $^{1}T^{\ast }T^{\ast }G$ with the action of $%
\mathfrak{g}^{\ast }$ results in $G\circledS \left( \mathfrak{g}_{2}^{\ast
}\times \mathfrak{g}_{3}\right) $ endowed with the bracket
\begin{eqnarray}
&&\left\{ H,K\right\} _{G\circledS \left( \mathfrak{g}_{2}^{\ast }\times
\mathfrak{g}_{3}\right) }\left( g,\nu ,\xi \right) =\left\langle T_{e}^{\ast
}R_{g}\frac{\delta K}{\delta g},\frac{\delta H}{\delta \nu }\right\rangle
-\left\langle T_{e}^{\ast }R_{g}\frac{\delta H}{\delta g},\frac{\delta K}{%
\delta \nu }\right\rangle  \notag \\
&&+\left\langle \left[ \frac{\delta K}{\delta \nu },\xi \right] ,\frac{%
\delta H}{\delta \xi }\right\rangle -\left\langle \left[ \frac{\delta H}{%
\delta \nu },\xi \right] ,\frac{\delta K}{\delta \xi }\right\rangle
+\left\langle \nu ,\left[ \frac{\delta H}{\delta \nu },\frac{\delta K}{%
\delta \nu }\right] \right\rangle .  \label{PoGgg*}
\end{eqnarray}%
The application of Marsden-Weinstein symplectic reduction with the action of
$\mathfrak{g}^{\ast }$ on $^{1}T^{\ast }T^{\ast }G$ having the momentum
mapping
\begin{equation*}
\mathbf{J}_{\text{ }^{1}T^{\ast }T^{\ast }G}^{\mathfrak{g}^{\ast }}:\text{ }%
^{1}T^{\ast }T^{\ast }G\rightarrow \mathfrak{g}_{2}:\left( g,\mu ,\nu ,\xi
\right) \rightarrow \xi
\end{equation*}
results in the reduced symplectic space $\left( \mathbf{J}_{\text{ }%
^{1}T^{\ast }T^{\ast }G}^{\mathfrak{g}^{\ast }}\right) ^{-1}\left( \xi
\right) /\mathfrak{g}^{\ast }$ isomorphic to $G\circledS \mathfrak{g}%
_{3}^{\ast }$ with the canonical symplectic two-from $\Omega _{G\circledS
\mathfrak{g}_{2}^{\ast }}$ in Eq.(\ref{Ohm2T*G}).
\end{proposition}

\begin{remark}
The actions of subgroups $\mathfrak{g}_{2}^{\ast }$ and $\mathfrak{g}_{3}$,
and hence any subgroup in the list (\ref{immg2}) containing $\mathfrak{g}%
_{2}^{\ast }$ and $\mathfrak{g}_{3}$, are not symplectic. Thus, there
remains only the action of the group $G\circledS \mathfrak{g}_{1}^{\ast }$.
\end{remark}

\subsection{Reduction by $G\circledS \mathfrak{g}^{\ast }$}

Lie algebra of the group $G\circledS \mathfrak{g}_{1}^{\ast }$ is the space $%
\mathfrak{g}\circledS \mathfrak{g}^{\ast }$ carrying the bracket%
\begin{equation}
\left[ \left( \xi ,\mu \right) ,\left( \eta ,\nu \right) \right] _{\mathfrak{%
g}\circledS \mathfrak{g}^{\ast }}=\left( \left[ \xi ,\eta \right] ,ad_{\eta
}^{\ast }\mu -ad_{\xi }^{\ast }\nu \right) .  \label{LBgg*}
\end{equation}%
The dual space $\mathfrak{g}_{2}^{\ast }\times \mathfrak{g}_{3}$ carries the
Lie-Poisson bracket
\begin{eqnarray}
\left\{ F,E\right\} _{\mathfrak{g}_{2}^{\ast }\times \mathfrak{g}_{3}}\left(
\nu ,\xi \right) &=&\left\langle \left( \nu ,\xi \right) ,\left[ \left(
\frac{\delta F}{\delta \nu },\frac{\delta F}{\delta \xi }\right) ,\left(
\frac{\delta E}{\delta \nu },\frac{\delta E}{\delta \xi }\right) \right] _{%
\mathfrak{g}\circledS \mathfrak{g}^{\ast }}\right\rangle  \notag \\
&=&\left\langle \left( \nu ,\xi \right) ,\left( \left[ \frac{\delta F}{%
\delta \nu },\frac{\delta E}{\delta \nu }\right] ,ad_{\frac{\delta E}{\delta
\nu }}^{\ast }\frac{\delta F}{\delta \xi }-ad_{\frac{\delta F}{\delta \nu }%
}^{\ast }\frac{\delta E}{\delta \xi }\right) \right\rangle  \notag \\
&=&\left\langle \nu ,\left[ \frac{\delta F}{\delta \nu },\frac{\delta E}{%
\delta \nu }\right] \right\rangle +\left\langle \xi ,ad_{\frac{\delta E}{%
\delta \nu }}^{\ast }\frac{\delta F}{\delta \xi }-ad_{\frac{\delta F}{\delta
\nu }}^{\ast }\frac{\delta E}{\delta \xi }\right\rangle ,  \label{LPE}
\end{eqnarray}%
where $\left[ \text{ },\text{ }\right] _{\mathfrak{g}\circledS \mathfrak{g}%
^{\ast }}$ is the Lie algebra bracket on $\mathfrak{g}\circledS \mathfrak{g}%
^{\ast }$ given in Eq.(\ref{LBgg*}).

\begin{proposition}
The Lie-Poisson bracket, in Eq.(\ref{LPE}), on $\mathfrak{g}_{2}^{\ast
}\times \mathfrak{g}_{3}$ defines the Hamiltonian vector field $X_{E}^{%
\mathfrak{g}^{\ast }\times \mathfrak{g}}$ by
\begin{equation*}
\left\{ F,E\right\} _{\mathfrak{g}_{2}^{\ast }\times \mathfrak{g}%
_{3}}=-\left\langle dF,X_{E}^{\mathfrak{g}_{2}^{\ast }\times \mathfrak{g}%
_{3}}\right\rangle
\end{equation*}%
whose components are the Lie-Poisson equations
\begin{equation}
\frac{d\nu }{dt}=ad_{\frac{\delta H}{\delta \nu }}^{\ast }\nu -ad_{\xi
}^{\ast }\frac{\delta H}{\delta \xi },\ \ \frac{d\xi }{dt}=[\xi ,\frac{%
\delta H}{\delta \nu }].  \label{LPg*g}
\end{equation}
\end{proposition}

Although the calculation in (\ref{LPE}) is a proof of the proposition, it is
possible to arrive the Lie-Poisson equations (\ref{LPg*g}) by starting from
the Hamilton's equations (\ref{HamT*T*G1})-(\ref{HamT*T*G}) on $^{1}T^{\ast
}T^{\ast }G$ and applying this system a Poisson reduction with the action of
$G\circledS \mathfrak{g}^{\ast }$ given by
\begin{eqnarray}
\left( G\circledS \mathfrak{g}^{\ast }\right) \times \text{ }^{1}T^{\ast
}T^{\ast }G &\rightarrow &\text{ }^{1}T^{\ast }T^{\ast }G:
\label{T*GonT*T*G} \\
((g,\mu),(h,\nu ,\lambda _{2},\xi _{2}))&\rightarrow &(gh,\mu
+Ad_{g^{-1}}^{\ast }\nu ,Ad_{g^{-1}}^{\ast }\lambda _{2}-ad_{Ad_{g^{-1}}\xi
_{2}}^{\ast }\mu ,Ad_{g^{-1}}\xi _{2}) .  \notag
\end{eqnarray}%
In short, the Poisson reduction is to choose the Hamiltonian function $H$ in
Eqs.(\ref{HamT*T*G1})-(\ref{HamT*T*G}) depending on fiber variables, that is
$H=H\left( \nu ,\xi \right) $ and to arrive the Lie-Poisson equations (\ref%
{LPg*g}).

To reduce the Hamilton's equations (\ref{HamT*T*G1})-(\ref{HamT*T*G}) on $%
^{1}T^{\ast }T^{\ast }G$ symplectically, we first compute the momentum
mapping
\begin{equation*}
\mathbf{J}_{G\circledS \mathfrak{g}_{3}^{\ast }}^{G_{\xi }}:\text{ }%
^{1}T^{\ast }T^{\ast }G\rightarrow \mathfrak{g}^{\ast }\times \mathfrak{g}%
:\left( g,\mu ,\nu ,\xi \right) \rightarrow \left( \nu ,\xi \right) ,
\end{equation*}%
associated with the action of $G\circledS \mathfrak{g}^{\ast }$ in Eq.(\ref%
{T*GonT*T*G}) and the quotient space%
\begin{equation}
\left. \left( \mathbf{J}_{G\circledS \mathfrak{g}_{3}^{\ast }}^{G_{\xi
}}\right) ^{-1}\left( \nu ,\xi \right) \right/ G_{\left( \nu ,\xi \right)
}\simeq \mathcal{O}_{\left( \nu ,\xi \right) }.  \label{Onuxi}
\end{equation}%
Here, $G_{\left( \nu ,\xi \right) }$ is the isotropy subgroup of $G\circledS
\mathfrak{g}^{\ast }$ consisting of elements preserved under the coadjoint
action $G\circledS \mathfrak{g}^{\ast }$\ on the dual space $\mathfrak{g}%
^{\ast }\times \mathfrak{g}$ of its Lie algebra
\begin{eqnarray}
Ad^{\ast } &:&\left( G\circledS \mathfrak{g}^{\ast }\right) \times \left(
\mathfrak{g}^{\ast }\times \mathfrak{g}\right) \rightarrow \mathfrak{g}%
^{\ast }\times \mathfrak{g}  \notag \\
&:&\left( \left( g,\mu \right) ,\left( \nu ,\xi \right) \right) \rightarrow
\left( Ad_{g}^{\ast }\left( \nu +ad_{\xi }^{\ast }\mu \right) ,Ad_{g}\xi
\right)  \label{Coad}
\end{eqnarray}%
and the space $\mathcal{O}_{\left( \nu ,\xi \right) }$ is the coadjoint
orbit passing through the point $\left( \nu ,\xi \right) $ under this
coadjoint action. We arrive the reduced symplectic space $\mathcal{O}%
_{\left( \nu ,\xi \right) }$ endowed with the reduced symplectic two-form $%
\Omega _{\ ^{1}T^{\ast }T^{\ast }G}^{G\circledS \mathfrak{g}^{\ast
}\backslash }$, and summarize these results in the following proposition.

\begin{proposition}
The symplectic reduction of $^{1}T^{\ast }T^{\ast }G$ results in the
coadjoint orbit $\mathcal{O}_{\left( \nu ,\xi \right) }$ in $\mathfrak{g}%
_{2}^{\ast }\times \mathfrak{g}^{\ast }$ through the point $\left( \nu ,\xi
\right) $. The reduced symplectic two-form $\Omega _{\ ^{1}T^{\ast }T^{\ast
}G}^{G\circledS \mathfrak{g}^{\ast }\backslash }$ (denoted simply by $\Omega
_{\mathcal{O}_{\left( \nu ,\xi \right) }}$) takes the value%
\begin{equation}
\left\langle \Omega _{\mathcal{O}_{\left( \nu ,\xi \right) }};\left( \lambda
,\eta \right) ,\left( \bar{\lambda},\bar{\eta}\right) \right\rangle \left(
\nu ,\xi \right) =\left\langle \nu ,[\bar{\eta},\eta ]\right\rangle
+\left\langle \xi ,ad_{\eta }^{\ast }\bar{\lambda}-ad_{\bar{\eta}}^{\ast
}\lambda ]\right\rangle  \label{SymOr}
\end{equation}%
on two vectors $\left( \lambda ,\eta \right) $ and $\left( \bar{\lambda},%
\bar{\eta}\right) $ in $T_{\left( \nu ,\xi \right) }\mathcal{O}_{\left( \nu
,\xi \right) }$.
\end{proposition}

To arrive the previous proposition, we performed the reduction $^{1}T^{\ast
}T^{\ast }G\rightarrow \mathcal{O}_{\left( \nu ,\xi \right) }$ in one step
by applying the Marsden-Weinstein theorem to $^{1}T^{\ast }T^{\ast }G$ with
the action of $G\circledS \mathfrak{g}^{\ast }$. Alternatively, we can
perform this reduction in two steps by applying the Hamiltonian reduction by
stages theorem \cite{MaMiOrPeRa07}. In the first step, the symplectic
reduction of $^{1}T^{\ast }T^{\ast }G$ with the action of $\mathfrak{g}%
^{\ast }$ must be performed. This has already been established in the
previous subsection and resulted in the reduced symplectic space $\left(
\mathbf{J}_{\text{ }^{1}T^{\ast }T^{\ast }G}^{\mathfrak{g}^{\ast }}\right)
^{-1}\left( \xi \right) /\mathfrak{g}^{\ast }$, isomorphic to $G\circledS
\mathfrak{g}_{3}^{\ast }$, with the canonical symplectic two-from $\Omega
_{G\circledS \mathfrak{g}_{3}^{\ast }}$ in Eq.(\ref{Ohm2T*G}). For the
second step, we recall the adjoint group action $Ad_{g^{-1}}$ of $G$ on $%
\mathfrak{g}$ and define the isotropy subgroup
\begin{equation}
G_{\xi }=\left\{ g\in G:Ad_{g^{-1}}\xi =\xi \right\}  \label{Gxi}
\end{equation}%
for an element$\ \xi \in \mathfrak{g}$ under the adjoint action. Lie
subalgebra $\mathfrak{g}_{\xi }$ of $G_{\xi }$ consists of vectors $\eta \in
\mathfrak{g}$ satisfying $\left[ \eta ,\xi \right] =0$. The isotropy
subgroup $G_{\xi }$ acts on $G\circledS \mathfrak{g}_{3}^{\ast }$ by the
same way as described in Eq.(\ref{GonGxg*}). This action is Hamiltonian and
has the momentum mapping
\begin{equation*}
\mathbf{J}_{G\circledS \mathfrak{g}_{3}^{\ast }}^{G_{\xi }}:G\circledS
\mathfrak{g}_{3}^{\ast }\rightarrow \mathfrak{g}_{\xi }^{\ast },
\end{equation*}%
where $\mathfrak{g}_{\xi }^{\ast }$ is the dual space of $\mathfrak{g}_{\xi
} $. The quotient space%
\begin{equation*}
\left. \left( \mathbf{J}_{G\circledS \mathfrak{g}_{3}^{\ast }}^{G_{\xi
}}\right) ^{-1}\left( \nu \right) \right/ G_{\xi ,\nu }\simeq \mathcal{O}%
_{\left( \nu ,\xi \right) }
\end{equation*}%
is diffeomorphic to the coadjoint orbit $\mathcal{O}_{\left( \nu ,\xi
\right) }$ in Eq.(\ref{Onuxi}).

Following diagram summarizes reductions of $^{1}T^{\ast }T^{\ast }G$ and its
subbundles.

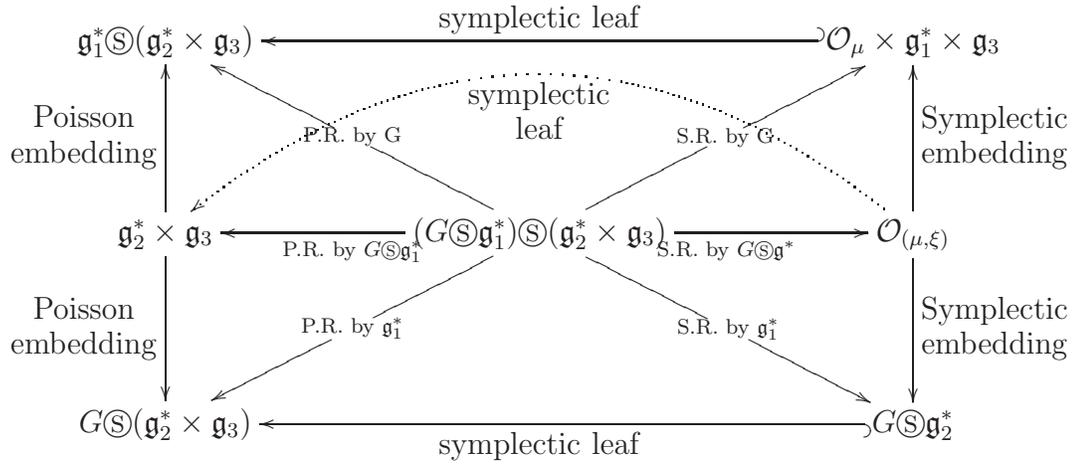
\begin{equation}
\xymatrix{\mathfrak{g}_{1}^{\ast }\circledS(\mathfrak{g}_{2}^{\ast
}\times\mathfrak{g}_{3}) &&&& \mathcal{O}_{\mu }\times
\mathfrak{g}_{1}^{\ast }\times\mathfrak{g}_{3}
\ar@{_{(}->}[llll]_{\txt{symplectic leaf}} \\\\ \mathfrak{g}_{2}^{\ast
}\times\mathfrak{g}_{3} \ar[dd]_{\txt{Poisson\\embedding}}
\ar[uu]^{\txt{Poisson\\embedding}} && (G\circledS\mathfrak{g}_{1}^{\ast
})\circledS(\mathfrak{g}_{2}^{\ast }\times\mathfrak{g}_{3})
\ar[uull]|-{\text{P.R. by G}} \ar[uurr]|-{\text{S.R. by G}}
\ar[ddll]|-{\text{P.R. by } \mathfrak{g}^{\ast }_{1}} \ar[ddrr]|-{\text{S.R.
by }\mathfrak{g}_{1}^{\ast }} \ar[rr]_{\text{S.R. by
}G\circledS\mathfrak{g}^{\ast }} \ar[ll]^{\text{P.R. by
}G\circledS\mathfrak{g}_{1}^{\ast }} && \mathcal{O}_{(\mu,\xi)}
\ar[dd]^{\txt{Symplectic\\embedding}} \ar[uu]_{\txt{Symplectic\\embedding}}
\ar@/_{5pc}/@{.>}[llll]^{\txt{symplectic\\leaf}} \\\\
G\circledS(\mathfrak{g}_{2}^{\ast }\times\mathfrak{g}_{3}) &&&&
G\circledS\mathfrak{g}_{2}^{\ast } \ar@{^{(}->}[llll]^{\txt{symplectic
leaf}} }  \label{T*T*G}
\end{equation}
\newpage
\section{Tangent Bundle of Cotangent Group}

$TT^{\ast }G\simeq T\left( G\circledS \mathfrak{g}^{\ast }\right) .$ $%
T\left( G\circledS \mathfrak{g}^{\ast }\right) $ can be trivialized as the
semidirect product of the group $G\circledS \mathfrak{g}^{\ast }$ and its
Lie algebra $\mathfrak{g}\circledS \mathfrak{g}^{\ast }$. This reads%
\begin{eqnarray}
tr_{TT^{\ast }G}^{1} &:&T\left( G\circledS \mathfrak{g}^{\ast }\right)
\rightarrow \left( G\circledS \mathfrak{g}_{1}^{\ast }\right) \circledS
\left( \mathfrak{g}_{2}\circledS \mathfrak{g}_{3}^{\ast }\right) =:\
^{1}TT^{\ast }G  \notag \\
&:&\left( V_{g},V_{\mu }\right) \rightarrow \left( g,\mu
,TR_{g^{-1}}V_{g},V_{\mu }-ad_{TR_{g^{-1}}V_{g}}^{\ast }\mu \right) ,
\label{trTT*G}
\end{eqnarray}%
where $\left( V_{g},V_{\mu }\right) \in T_{\left( g,\mu \right) }\left(
G\circledS \mathfrak{g}^{\ast }\right) $ \cite{EsGu14a}. The semidirect
product group multiplication on $^{1}TT^{\ast }G$ is
\begin{eqnarray}
&&\left( g,\mu ,\xi _{1},\nu _{1}\right) \left( h,\lambda ,\xi _{2},\nu
_{2}\right)  \notag \\
&=&\left( gh,\mu +Ad_{g^{-1}}^{\ast }\lambda ,\xi _{1}+Ad_{g^{-1}}\xi
_{2},\nu _{1}+Ad_{g^{-1}}^{\ast }\nu _{2}-ad_{Ad_{g^{-1}}\xi _{2}}^{\ast
}\mu \right)  \label{GrTT*G}
\end{eqnarray}%
and embedded subgroups of $^{1}TT^{\ast }G$ follows.

\begin{proposition}
The embeddings define subgroups%
\begin{eqnarray}
&&G,\mathfrak{g}_{1}^{\ast },\mathfrak{g}_{2},\mathfrak{g}_{3}^{\ast }\text{%
, }G\circledS \mathfrak{g}_{1}^{\ast },G\circledS \mathfrak{g}%
_{2},G\circledS \mathfrak{g}_{3}^{\ast },\text{ }\mathfrak{g}_{1}^{\ast
}\times \mathfrak{g}_{3}^{\ast }\text{, }\mathfrak{g}_{2}\times \mathfrak{g}%
_{3}^{\ast },  \notag \\
&&G\circledS \left( \mathfrak{g}_{1}^{\ast }\times \mathfrak{g}_{3}^{\ast
}\right) \text{, }G\circledS \left( \mathfrak{g}_{2}\times \mathfrak{g}%
_{3}^{\ast }\right) ,\text{ }\left( \mathfrak{g}_{1}^{\ast }\times \mathfrak{%
g}_{2}\right) \circledS \mathfrak{g}_{3}^{\ast }  \notag
\end{eqnarray}%
of$\ ^{1}TT^{\ast }G$ define subgroups. Group structures on $G\circledS
\mathfrak{g,}$ $G\circledS \mathfrak{g}^{\ast }$ are defined by Eqs.(\ref%
{tgtri}) and (\ref{rgc}), respectively and, group structures on $G\circledS
\left( \mathfrak{g}_{1}^{\ast }\times \mathfrak{g}_{3}^{\ast }\right) $, $%
G\circledS \left( \mathfrak{g}_{2}\times \mathfrak{g}_{3}^{\ast }\right) $
and $\left( \mathfrak{g}_{1}^{\ast }\times \mathfrak{g}_{2}\right) \circledS
\mathfrak{g}_{3}^{\ast }$ are defined (up to some ordering) by Eqs.(\ref%
{GrGg*g*}),(\ref{GrP}) and (\ref{Grgg*g*}), respectively. The group
structures on $\mathfrak{g}_{1}^{\ast },$ $\mathfrak{g}_{2},$ $\mathfrak{g}%
_{3}^{\ast }$, $\mathfrak{g}_{1}^{\ast }\times \mathfrak{g}_{3}^{\ast }$ and
$\mathfrak{g}_{2}\times \mathfrak{g}_{3}^{\ast }$ are vector additions.
\end{proposition}

\subsection{Hamiltonian Dynamics}

An element $\left( \xi _{2},\nu _{2},\xi _{3},\nu _{3}\right) $
in the semidirect product Lie algebra $\left( \mathfrak{g}\circledS \mathfrak{g}^{\ast }\right) \circledS
\left( \mathfrak{g}\circledS \mathfrak{g}^{\ast }\right) $ defines a right
invariant vector field on $\ ^{1}TT^{\ast }G$ by the tangent lift of right
translation in $^{1}TT^{\ast }G$. At a point $\left( g,\mu ,\xi ,\nu \right)
$, a right invariant vector is given by%
\begin{equation}
X_{\left( \xi _{2},\nu _{2},\xi _{3},\nu _{3}\right) }^{\ ^{1}TT^{\ast
}G}=\left( TR_{g}\xi _{2},\nu _{2}+ad_{\xi _{2}}^{\ast }\mu ,\xi _{3}+\left[
\xi ,\xi _{2}\right] _{\mathfrak{g}},\nu _{3}+ad_{\xi _{2}}^{\ast }\nu
-ad_{\xi }^{\ast }\nu _{2}\right) .  \label{RITT*G}
\end{equation}%
The bundle $T\left( G\circledS \mathfrak{g}^{\ast }\right) $ carries
Tulczyjew's symplectic two-form $\Omega _{T\left( G\circledS \mathfrak{g}%
^{\ast }\right) }$ with two potential one-forms. The one-forms $\theta _{1}$
and $\theta _{2}$ are obtained by taking derivations of the symplectic
two-form $\Omega _{G\circledS \mathfrak{g}^{\ast }}$ and the canonical
one-form $\theta _{G\circledS \mathfrak{g}^{\ast }}$ given in Eqs.(\ref%
{OhmT*G}), respectively \cite{EsGu14a}. By requiring the trivialization $%
tr_{TT^{\ast }G}^{1}$ in Eq.(\ref{trTT*G}) be a symplectic mapping, we
obtain an exact symplectic structure $\Omega _{^{1}TT^{\ast }G}$ with two
potential one-forms $\theta _{1}$ and $\theta _{2}$ taking the values%
\begin{eqnarray}
&&\left\langle \Omega _{\ ^{1}TT^{\ast }G};\left( X_{\left( \xi _{2},\nu
_{2},\xi _{3},\nu _{3}\right) }^{\ ^{1}TT^{\ast }G},X_{\left( \bar{\xi}_{2},%
\bar{\nu}_{2},\bar{\xi}_{3},\bar{\nu}_{3}\right) }^{\ ^{1}TT^{\ast
}G}\right) \right\rangle =\left\langle \nu _{3},\bar{\xi}_{2}\right\rangle
+\left\langle \nu _{2},\bar{\xi}_{3}\right\rangle -\left\langle \bar{\nu}%
_{2},\xi _{3}\right\rangle \notag \\
&&
-\left\langle \bar{\nu}_{3},\xi _{2}\right\rangle
+\left\langle \nu ,\left[ \xi _{2},\bar{\xi}_{2}\right] \right\rangle
+\left\langle \mu ,\left[ \xi _{3},\bar{\xi}_{2}\right] +\left[ \xi _{2},%
\bar{\xi}_{3}\right] +\left[ \xi ,\left[ \xi _{2},\bar{\xi}_{2}\right] %
\right] \right\rangle ,  \label{SymTT*G} \\
&&\left\langle \theta _{1},X_{\left( \xi _{2},\nu _{2},\xi _{3},\nu
_{3}\right) }^{\ ^{1}TT^{\ast }G}\right\rangle =\left\langle \nu ,\xi
_{2}\right\rangle -\left\langle \nu _{2},\xi \right\rangle +\left\langle \mu
,\left[ \xi ,\xi _{2}\right]\right\rangle ,  \label{1} \\
&&\left\langle \theta _{2},X_{\left( \xi _{2},\nu _{2},\xi _{3},\nu
_{3}\right) }^{\ ^{1}TT^{\ast }G}\right\rangle =\left\langle \mu ,\xi
_{3}\right\rangle +\left\langle \nu ,\xi _{2}\right\rangle +\left\langle \mu
,\left[ \xi ,\xi _{2}\right]\right\rangle ,  \label{2}
\end{eqnarray}%
on the right invariant vector fields as in Eq.(\ref{RITT*G}). At a point $%
\left( g,\mu ,\xi ,\nu \right) \in \ ^{1}TT^{\ast }G$, the musical
isomorphism $\Omega _{\ ^{1}TT^{\ast }G}^{\flat }$, induced from $\Omega _{\
^{1}TT^{\ast }G}$, maps the image of a right invariant vector field $%
X_{\left( \xi _{2},\nu _{2},\xi _{3},\nu _{3}\right) }^{\ ^{1}TT^{\ast }G}$
to an element
\begin{equation*}
\Omega _{\ ^{1}TT^{\ast }G}^{\flat }( X_{\left( \xi _{2},\nu _{2},\xi
_{3},\nu _{3}\right) }^{\ ^{1}TT^{\ast }G}) =( T_{g}^{\ast
}R_{g^{-1}}\left( \nu _{3}-ad_{\xi }^{\ast }\nu _{2}\right) ,-( \xi
_{3}+\left[ \xi ,\xi _{2}\right]) ,\nu _{2}+ad_{\xi
_{2}}^{\ast }\mu ,-\xi _{2})
\end{equation*}%
of $T_{\left( g,\mu ,\xi ,\nu \right) }^{\ast }\left( \ ^{1}TT^{\ast
}G\right) .$

\begin{proposition}
Given a Hamiltonian function $E$ on $^{1}TT^{\ast }G$, Hamilton's equation
\begin{equation*}
i_{X_{E}^{\ ^{1}TT^{\ast }G}}\Omega _{\ ^{1}TT^{\ast }G}=-DE
\end{equation*}%
defines a Hamiltonian vector field $X_{E}^{\ ^{1}TT^{\ast }G}$ which is a
right invariant vector field generated by the element
\begin{equation*}
\left( \frac{\delta E}{\delta \nu },-\left( \frac{\delta E}{\delta \xi }+ad_{%
\frac{\delta E}{\delta \nu }}^{\ast }\mu \right) ,\frac{\delta E}{\delta \mu
}-ad_{\xi }\frac{\delta E}{\delta \nu },-\left( T^{\ast }R_{g}\frac{\delta E%
}{\delta g}+ad_{\xi }^{\ast }\frac{\delta E}{\delta \xi }+ad_{\xi }^{\ast
}ad_{\frac{\delta E}{\delta \nu }}^{\ast }\mu \right) \right)
\end{equation*}%
of the Lie algebra $\left( \mathfrak{g}\circledS \mathfrak{g}^{\ast }\right)
\circledS \left( \mathfrak{g}\circledS \mathfrak{g}^{\ast }\right) .$
Components of $X_{E}^{\ ^{1}TT^{\ast }G}$ define the Hamilton's equations
\begin{equation}
\dot{g}=TR_{g}\left( \frac{\delta E}{\delta \nu }\right) ,\text{ \ \ }\dot{%
\mu}=-\frac{\delta E}{\delta \xi },\text{ \ \ }\dot{\xi}=\frac{\delta E}{%
\delta \mu },\text{ \ \ }\dot{\nu}=ad_{\frac{\delta E}{\delta \nu }}^{\ast
}\nu -T^{\ast }R_{g}\left( \frac{\delta E}{\delta g}\right) .
\label{HamTT*G}
\end{equation}
\end{proposition}

\subsubsection{Reduction by $G$}

The action of the group $G$ on $^{1}TT^{\ast }G$, given by
\begin{equation}
\left( g;\left( h,\lambda ,\xi ,\nu \right) \right) \rightarrow \left(
gh,Ad_{g^{-1}}^{\ast }\lambda ,Ad_{g^{-1}}\xi ,Ad_{g^{-1}}^{\ast }\nu
\right) ,  \label{GonTT*G}
\end{equation}%
is a symplectic action.

\begin{proposition}
The Poisson reduction of $^{1}TT^{\ast }G$ under the action of $G$ results
in the total space $\left( \mathfrak{g}_{1}^{\ast }\times \mathfrak{g}%
_{2}\right) \circledS \mathfrak{g}_{3}^{\ast }$ endowed with the Poisson
bracket%
\begin{equation}
\left\{ E,F\right\} _{\left( \mathfrak{g}_{1}^{\ast }\times \mathfrak{g}%
_{2}\right) \circledS \mathfrak{g}_{3}^{\ast }}\left( \mu ,\xi ,\nu \right)
=\left\langle \frac{\delta F}{\delta \xi },\frac{\delta E}{\delta \mu }%
\right\rangle -\left\langle \frac{\delta E}{\delta \xi },\frac{\delta F}{%
\delta \mu }\right\rangle +\left\langle \nu ,\left[ \frac{\delta E}{\delta
\nu },\frac{\delta F}{\delta \nu }\right] \right\rangle .  \label{Poig*gg*}
\end{equation}
\end{proposition}

\begin{remark}
Here, the Poisson bracket on $\left( \mathfrak{g}_{1}^{\ast }\times
\mathfrak{g}_{2}\right) \circledS \mathfrak{g}_{3}^{\ast }$ is the direct
product of canonical Poisson bracket on $\mathfrak{g}_{1}^{\ast }\times
\mathfrak{g}_{2}$ and Lie-Poisson bracket on $\mathfrak{g}_{3}^{\ast }$
whereas in Eq.(\ref{Poigxg*xg*}) we obtained a Poisson bracket, on an
isomorphic space $\mathfrak{g}\circledS \left( \mathfrak{g}^{\ast }\times
\mathfrak{g}^{\ast }\right) $, which is not in the form of a direct product
form.
\end{remark}

The action in Eq.(\ref{GonTT*G}) is Hamiltonian with the momentum mapping
\begin{equation}
\mathbf{J}_{\text{ }^{1}TT^{\ast }G}^{G}:\text{ }^{1}TT^{\ast }G\rightarrow
\mathfrak{g}^{\ast }:\left( g,\mu ,\xi ,\nu \right) \rightarrow \nu +ad_{\xi
}^{\ast }\mu .  \label{MGonTT*G}
\end{equation}%
The quotient space of the preimage $\mathbf{J}_{\text{ }^{1}TT^{\ast
}G}^{-1}\left( \lambda \right) $ of an element $\lambda \in \mathfrak{g}%
^{\ast }$ under the action of isotropy subgroup $G_{\lambda }$ is
\begin{equation*}
\left. \mathbf{J}_{\text{ }^{1}TT^{\ast }G}^{-1}\left( \lambda \right)
\right/ G_{\lambda }\simeq \mathcal{O}_{\lambda }\times \mathfrak{g}^{\ast
}\times \mathfrak{g}\text{. }
\end{equation*}%
Pushing forward a right invariant vector field $X_{\left( \eta ,\upsilon
,\zeta ,\tilde{\upsilon}\right) }^{\ ^{1}TT^{\ast }G}$ in the form of Eq.(%
\ref{RITT*G}) by the symplectic projection $^{1}TT^{\ast }G\rightarrow
\mathcal{O}_{\lambda }\times \mathfrak{g}^{\ast }\times \mathfrak{g}$, we
arrive at the vector field
\begin{equation}
X_{\left( \eta ,\upsilon ,\zeta \right) }^{\mathcal{O}_{\lambda }\times
\mathfrak{g}^{\ast }\times \mathfrak{g}}\left( Ad_{g^{-1}}^{\ast }\lambda
,\mu ,\xi \right) =\left( ad_{\eta }^{\ast }\circ Ad_{g^{-1}}^{\ast }\lambda
,\upsilon +ad_{\eta }^{\ast }\mu ,\zeta +\left[ \xi ,\eta \right] \right)
\label{VfO}
\end{equation}%
on the quotient space $\mathcal{O}_{\lambda }\times \mathfrak{g}^{\ast
}\times \mathfrak{g}$. We refer \cite{EsGu14b} for the proof of the
following proposition and other details.

\begin{proposition}
The reduced Tulczyjew's space $\mathcal{O}_{\lambda }\times \mathfrak{g}%
^{\ast }\times \mathfrak{g}$ has an exact symplectic two-form $\Omega _{%
\mathcal{O}_{\lambda }\times \mathfrak{g}^{\ast }\times \mathfrak{g}}$ with
two potential one-forms $\chi _{1}$ and $\chi _{2}$ whose values on vector
fields of the form of Eq.(\ref{RITT*G}) at the point $\left(
Ad_{g^{-1}}^{\ast }\lambda ,\mu ,\xi \right) $ are
\begin{eqnarray}
\left\langle \Omega _{\mathcal{O}_{\lambda }\times \mathfrak{g}^{\ast
}\times \mathfrak{g}},\left( X_{\left( \eta ,\upsilon ,\zeta \right) }^{%
\mathcal{O}_{\lambda }\times \mathfrak{g}^{\ast }\times \mathfrak{g}%
},X_{\left( \bar{\eta},\bar{\upsilon},\bar{\zeta}\right) }^{\mathcal{O}%
_{\lambda }\times \mathfrak{g}^{\ast }\times \mathfrak{g}}\right)
\right\rangle &=&\left\langle \upsilon ,\bar{\zeta}\right\rangle
-\left\langle \bar{\upsilon},\zeta \right\rangle -\left\langle \lambda
,[\eta ,\bar{\eta}]\right\rangle, \\
\left\langle \chi _{1},X_{\left( \eta ,\upsilon ,\zeta \right) }^{\mathcal{O}%
_{\lambda }\times \mathfrak{g}^{\ast }\times \mathfrak{g}}\right\rangle
\left( Ad_{g^{-1}}^{\ast }\lambda ,\mu ,\xi \right) &=&\left\langle \lambda
,\eta \right\rangle -\left\langle \upsilon ,\xi \right\rangle, \\
\left\langle \chi _{2},X_{\left( \eta ,\upsilon ,\zeta \right) }^{\mathcal{O}%
_{\lambda }\times \mathfrak{g}^{\ast }\times \mathfrak{g}}\right\rangle
\left( Ad_{g^{-1}}^{\ast }\lambda ,\mu ,\xi \right) &=&\left\langle \lambda
,\eta \right\rangle +\left\langle \mu ,\zeta \right\rangle .
\end{eqnarray}
\end{proposition}

\subsubsection{Reduction by $\mathfrak{g}$}

\begin{proposition}
The action of $\mathfrak{g}_{2}$ on $^{1}TT^{\ast }G$ given by
\begin{equation}
\varphi _{\eta }:\ ^{1}TT^{\ast }G\rightarrow \ ^{1}TT^{\ast }G:\left( \eta
;\left( g,\mu ,\xi ,\nu \right) \right) \rightarrow \left( g,\mu ,\xi +\eta
,\nu \right) ,  \label{gonTT*G}
\end{equation}%
is symplectic for some $\eta \in \mathfrak{g}_{2}$.
\end{proposition}

\begin{proof}
Push forward of a vector field $X_{\left( \xi _{2},\nu _{2},\xi _{3},\nu
_{3}\right) }^{\ ^{1}TT^{\ast }G}$ in the form of Eq.(\ref{RITT*G}) by the
transformation $\varphi _{\eta }$ is also a right invariant vector field
\begin{equation*}
\left( \varphi _{\eta }\right) _{\ast }X_{\left( \xi _{2},\nu _{2},\xi
_{3},\nu _{3}\right) }^{\ ^{1}TT^{\ast }G}=X_{\left( \xi _{2},\nu _{2},\xi
_{3}-[\eta ,\xi _{2}],\nu _{3}+ad_{\eta }^{\ast }\nu _{2}\right) }^{\
^{1}TT^{\ast }G}.
\end{equation*}
Establishing the identity
\begin{equation}
\varphi _{\eta }^{\ast }\Omega _{\ ^{1}TT^{\ast }G}\left( X,Y\right) \left(
g,\mu ,\xi ,\nu \right) =\Omega _{\ ^{1}TT^{\ast }G}\left( \left( \varphi
_{\eta }\right) _{\ast }X,\left( \varphi _{\eta }\right) _{\ast }Y\right)
\left( g,\mu ,\xi +\eta ,\nu \right)   \label{gonTT*Gsym}
\end{equation}%
requires direct calculation and gives the desired result that the action (%
\ref{gonTT*G}) is symplectic. In Eq.(\ref{gonTT*Gsym}) $X$ and $Y$ are right
invariant vector fields as in Eq.(\ref{RITT*G}) and $\Omega _{\ ^{1}TT^{\ast
}G}$ is the symplectic two-form on $\Omega _{\ ^{1}TT^{\ast }G}$ given in
Eq.(\ref{SymTT*G}).
\end{proof}

\begin{proposition}
The Poisson reduction of $^{1}TT^{\ast }G$ under the action of $\mathfrak{g}%
_{2}$ in Eq.(\ref{gonTT*G}) results in $G\circledS \left( \mathfrak{g}%
_{1}^{\ast }\times \mathfrak{g}_{3}^{\ast }\right) $ endowed with the
bracket
\begin{equation*}
\left\{ E,F\right\} _{G\circledS \left( \mathfrak{g}_{1}^{\ast }\times
\mathfrak{g}_{3}^{\ast }\right) }\left( g,\mu ,\nu \right) =\left\langle
T_{e}^{\ast }R_{g}\frac{\delta F}{\delta g},\frac{\delta E}{\delta \nu }%
\right\rangle -\left\langle T_{e}^{\ast }R_{g}\frac{\delta E}{\delta g},%
\frac{\delta F}{\delta \nu }\right\rangle +\left\langle \nu ,\left[ \frac{%
\delta E}{\delta \nu },\frac{\delta F}{\delta \nu }\right] \right\rangle .
\end{equation*}
\end{proposition}

\begin{remark}
The Poisson bracket $\left\{ E,F\right\} _{G\circledS \left( \mathfrak{g}%
_{1}^{\ast }\times \mathfrak{g}_{3}^{\ast }\right) }$ is independent of the
derivatives of the functions with respect to $\mu $, that is, it does not
involve $\delta E/\delta \mu $ and $\delta F/\delta \mu $. Its structure
resembles the canonical Poisson bracket $\left\{ \text{ },\text{ }\right\}
_{G\circledS \mathfrak{g}^{\ast }}$ in Eq.(\ref{PoissonGg*}) on $G\circledS
\mathfrak{g}^{\ast }$. We recall that, on $G\circledS \left( \mathfrak{g}%
^{\ast }\times \mathfrak{g}^{\ast }\right) $, we have derived the Poisson
bracket $\left\{ \text{ },\text{ }\right\} _{G\circledS \left( \mathfrak{g}%
_{2}^{\ast }\times \mathfrak{g}_{3}^{\ast }\right) }$ in Eq.(\ref{PoGxg*xg*}%
) involving $\delta E/\delta \mu ,$ $\delta F/\delta \mu $, $\delta E/\delta
\nu $ and $\delta F/\delta \nu $.
\end{remark}

The infinitesimal generator $X_{\left( 0,0,\xi _{3},0\right) }^{\
^{1}TT^{\ast }G}$ of the action in Eq.(\ref{gonTT*G}) corresponds to the
element $\xi _{3}\in \mathfrak{g}$ and is a right invariant vector field.
Since the action is Hamiltonian, and the symplectic two-form is exact, we
can derive the associated momentum map $\mathbf{J}_{\text{ }^{1}TT^{\ast
}G}^{\mathfrak{g}_{2}}$ by the equation
\begin{equation*}
\left\langle \mathbf{J}_{\text{ }^{1}TT^{\ast }G}^{\mathfrak{g}_{2}}\left(
g,\mu ,\xi ,\nu \right) ,\xi _{3}\right\rangle =\left\langle \theta
_{2},X_{\left( 0,0,\xi _{3},0\right) }^{\ ^{1}TT^{\ast }G}\right\rangle
=\left\langle \mu ,\xi _{3}\right\rangle ,
\end{equation*}%
where $\theta _{2}$ is the potential one-form in Eq.(\ref{2}) satisfying $%
d\theta _{2}=\Omega _{\ ^{1}TT^{\ast }G}$. We find that
\begin{equation}
\mathbf{J}_{\text{ }^{1}TT^{\ast }G}^{\mathfrak{g}_{2}}:\ ^{1}TT^{\ast
}G\rightarrow Lie^{\ast }\left( \mathfrak{g}_{2}\right) =\mathfrak{g}^{\ast
}:\left( g,\mu ,\xi ,\nu \right) \rightarrow \mu  \label{MgonTT*G}
\end{equation}%
is the projection to the second entry in $^{1}TT^{\ast }G$. The preimage of
an element $\mu \in \mathfrak{g}^{\ast }$ by $\mathbf{J}_{\text{ }%
^{1}TT^{\ast }G}^{\mathfrak{g}_{2}}$ is the space $G\circledS \left(
\mathfrak{g}_{2}\circledS \mathfrak{g}_{3}^{\ast }\right) $. Following
proposition describes the symplectic reduction of $^{1}TT^{\ast }G$ with the
action of $\mathfrak{g}_{2}$.

\begin{proposition}
The symplectic reduction of $^{1}TT^{\ast }G$ under the action of $\mathfrak{%
g}_{2}$ defined in Eq.(\ref{gonTT*G}) gives the reduced space
\begin{equation*}
\left. \left( \mathbf{J}_{\text{ }^{1}TT^{\ast }G}^{\mathfrak{g}_{2}}\right)
^{-1}\left( \mu \right) \right/ \mathfrak{g}_{2}\simeq G\circledS \mathfrak{g%
}_{3}^{\ast }
\end{equation*}%
with the canonical symplectic two-from $\Omega _{G\circledS \mathfrak{g}%
_{3}^{\ast }}$ as in Eq.(\ref{Ohm2T*G}).
\end{proposition}

\begin{remark}
Existence of the symplectic action of $\mathfrak{g}_{2}$ on $^{1}TT^{\ast }G$
is directly related to existence of the symplectic diffeomorphism
\begin{equation*}
^{1}\bar{\sigma}_{G}:\text{ }^{1}TT^{\ast }G\rightarrow \text{ }^{1}T^{\ast
}TG:\left( g,\mu ,\xi ,\nu \right) \rightarrow \left( g,\xi ,\nu +ad_{\xi
}^{\ast }\mu ,\mu \right) .
\end{equation*}%
For the mapping $^{1}\bar{\sigma}_{G}$, we refer \cite{EsGu14a}.
\end{remark}

\subsubsection{Reduction by $\mathfrak{g}^{\ast }$}

Induced from the group operation on $^{1}TT^{\ast }G$, there are two
canonical actions of $\mathfrak{g}^{\ast }$ on $^{1}TT^{\ast }G$ given by
\begin{eqnarray}
\psi &:&\mathfrak{g}_{1}^{\ast }\times \ ^{1}TT^{\ast }G\rightarrow \
^{1}TT^{\ast }G:\left( \lambda ;\left( g,\mu ,\xi ,\nu \right) \right)
\rightarrow \left( g,\mu +\lambda ,\xi ,\nu \right) ,  \label{psi} \\
\phi &:&\mathfrak{g}_{3}^{\ast }\times \ ^{1}TT^{\ast }G\rightarrow \
^{1}TT^{\ast }G:\left( \lambda ;\left( g,\mu ,\xi ,\nu \right) \right)
\rightarrow \left( g,\mu ,\xi ,\nu +\lambda \right) .  \label{phi}
\end{eqnarray}

\begin{proposition}
$\psi $ is a symplectic action whereas $\phi $ is not.
\end{proposition}

\begin{proof}
Pushing forward of a vector field $X_{\left( \xi _{2},\nu _{2},\xi _{3},\nu
_{3}\right) }^{\ ^{1}TT^{\ast }G}$ in form of Eq.(\ref{RITT*G}) by
transformations $\psi _{\lambda }$ and $\phi _{\lambda }$ results in right
invariant vector fields
\begin{eqnarray*}
\left( \psi _{\lambda }\right) _{\ast }X_{\left( \xi _{2},\nu _{2},\xi
_{3},\nu _{3}\right) }^{\ ^{1}TT^{\ast }G} &=&X_{\left( \xi _{2},\nu
_{2}-ad_{\xi _{2}}^{\ast }\lambda ,\xi _{3},\nu _{3}-ad_{\xi }^{\ast
}ad_{\xi _{2}}^{\ast }\lambda \right) }^{\ ^{1}TT^{\ast }G}, \\
\left( \phi _{\lambda }\right) _{\ast }X_{\left( \xi _{2},\nu _{2},\xi
_{3},\nu _{3}\right) }^{\ ^{1}TT^{\ast }G} &=&X_{\left( \xi _{2},\nu
_{2},\xi _{3},\nu _{3}-ad_{\xi _{2}}^{\ast }\lambda \right) }^{\
^{1}TT^{\ast }G}.
\end{eqnarray*}%
If $\Omega _{\ ^{1}TT^{\ast }G}$ is the symplectic two-form on $^{1}TT^{\ast
}G$ given in Eq.(\ref{SymTT*G}), the identity
\begin{equation}
\psi _{\lambda }{}^{\ast }\Omega _{\ ^{1}TT^{\ast }G}\left( X,Y\right)
\left( g,\mu ,\xi ,\nu \right) =\Omega _{\ ^{1}TT^{\ast }G}\left( \left(
\psi _{\lambda }\right) _{\ast }X,\left( \psi _{\lambda }\right) _{\ast
}Y\right) \left( g,\mu +\lambda ,\xi ,\nu \right)
\end{equation}%
holds for all vector fields $X$ and $Y$, and $\lambda \in \mathfrak{g}^{\ast
}$ whereas
\begin{equation}
\phi _{\lambda }{}^{\ast }\Omega _{\ ^{1}TT^{\ast }G}\left( X,Y\right)
\left( g,\mu ,\xi ,\nu \right) =\Omega _{\ ^{1}TT^{\ast }G}\left( \left(
\phi _{\lambda }\right) _{\ast }X,\left( \phi _{\lambda }\right) _{\ast
}Y\right) \left( g,\mu ,\xi ,\nu +\lambda \right)
\end{equation}%
does not necessarily hold. Hence, $\psi _{\lambda }$ is a symplectic action
but not $\phi _{\lambda }$.
\end{proof}

Thus, Poisson and symplectic reductions of $^{1}TT^{\ast }G$ under the
action of $\mathfrak{g}^{\ast }$ is possible only for $\psi $.

\begin{proposition}
Poisson reduction of $^{1}TT^{\ast }G$ under the action $\psi $ of $%
\mathfrak{g}_{1}^{\ast }$ results in $G\circledS \left( \mathfrak{g}%
_{2}\times \mathfrak{g}_{3}^{\ast }\right) $ endowed with the bracket
\begin{equation*}
\left\{ E,F\right\} _{G\circledS \left( \mathfrak{g}_{2}\times \mathfrak{g}%
_{3}^{\ast }\right) }\left( g,\xi ,\nu \right) =\left\langle T_{e}^{\ast
}R_{g}\frac{\delta F}{\delta g},\frac{\delta E}{\delta \nu }\right\rangle
-\left\langle T_{e}^{\ast }R_{g}\frac{\delta E}{\delta g},\frac{\delta F}{%
\delta \nu }\right\rangle +\left\langle \nu ,\left[ \frac{\delta E}{\delta
\nu },\frac{\delta F}{\delta \nu }\right] \right\rangle .
\end{equation*}
\end{proposition}

\begin{remark}
The Poisson bracket $\left\{ E,F\right\} _{G\circledS \left( \mathfrak{g}%
_{2}\times \mathfrak{g}_{3}^{\ast }\right) }$ is independent of derivatives
of functions with respect to $\xi $ and it resembles to the canonical
Poisson bracket $\left\{ \text{ },\text{ }\right\} _{G\circledS \mathfrak{g}%
^{\ast }}$ in Eq.(\ref{PoissonGg*}) on $G\circledS \mathfrak{g}^{\ast }$.
The space $G\circledS \left( \mathfrak{g}^{\ast }\times \mathfrak{g}\right) $
is isomorphic to $G\circledS \left( \mathfrak{g}_{2}\times \mathfrak{g}%
_{3}^{\ast }\right) $ on which we derived the Poisson bracket $\left\{ \text{
},\text{ }\right\} _{G\circledS \left( \mathfrak{g}^{\ast }\times \mathfrak{g%
}\right) }$ in Eq.(\ref{PoGgg*}) involving derivatives with respect to both
of $\xi $ and $\nu $.
\end{remark}

The infinitesimal generators $X_{\left( 0,\nu _{2},0,0\right) }^{\
^{1}TT^{\ast }G}$ of the action are defined by $\nu _{2}\in Lie\left(
\mathfrak{g}_{1}^{\ast }\right) $. We compute the associated momentum map
from the equation
\begin{equation*}
\left\langle \mathbf{J}_{\text{ }^{1}TT^{\ast }G}^{\mathfrak{g}_{1}^{\ast
}}\left( g,\mu ,\xi ,\nu \right) ,\nu _{2}\right\rangle =\left\langle \theta
_{1},X_{\left( 0,\nu _{2},0,0\right) }^{\ ^{1}TT^{\ast }G}\right\rangle
=-\left\langle \nu _{2},\xi \right\rangle ,
\end{equation*}%
where $\theta _{1}$ is the potential one-form in Eq.(\ref{1}). We find that
\begin{equation}
\mathbf{J}_{\text{ }^{1}TT^{\ast }G}^{\mathfrak{g}_{1}^{\ast }}:\
^{1}TT^{\ast }G\rightarrow Lie^{\ast }\left( \mathfrak{g}_{1}^{\ast }\right)
\simeq \mathfrak{g}:\left( g,\mu ,\xi ,\nu \right) \rightarrow -\xi
\label{Mg*onTT*G}
\end{equation}%
is minus the projection to third factor in $^{1}TT^{\ast }G$. The preimage
of an element $\xi \in \mathfrak{g}$ is the space $G\circledS \left(
\mathfrak{g}_{1}^{\ast }\times \mathfrak{g}_{3}^{\ast }\right) $.

\begin{proposition}
The symplectic reduction of $^{1}TT^{\ast }G$ under the action of $\mathfrak{%
g}^{\ast }$ defined in Eq.(\ref{gonTT*G}) results in the reduced space
\begin{equation*}
\left. \left( \mathbf{J}_{\text{ }^{1}TT^{\ast }G}^{\mathfrak{g}_{1}^{\ast
}}\right) ^{-1}\left( \xi \right) \right/ \mathfrak{g}_{1}^{\ast }\simeq
\left. G\circledS \left( \mathfrak{g}_{1}^{\ast }\times \mathfrak{g}%
_{3}^{\ast }\right) \right/ \mathfrak{g}_{1}^{\ast }\simeq G\circledS
\mathfrak{g}_{3}^{\ast }
\end{equation*}%
with the canonical symplectic two-from $\Omega _{G\circledS \mathfrak{g}%
_{3}^{\ast }}$ as given in Eq.(\ref{Ohm2T*G}).
\end{proposition}

\begin{remark}
Existence of the symplectic action of $\mathfrak{g}^{\ast }$ on $%
^{1}TT^{\ast }G$ is directly related to existence of the symplectic
diffeomorphism
\begin{equation}
^{1}\Omega _{G\circledS \mathfrak{g}^{\ast }}^{\flat }:\text{ }^{1}TT^{\ast
}G\rightarrow \text{ }^{1}T^{\ast }T^{\ast }G:\left( g,\mu ,\xi ,\nu \right)
\rightarrow \left( g,\mu ,\nu +ad_{\xi }^{\ast }\mu ,-\xi \right) .
\end{equation}%
For the mapping $^{1}\Omega _{G\circledS \mathfrak{g}^{\ast }}^{\flat }$, we
refer \cite{EsGu14a}.
\end{remark}

In the following proposition, we discuss the actions $\psi $ and $\phi $ of $%
\mathfrak{g}^{\ast }$ on $^{1}TT^{\ast }G$ in Eqs(\ref{psi}) and (\ref{phi})
from a different point of view.

\begin{proposition}
The mappings
\begin{eqnarray}
Emb_{1} &:&G\circledS \mathfrak{g}^{\ast }\hookrightarrow \ ^{1}TT^{\ast
}G:\left( g,\mu \right) \rightarrow \left( g,\mu ,0,0\right)  \notag \\
Emb_{2} &:&G\circledS \mathfrak{g}^{\ast }\hookrightarrow \ ^{1}TT^{\ast
}G:\left( g,\nu \right) \rightarrow \left( g,0,0,\nu \right)  \label{Emb2}
\end{eqnarray}%
define a Lagrangian and a symplectic, respectively, embeddings of $%
G\circledS \mathfrak{g}^{\ast }$ into $^{1}TT^{\ast }G$.
\end{proposition}

\begin{proof}
The first embedding is Lagrangian because it is the zero section of the
fibration $\ ^{1}TT^{\ast }G\rightarrow G\circledS \mathfrak{g}_{1}^{\ast }$%
. The second one is  symplectic because the pull-back of $\Omega _{\
^{1}TT^{\ast }G}$ to $G\circledS \mathfrak{g}^{\ast }$ by $Emb_{2}$ results
in the symplectic two-form $\Omega _{G\circledS \mathfrak{g}^{\ast }}$ in
Eq.(\ref{OhmT*G}). On the image of $Emb_{2}$, the Hamilton's equations (\ref%
{HamTT*G}) reduce to the trivialized Hamilton's equations (\ref{ULP}) on $%
G\circledS \mathfrak{g}^{\ast }$. Consequently, the embedding $\mathfrak{g}%
_{3}^{\ast }\rightarrow \ ^{1}TT^{\ast }G$ is a Poisson map. When $E=h\left(
\nu \right) $ the Hamilton's equations (\ref{HamTT*G}) reduce to the
Lie-Poisson equations (\ref{LP}).
\end{proof}

\subsubsection{Reduction by $G\circledS \mathfrak{g}$}

The action of $G\circledS \mathfrak{g}$ on $^{1}TT^{\ast }G$
\begin{eqnarray}
\vartheta &:&\left( G\circledS \mathfrak{g}_{2}\right) \times \text{ }%
^{1}TT^{\ast }G\rightarrow \text{ }^{1}TT^{\ast }G:\left( \left( h,\eta
\right) ,\left( g,\mu ,\xi ,\nu \right) \right) \rightarrow \vartheta
_{\left( h,\eta \right) }\left( g,\mu ,\xi ,\nu \right)  \notag \\
&:&\left( \left( h,\eta \right) ,\left( g,\mu ,\xi ,\nu \right) \right)
\rightarrow \left( hg,Ad_{h^{-1}}^{\ast }\mu ,\eta +Ad_{h^{-1}}\xi
,Ad_{h^{-1}}^{\ast }\nu \right)  \label{varpsi}
\end{eqnarray}
can be described as a composition
\begin{equation*}
\vartheta _{\left( h,\eta \right) }=\vartheta _{\left( h,0\right) }\circ
\vartheta _{\left( e,Ad_{g}\eta \right) },
\end{equation*}%
where $\vartheta _{\left( h,0\right) }$ and $\vartheta _{\left( e,Ad_{g}\eta
\right) }$ can be identified with the actions of $G$ and $\mathfrak{g}$ on $%
^{1}TT^{\ast }G$ given in Eqs.(\ref{GonTT*G}) and (\ref{gonTT*G}),
respectively. Since both of these are symplectic, the action $\vartheta $ of
$G\circledS \mathfrak{g}$ on $^{1}TT^{\ast }G$ is symplectic.

\begin{proposition}
The Poisson reduction of $^{1}TT^{\ast }G$ under the action of $G\circledS
\mathfrak{g}_{2}$ in Eq.(\ref{varpsi})\ results in $\mathfrak{g}_{1}^{\ast
}\times \mathfrak{g}_{3}^{\ast }$ endowed with the bracket
\begin{equation}
\left\{ E,F\right\} _{\mathfrak{g}_{1}^{\ast }\times \mathfrak{g}_{3}^{\ast
}}\left( \mu ,\nu \right) =\left\langle \nu ,\left[ \frac{\delta E}{\delta
\nu },\frac{\delta F}{\delta \nu }\right] \right\rangle .\text{\ }
\label{Poig*g*}
\end{equation}
\end{proposition}

\begin{remark}
Although the Poisson bracket (\ref{Poig*g*}) structurally resembles the
Lie-Poisson bracket on $\mathfrak{g}_{3}^{\ast }$, it is not a Lie-Poisson
bracket on $\mathfrak{g}_{1}^{\ast }\times \mathfrak{g}_{3}^{\ast }$
considered as dual of Lie algebra $\mathfrak{g}\circledS \mathfrak{g}$ of
the group $G\circledS \mathfrak{g}$. We refer to the Poisson bracket in Eq.(%
\ref{LPBg*xg*}) for the Lie-Poisson structure on $\mathfrak{g}^{\ast }\times
\mathfrak{g}^{\ast }$.
\end{remark}

Right invariant vector field generating the action is associted to two
tuples $\left( \xi _{2},\xi _{3}\right) $ in the Lie algebra of $G\circledS
\mathfrak{g}_{2}$, and is given by%
\begin{equation}
X_{\left( \xi _{2},0,\xi _{3},0\right) }^{\ ^{1}TT^{\ast }G}=\left(
TR_{g}\xi _{2},ad_{\xi _{2}}^{\ast }\mu ,\xi _{3}+\left[ \xi ,\xi _{2}\right]
,ad_{\xi _{2}}^{\ast }\nu _{2}\right) .
\end{equation}%
The momentum mapping for this Hamiltonian action is defined by the equation
\begin{equation*}
\left\langle \mathbf{J}_{\text{ }^{1}TT^{\ast }G}^{G\circledS \mathfrak{g}%
_{2}}\left( g,\mu ,\xi ,\nu \right) ,\left( \xi _{2},\xi _{3}\right)
\right\rangle =\left\langle \theta _{2},X_{\left( \xi _{2},0,\xi
_{3},0\right) }^{\ ^{1}TT^{\ast }G}\right\rangle =\left\langle \mu ,\xi
_{3}\right\rangle +\left\langle \nu +ad_{\xi }^{\ast }\mu ,\xi
_{2}\right\rangle ,
\end{equation*}%
where $\theta _{2}$, in Eq.(\ref{2}), is the potential one-form on $%
^{1}TT^{\ast }G$. We find
\begin{equation*}
\mathbf{J}_{\text{ }^{1}TT^{\ast }G}^{G\circledS \mathfrak{g}_{2}}:\
^{1}TT^{\ast }G\rightarrow Lie^{\ast }\left( G\circledS \mathfrak{g}%
_{2}\right) =\mathfrak{g}^{\ast }\times \mathfrak{g}^{\ast }:\left( g,\mu
,\xi ,\nu \right) =\left( \nu +ad_{\xi }^{\ast }\mu ,\mu \right) .
\end{equation*}%
Note that, we have the following relation
\begin{equation*}
\mathbf{J}_{\text{ }^{1}TT^{\ast }G}^{G\circledS \mathfrak{g}_{2}}\left(
g,\mu ,\xi ,\nu \right) =\left( \mathbf{J}_{\text{ }^{1}TT^{\ast
}G}^{G}\left( g,\mu ,\xi ,\nu \right) ,\mathbf{J}_{\text{ }^{1}TT^{\ast }G}^{%
\mathfrak{g}_{2}}\left( g,\mu ,\xi ,\nu \right) \right)
\end{equation*}%
for momentum mappings in Eqs.(\ref{MGonTT*G}) and (\ref{MgonTT*G}) for the
actions of $G$ and $\mathfrak{g}_{2}$ on $^{1}TT^{\ast }G$. The preimage of
an element $\left( \lambda ,\mu \right) \in \mathfrak{g}^{\ast }\times
\mathfrak{g}^{\ast }$ is
\begin{equation*}
\left( \mathbf{J}_{\text{ }^{1}TT^{\ast }G}^{G\circledS \mathfrak{g}%
_{2}}\right) ^{-1}\left( \lambda ,\mu \right) =\left\{ \left( g,\mu ,\xi
,\nu \right) :\nu =\lambda -ad_{\xi }^{\ast }\mu \text{ for fixed }\mu \text{
and }\lambda \right\}
\end{equation*}%
which we may identify with the semidirect product $G\circledS \mathfrak{g}%
_{2}$. Recall the coadjoint action $Ad_{\left( g,\xi \right) }^{\ast }$, in
Eq.(\ref{coad2}), of the group $G\circledS \mathfrak{g}_{2}$ on the dual $%
\mathfrak{g}^{\ast }\times \mathfrak{g}^{\ast }$ of its Lie algebra. The
isotropy subgroup $\left( G\circledS \mathfrak{g}_{2}\right) _{\left(
\lambda ,\mu \right) }$ of this coadjoint action is
\begin{equation*}
\left( G\circledS \mathfrak{g}_{2}\right) _{\left( \lambda ,\mu \right)
}=\left\{ \left( g,\xi \right) \in G\circledS \mathfrak{g}_{2}:Ad_{\left(
g,\xi \right) }^{\ast }\left( \lambda ,\mu \right) =\left( \lambda ,\mu
\right) \right\}
\end{equation*}%
and acts on the preimage $\left( \mathbf{J}_{\text{ }^{1}TT^{\ast
}G}^{G\circledS \mathfrak{g}_{2}}\right) ^{-1}\left( \lambda ,\mu \right) $.
A generic quotient space
\begin{equation*}
\left. \left( \mathbf{J}_{\text{ }^{1}TT^{\ast }G}^{G\circledS \mathfrak{g}%
_{2}}\right) ^{-1}\left( \lambda ,\mu \right) \right/ \left( G\circledS
\mathfrak{g}_{2}\right) _{\left( \lambda ,\mu \right) }\simeq \left.
G\circledS \mathfrak{g}_{2}\right/ \left( G\circledS \mathfrak{g}_{2}\right)
_{\left( \lambda ,\mu \right) }\simeq \mathcal{O}_{\left( \lambda ,\mu
\right) }
\end{equation*}%
is a coadjoint orbit in $\mathfrak{g}^{\ast }\times \mathfrak{g}^{\ast }$
through the point $\left( \lambda ,\mu \right) $ under the coadjoint action $%
Ad_{\left( g,\xi \right) }^{\ast }$ in Eq.(\ref{coad2}).

\begin{proposition}
The symplectic reduction of $^{1}TT^{\ast }G$ under the action of $%
G\circledS \mathfrak{g}_{2}$ given in Eq.(\ref{varpsi}) results in the
coadjoint orbit $\mathcal{O}_{\left( \lambda ,\mu \right) }$in $\mathfrak{g}%
^{\ast }\times \mathfrak{g}^{\ast }$ through the point $\left( \lambda ,\mu
\right) $ under the coadjoint action $Ad_{\left( g,\xi \right) }^{\ast }$ in
Eq.(\ref{coad2}) as the total space and the symplectic two-from $\Omega _{%
\mathcal{O}_{\left( \lambda ,\mu \right) }}$ in Eq.(\ref{Symp/Gxg}).
\end{proposition}

It is possible to arrive at the symplectic space $\mathcal{O}_{\left(
\lambda ,\mu \right) }$ in two steps. To this end, we first recall the
symplectic reduction of $^{1}TT^{\ast }G$ under the action of $\mathfrak{g}%
_{2}$ at $\mu \in \mathfrak{g}^{\ast }$ which results in $G\circledS
\mathfrak{g}_{3}^{\ast }$ with the canonical symplectic two-from $\Omega
_{G\circledS \mathfrak{g}_{3}^{\ast }}$. Then we consider the action of
isotropy subgroup $G_{\mu }$ on $G\circledS \mathfrak{g}_{3}^{\ast }$ and
apply the symplectic reduction which results in $\left( \mathcal{O}_{\left(
\lambda ,\mu \right) },\Omega _{\mathcal{O}_{\left( \lambda ,\mu \right)
}}\right) $. The following is the diagram summarizing this two stage
reduction of $^{1}TT^{\ast }G$

\begin{equation}
\xymatrix{&&
G\circledS\mathfrak{g}_{1}^{\ast})\circledS(\mathfrak{g}_{2}\circledS%
\mathfrak{g}_{3}^{\ast } \ar[dll]|-{\text{S.R. by } \mathfrak{g}_{2}\text{
at }\mu \text{ } } \ar[dd]|-{\text{S.R. by } G\circledS\mathfrak{g}_{2}
\text{ at }(\lambda,\mu)\text{ } } \\ G\circledS\mathfrak{g}_{3}^{\ast}
\ar[drr]|-{\text{S.R. by } G_{\mu} \text{ at } \lambda\text{ } } \\
&&\mathcal{O}_{(\lambda,\mu)} }
\end{equation}

\subsubsection{Reduction by $G\circledS \mathfrak{g}^{\ast }$}

The action of $G\circledS \mathfrak{g}_{1}^{\ast }$ on $^{1}TT^{\ast }G$ is
given by%
\begin{eqnarray}
\alpha &:&\left( G\circledS \mathfrak{g}_{1}^{\ast }\right) \times \text{ }%
^{1}TT^{\ast }G\rightarrow \text{ }^{1}TT^{\ast }G \label{alpha} \\
&:&\left( \left( h,\lambda
\right) ,\left( g,\mu ,\xi ,\nu \right) \right) \rightarrow \alpha _{\left(
h,\lambda \right) }\left( g,\mu ,\xi ,\nu \right)  \notag \\
&:&((h,\lambda),(g,\mu ,\xi ,\nu))
\rightarrow (hg,\lambda +Ad_{h^{-1}}^{\ast }\mu ,Ad_{h^{-1}}\xi
,Ad_{h^{-1}}^{\ast }\nu _{2}-ad_{Ad_{h^{-1}}\xi }^{\ast }\lambda)
\notag
\end{eqnarray}%
and, as in the case of the action of $G\circledS \mathfrak{g}_{2}$, it can
also be achieved by composition of two actions
\begin{equation*}
\alpha _{\left( h,\lambda \right) }=\alpha _{\left( h,0\right) }\circ \alpha
_{\left( e,Ad_{g}^{\ast }\lambda \right) },
\end{equation*}%
where, $\alpha _{\left( h,0\right) }$ and $\alpha _{\left( e,Ad_{g}^{\ast
}\lambda \right) }$ can be identified with the actions of $G$ and $\mathfrak{%
g}_{1}^{\ast }$ on $^{1}TT^{\ast }G$ given in Eqs.(\ref{GonTT*G}) and (\ref%
{psi}), respectively. Since both of them are symplectic, $\alpha $ is also
symplectic.

\begin{proposition}
Poisson reduction of $^{1}TT^{\ast }G$ under the action of $G\circledS
\mathfrak{g}_{1}^{\ast }$ results in $\mathfrak{g}_{2}\times \mathfrak{g}%
_{3}^{\ast }$ endowed with the bracket
\begin{equation}
\left\{ F,H\right\} _{\mathfrak{g}_{2}\times \mathfrak{g}_{3}^{\ast }}\left(
\xi ,\nu \right) =\left\langle \nu ,\left[ \frac{\delta E}{\delta \nu },%
\frac{\delta F}{\delta \nu }\right] \right\rangle .  \label{Poigg*}
\end{equation}
\end{proposition}

\begin{remark}
By considering $\mathfrak{g}^{\ast }\times \mathfrak{g}$ as the dual of Lie
algebra $\mathfrak{g}\circledS \mathfrak{g}^{\ast }$ of $G\circledS
\mathfrak{g}^{\ast }$, we derived a Lie-Poisson bracket on $\mathfrak{g}%
^{\ast }\times \mathfrak{g}$, given in Eq.(\ref{LPE}). Although $\mathfrak{g}%
^{\ast }\times \mathfrak{g}$ and $\mathfrak{g}_{2}\times \mathfrak{g}%
_{3}^{\ast }$ are isomorphic, the Lie-Poisson bracket in Eq.(\ref{LPE}) is
different from the Poisson bracket in Eq.(\ref{Poigg*}).
\end{remark}

The infinitesimal generator of $\alpha $ is associated to two tuples $\left(
\xi _{2},\nu _{2}\right) $ in the Lie algebra $\mathfrak{g}\circledS
\mathfrak{g}^{\ast }$ of $G\circledS \mathfrak{g}_{1}^{\ast }$ and is in the
form
\begin{equation}
X_{\left( \xi _{2},\nu _{2},0,0\right) }^{\ ^{1}TT^{\ast }G}=\left(
TR_{g}\xi _{2},\nu _{2}+ad_{\xi _{2}}^{\ast }\mu ,\left[ \xi ,\xi _{2}\right]
_{\mathfrak{g}},ad_{\xi _{2}}^{\ast }\nu -ad_{\xi }^{\ast }\nu _{2}\right) .
\end{equation}%
The momentum mapping $\mathbf{J}_{\text{ }^{1}TT^{\ast }G}^{G\circledS
\mathfrak{g}_{1}^{\ast }}$ is defined by the equation
\begin{equation*}
\left\langle \mathbf{J}_{\text{ }^{1}TT^{\ast }G}^{G\circledS \mathfrak{g}%
_{1}^{\ast }}\left( g,\mu ,\xi ,\nu \right) ,\left( \xi _{2},\nu _{2}\right)
\right\rangle =\left\langle \theta _{1},X_{\left( \xi _{2},\nu
_{2},0,0\right) }^{\ ^{1}TT^{\ast }G}\right\rangle =-\left\langle \nu
_{2},\xi \right\rangle +\left\langle \nu +ad_{\xi }^{\ast }\mu ,\xi
_{2}\right\rangle ,
\end{equation*}%
where $\theta _{1}$ is the potential one-form given by Eq.(\ref{1}). We
obtain
\begin{equation*}
\mathbf{J}_{\text{ }^{1}TT^{\ast }G}^{G\circledS \mathfrak{g}_{1}^{\ast }}:\
^{1}TT^{\ast }G\rightarrow Lie^{\ast }\left( G\circledS \mathfrak{g}%
_{1}^{\ast }\right) =\mathfrak{g}^{\ast }\times \mathfrak{g}:\left( g,\mu
,\xi ,\nu \right) \rightarrow \left( \nu +ad_{\xi }^{\ast }\mu ,-\xi \right)
\end{equation*}%
which can be decomposed as
\begin{equation*}
\mathbf{J}_{\text{ }^{1}TT^{\ast }G}^{G\circledS \mathfrak{g}_{1}^{\ast
}}\left( g,\mu ,\xi ,\nu \right) =\left( \mathbf{J}_{\text{ }^{1}TT^{\ast
}G}^{G}\left( g,\mu ,\xi ,\nu \right) ,\mathbf{J}_{\text{ }^{1}TT^{\ast }G}^{%
\mathfrak{g}_{1}^{\ast }}\left( g,\mu ,\xi ,\nu \right) \right)
\end{equation*}%
where $\mathbf{J}_{\text{ }^{1}TT^{\ast }G}^{G}$ and $\mathbf{J}_{\text{ }%
^{1}TT^{\ast }G}^{\mathfrak{g}_{1}^{\ast }}$ are momentum mappings in Eqs.(%
\ref{MGonTT*G}) and (\ref{Mg*onTT*G}) for the actions of $G$ and $\mathfrak{g%
}_{1}^{\ast }$ on $^{1}TT^{\ast }G$, respectively. The preimage of an
element $\left( \lambda ,\xi \right) \in \mathfrak{g}^{\ast }\times
\mathfrak{g}$ is
\begin{equation*}
\left( \mathbf{J}_{\text{ }^{1}TT^{\ast }G}^{G\circledS \mathfrak{g}%
_{1}^{\ast }}\right) ^{-1}\left( \lambda ,\xi \right) =\left\{ \left( g,\mu
,-\xi ,\nu \right) :\nu =\lambda +ad_{\xi }^{\ast }\mu \right\}
\end{equation*}%
which can be hence we may identified with the space $G\circledS \mathfrak{g}%
_{1}^{\ast }$. The isotropy subgroup of the coadjoint action of $G\circledS
\mathfrak{g}_{2}$ on $\mathfrak{g}^{\ast }\times \mathfrak{g}$ is
\begin{equation*}
\left( G\circledS \mathfrak{g}_{1}^{\ast }\right) _{\left( \lambda ,\xi
\right) }=\left\{ \left( g,\mu \right) \in G\circledS \mathfrak{g}%
_{2}:Ad_{\left( g,\mu \right) }^{\ast }\left( \lambda ,\xi \right) =\left(
\lambda ,\xi \right) \right\} ,
\end{equation*}%
where the coadjoint action is given by Eq.(\ref{Coad}). The isotropy
subgroup acts on the preimage of $\left( \lambda ,\xi \right) $ and results
in the coadjoint orbit through the point $\left( \lambda ,\xi \right) \in
\mathfrak{g}^{\ast }\times \mathfrak{g}$
\begin{equation*}
\left. \left( \mathbf{J}_{\text{ }^{1}TT^{\ast }G}^{G\circledS \mathfrak{g}%
_{1}^{\ast }}\right) ^{-1}\left( \lambda ,\xi \right) \right/ \left(
G\circledS \mathfrak{g}_{1}^{\ast }\right) _{\left( \lambda ,\xi \right)
}\simeq \left. G\circledS \mathfrak{g}_{1}^{\ast }\right/ \left( G\circledS
\mathfrak{g}_{2}\right) _{\left( \lambda ,\xi \right) }\simeq \mathcal{O}%
_{\left( \lambda ,\xi \right) }.
\end{equation*}

\begin{proposition}
Symplectic reduction of $^{1}TT^{\ast }G$ under the action of $G\circledS
\mathfrak{g}_{1}^{\ast }$ given by Eq.(\ref{alpha}) results in the coadjoint
orbit $\mathcal{O}_{\left( \lambda ,\xi \right) }$ and the symplectic
two-from $\Omega _{\mathcal{O}_{\left( \lambda ,\xi \right) }}$ in Eq.(\ref%
{SymOr}).
\end{proposition}

Similar to the reduction of $^{1}TT^{\ast }G$ by $G\circledS \mathfrak{g}%
_{2} $, we may perform symplectic reduction of $^{1}TT^{\ast }G$ with the
action of $G\circledS \mathfrak{g}_{1}^{\ast }$ by two stages. First, we
recall the symplectic reduction of $^{1}TT^{\ast }G$ with the action of $%
\mathfrak{g}_{1}^{\ast }$ at $\xi \in \mathfrak{g}$ which results in $%
G\circledS \mathfrak{g}_{3}^{\ast }$ and the canonical symplectic two-from $%
\Omega _{G\circledS \mathfrak{g}_{3}^{\ast }}$. Then, we consider the action
of isotropy subgroup $G_{\xi }$, defined in Eq.(\ref{Gxi}), on $G\circledS
\mathfrak{g}_{3}^{\ast }$ and apply symplectic reduction. This gives $%
\mathcal{O}_{\left( \lambda ,\xi \right) }$ and the symplectic two-form $%
\Omega _{\mathcal{O}_{\left( \lambda ,\xi \right) }}$. Following diagram
shows this two stage reduction of $^{1}TT^{\ast }G$%
\begin{equation}
\xymatrix{&&
G\circledS\mathfrak{g}_{1}^{\ast})\circledS(\mathfrak{g}_{2}\circledS%
\mathfrak{g}_{3}^{\ast } \ar[dll]|-{\text{S.R. by }
\mathfrak{g}_{1}^{\ast}\text{ at }\xi } \ar[dd]|-{\text{S.R. by }
G\circledS\mathfrak{g}_{1}^{\ast} \text{ at }(\lambda,\xi)} \\
G\circledS\mathfrak{g}_{2}^{\ast} \ar[drr]|-{\text{S.R. by } G_{\xi} \text{
at } \lambda} \\ &&\mathcal{O}_{(\lambda,\xi)} }
\end{equation}

We summarize diagrammatically all possible reductions of Hamiltonian
dynamics on $^{1}TT^{\ast }G$

\begin{equation}
\xymatrix{(G\circledS\mathfrak{g}_{1}^{\ast
})\circledS\mathfrak{g}_{3}^{\ast } && G\circledS\mathfrak{g}_{3}^{\ast }
\ar@{^{(}->}[rr]^{\txt{symplectic\\leaf}}
\ar@{^{(}->}[ll]_{\txt{symplectic\\leaf}}
&&G\circledS(\mathfrak{g}_{2}\circledS\mathfrak{g}_{3}^{\ast })
\\\\\mathfrak{g}_{1}^{\ast }\circledS\mathfrak{g}_{3}^{\ast }
\ar[uu]^{\txt{Poisson\\embedding}} &&(G\circledS\mathfrak{g}_{1}^{\ast
})\circledS(\mathfrak{g}_{2}\circledS\mathfrak{g}_{3}^{\ast })
\ar[uull]^{\text{P.R. by }\mathfrak{g}_{2}} \ar[uurr]_{\text{P.R. by
}\mathfrak{g}_{1}^{\ast }} \ar@/_/[dd]_{\text{S.R. by }\mathfrak{g}_{2}}
\ar@/^/[dd]^{\text{S.R. by }\mathfrak{g}_{1}^{\ast }} \ar[ll]_{\text{P.R. by
}G\circledS\mathfrak{g}_{2}} \ar[rr]^{\text{P.R. by
}G\circledS\mathfrak{g}_{1}^{\ast }} \ar[ddll]_{\text{ S.R. by
}G\circledS\mathfrak{g}_{2}} \ar[ddrr]^{\text{P.R. by
}G\circledS\mathfrak{g}_{1}^{\ast }} \ar@/_/[uu]_{\text{S.R. by
}\mathfrak{g}_{2}} \ar@/^/[uu]^{\text{S.R. by }\mathfrak{g}_{1}^{\ast }}
&&\mathfrak{g}_{2}\circledS\mathfrak{g}_{3}^{\ast}
\ar[uu]_{\txt{Poisson\\embedding}} \\\\\mathcal{O}_{(\mu,\nu)}
\ar@{_{(}->}[uu]^{\txt{symplectic\\leaf}} &&G\circledS\mathfrak{g}_{3}^{\ast
} \ar[ll]_{\text{S.R. by }G_{\mu}} \ar[rr]^{\text{S.R. by }G_{\xi}}
&&\mathcal{O}_{(\mu,\xi)} \ar@{_{(}->}[uu]_{\txt{symplectic\\leaf}}}
\label{TT*Gd}
\end{equation}

\subsubsection{Reductions of Special Symplectic Structure}

$^{1}TT^{\ast }G$ admits the trivialized special symplectic structure
\begin{equation}
\left( \ ^{1}TT^{\ast }G,\ \tau _{G\circledS \mathfrak{g}^{\ast }}\text{,}\
^{1}T^{\ast }T^{\ast }G,\ \theta _{1},\ ^{1}\Omega _{G\circledS \mathfrak{g}%
^{\ast }}^{\flat }\right) ,  \label{SS}
\end{equation}%
where, $\tau _{G\circledS \mathfrak{g}^{\ast }}$ is the tangent bundle
projection, $\theta _{1}$ is the potential one-form in Eqs.(\ref{1}) and the
diffeomorphism
\begin{equation}
^{1}\Omega _{G\circledS \mathfrak{g}^{\ast }}^{\flat }:\text{ }^{1}TT^{\ast
}G\rightarrow \text{ }^{1}T^{\ast }T^{\ast }G:\left( g,\mu ,\xi ,\nu \right)
\rightarrow \left( g,\mu ,\nu +ad_{\xi }^{\ast }\mu ,-\xi \right)
\label{Ohmb}
\end{equation}%
is a symplectic mapping satisfying the equality
\begin{equation}
\left( \text{ }^{1}\Omega _{G\circledS \mathfrak{g}^{\ast }}^{\flat }\right)
^{\ast }\theta _{\ ^{1}T^{\ast }T^{\ast }G}=\theta _{1},  \label{s1}
\end{equation}%
and $\theta _{^{1}T^{\ast }T^{\ast }G}$ is the canonical one-form in Eq.(\ref%
{thet1T*T*G}).This structure may be presented as%
\begin{equation}
\xymatrix{ ^{1}TT^{\ast}G
\ar[rr]^{^{1}\Omega^{\flat}_{G\circledS\mathfrak{g}^{\ast}}}
\ar[d]_{^{1}\tau_{G\circledS\mathfrak{g}^{\ast}}} && ^{1}T^{\ast}T^{\ast}G
\ar[dll]^{^{1}\pi_{G\circledS\mathfrak{g}^{\ast}}} \\
G\circledS\mathfrak{g}^{\ast}}  \label{SS1d}
\end{equation}
where $\pi _{G\circledS \mathfrak{g}^{\ast }}$ is the cotangent bundle
projection. See \cite{EsGu14a} for explicit constructions and more on
Tulczyjew's construction for Lie groups.

By freezing the group variable $g$ in Eq.(\ref{Ohmb}), we define the mapping%
\begin{equation*}
\Omega ^{P}:\mathfrak{g}^{\ast }\circledS \left( \mathfrak{g}\times
\mathfrak{g}^{\ast }\right) \rightarrow \mathfrak{g}^{\ast }\circledS \left(
\mathfrak{g}^{\ast }\times \mathfrak{g}\right) :\left( \mu ,\xi ,\nu \right)
\rightarrow \left( \mu ,\nu +ad_{\xi }^{\ast }\mu ,-\xi \right) .
\end{equation*}%
satisfying
\begin{equation*}
\left( \Omega ^{P}\right) ^{\ast }\left\{ H,K\right\} _{\mathfrak{g}^{\ast
}\circledS \left( \mathfrak{g}^{\ast }\times \mathfrak{g}\right) }=\left\{
H,F\right\} _{\left( \mathfrak{g}^{\ast }\times \mathfrak{g}\right)
\circledS \mathfrak{g}^{\ast }},
\end{equation*}%
where, the Poisson bracket $\left\{ \text{ },\text{ }\right\} _{\mathfrak{g}%
^{\ast }\circledS \left( \mathfrak{g}^{\ast }\times \mathfrak{g}\right) }$
on the left is given by Eq.(\ref{Poig*g*g}) and the Poisson bracket $\left\{
\text{ },\text{ }\right\} _{\left( \mathfrak{g}^{\ast }\times \mathfrak{g}%
\right) \circledS \mathfrak{g}^{\ast }}$ on the right is given by Eq.(\ref%
{Poig*gg*}). The first is obtained by Poisson reduction of $^{1}T^{\ast
}T^{\ast }G$ with the action of $G$ whereas the latter is obtained by
Poisson reduction of $^{1}TT^{\ast }G$ with the action of $G$. This implies
the reduction
\begin{equation}
\xymatrix{
\mathfrak{g}^{\ast}\circledS(\mathfrak{g}\times\mathfrak{g}^{\ast})
\ar[rr]^{\Omega^{P}} \ar[d]_{^{1}\tau_{G\circledS\mathfrak{g}^{\ast}}^{P}}
&& (\mathfrak{g}^{\ast}\times\mathfrak{g})\circledS\mathfrak{g}^{\ast}
\ar[dll]^{^{1}\pi_{G\circledS\mathfrak{g}^{\ast}}^{P}} \\
\mathfrak{g}^{\ast}}
\end{equation}
of the special symplectic structure (\ref{SS1d}) with the action of $G$.
Here, the projections $\tau _{G\circledS \mathfrak{g}^{\ast }}^{P}$ and $\pi
_{G\circledS \mathfrak{g}^{\ast }}^{P}$ are obtained from $\tau _{G\circledS
\mathfrak{g}^{\ast }}$ and $\pi _{G\circledS \mathfrak{g}^{\ast }}$ by
freezing the group variable.

Applications of symplectic reduction to $G\circledS \mathfrak{g}^{\ast }$, $%
^{1}TT^{\ast }G$ and $^{1}T^{\ast }T^{\ast }G$ with the action of $G$ at $%
\lambda \in \mathfrak{g}^{\ast }$ lead us to the reduced diagram
\begin{equation}
\xymatrix{ \mathcal{O}_{\lambda}\times\mathfrak{g}^{\ast}\times\mathfrak{g}
\ar[rr]^{\Omega^{S}} \ar[d]_{\pi} &&
\mathcal{O}_{\lambda}\times\mathfrak{g}^{\ast}\times\mathfrak{g}
\ar[dll]^{\pi} \\ \mathcal{O}_{\lambda}}
\end{equation}
where the reduced symplectic diffeomorphism is given by
\begin{equation*}
\Omega ^{S}:\mathcal{O}_{\lambda }\times \mathfrak{g}^{\ast }\times
\mathfrak{g}\rightarrow \mathcal{O}_{\lambda }\times \mathfrak{g}^{\ast
}\times \mathfrak{g}:\left( Ad_{g^{-1}}^{\ast }\lambda ,\mu ,\xi \right)
\rightarrow \left( Ad_{g^{-1}}^{\ast }\lambda ,\mu ,-\xi \right) .
\end{equation*}

\begin{remark}
The symplectic manifolds $^{1}TT^{\ast }G$ and $^{1}T^{\ast }T^{\ast }G$
have further symmetries, for example given by the actions of $\mathfrak{g}%
^{\ast }$ and $G\circledS \mathfrak{g}^{\ast }$, but we cannot perform these
reductions on the whole special symplectic structure (\ref{SS1d}) since the
only symmetry that the base manifold $G\circledS \mathfrak{g}^{\ast }$ has
is $G$.
\end{remark}

\begin{remark}
We recall from \cite{EsGu14a}\ that, on $^{1}TT^{\ast }G$, there exists
another special symplectic structure, given by five-tuple $\left(
^{1}TT^{\ast }G,^{1}T\pi _{G}\text{,}^{1}T^{\ast }TG,\ \theta _{2},^{1}\bar{%
\sigma}_{G}\right) $. Here, $\theta _{2}$ is the potential one-form given by
Eq.(\ref{2}), $^{1}\pi _{G\circledS \mathfrak{g}}$ is the cotangent bundle
projection, $^{1}T\pi _{G}$ is the projection of $^{1}TT^{\ast }G$ to its
first and third entries, and%
\begin{equation*}
^{1}\bar{\sigma}_{G}:\text{ }^{1}TT^{\ast }G\rightarrow \text{ }^{1}T^{\ast
}TG:\left( g,\mu ,\xi ,\nu \right) \rightarrow \left( g,\xi ,\nu +ad_{\xi
}^{\ast }\mu ,\mu \right)
\end{equation*}%
is a symplectic mapping. Application of symplectic reduction to this special
symplectic structure is not possible because the base manifold $G\circledS
\mathfrak{g}$ is not symplectic. When obtaining reduction of the Tulczyjew's
triplet in \cite{EsGu14a}, we have performed symplectic reductions on $%
^{1}TT^{\ast }G$ and $^{1}T^{\ast }TG$ and Lagrangian reduction to the base $%
G\circledS \mathfrak{g}$.
\end{remark}

\subsection{Lagrangian Dynamics}

Since it is a tangent bundle, we can study Lagrangian dynamics on $%
^{1}TT^{\ast }G\simeq \left( G\circledS \mathfrak{g}_{1}^{\ast }\right)
\circledS \left( \mathfrak{g}_{2}\circledS \mathfrak{g}_{3}^{\ast }\right) $%
. We define the variation of the base element $\left( g,\mu \right) \in
G\circledS \mathfrak{g}_{1}^{\ast }$ by the tangent lift of right
translation of the Lie algebra element $\left( \eta ,\lambda \right) \in
\mathfrak{g}\circledS \mathfrak{g}^{\ast }$, that is
\begin{equation*}
\delta \left( g,\mu \right) =T_{\left( e,0\right) }R_{\left( g,\mu \right)
}\left( \eta ,\lambda \right) =\left( T_{e}R_{g}\eta ,\lambda +ad_{\eta
}^{\ast }\mu \right) .
\end{equation*}%
To obtain the reduced variational principle $\delta \left( \xi ,\nu \right) $
on the Lie algebra $\mathfrak{g}_{2}\circledS \mathfrak{g}_{3}^{\ast }$ we
compute
\begin{eqnarray}
\delta \left( \xi ,\nu \right) &=&\frac{d}{dt}\left( \eta ,\lambda \right)
+[\left( \xi ,\nu \right) ,\left( \eta ,\lambda \right) ]_{\mathfrak{g}%
\circledS \mathfrak{g}^{\ast }}  \label{var} \\
&=&\frac{d}{dt}\left( \eta ,\lambda \right) +\left( [\xi ,\eta ],ad_{\eta
}^{\ast }\nu -ad_{\xi }^{\ast }\lambda \right)  \notag \\
&=&\left( \dot{\eta}+[\xi ,\eta ],\dot{\lambda}+ad_{\eta }^{\ast }\nu
-ad_{\xi }^{\ast }\lambda \right) ,  \notag
\end{eqnarray}%
for any $\left( \eta ,\lambda \right) \in \mathfrak{g}\circledS \mathfrak{g}%
^{\ast }$. Here, $[$ $,$ $]_{\mathfrak{g}\circledS \mathfrak{g}^{\ast }}$
given by Eq.(\ref{var}) stands for the Lie bracket on $\mathfrak{g}\circledS
\mathfrak{g}^{\ast }$ \cite{EsGu14a}. Assuming $\delta \left( \xi ,\nu
\right) =\left( \delta \xi ,\delta \nu \right) $ and $\delta \left( g,\mu
\right) =\left( \delta g,\delta \mu \right) $, we have the full set of
variations
\begin{equation}
\delta g=T_{e}R_{g}\eta ,\text{ \ \ }\delta \mu =\lambda +ad_{\eta }^{\ast
}\mu ,\text{ \ \ }\delta \xi =\dot{\eta}+[\xi ,\eta ]\text{, \ \ }\delta \nu
=\dot{\lambda}+ad_{\eta }^{\ast }\nu -ad_{\xi }^{\ast }\lambda
\label{VarTT*G}
\end{equation}%
for an arbitrary choice of $\left( \eta ,\lambda \right) \in \mathfrak{g}%
\circledS \mathfrak{g}^{\ast }$. Note that, these variations are the image
of the right invariant vector field $X_{\left( \eta ,\lambda ,\dot{\eta},%
\dot{\lambda}\right) }^{\ ^{1}TT^{\ast }G}$ generated by $\left( \eta
,\lambda ,\dot{\eta},\dot{\lambda}\right) .$

\begin{proposition}
For a given Lagrangian $E$ on $^{1}TT^{\ast }G$, the extremals of the action
integral are defined by the trivialized Euler-Lagrange equations%
\begin{eqnarray}
\frac{d}{dt}\left( \frac{\delta E}{\delta \xi }\right) &=&T_{e}^{\ast
}R_{g}\left( \frac{\delta E}{\delta g}\right) -ad_{\frac{\delta E}{\delta
\mu }}^{\ast }\mu +ad_{\xi }^{\ast }\left( \frac{\delta E}{\delta \xi }%
\right) -ad_{\frac{\delta E}{\delta \nu }}^{\ast }\nu \\
\frac{d}{dt}\left( \frac{\delta E}{\delta \nu }\right) &=&\frac{\delta E}{%
\delta \mu }-ad_{\xi }\frac{\delta E}{\delta \nu }  \label{UnEPTT*G}
\end{eqnarray}%
obtained by the variational principles in Eq.(\ref{VarTT*G}).
\end{proposition}

\subsubsection{Reductions}

\begin{proposition}
The Euler-Poincaré equations on the Lie algebra $\mathfrak{g}_{2}\circledS
\mathfrak{g}_{3}^{\ast }$ are
\begin{equation}
\frac{d}{dt}\left( \frac{\delta E}{\delta \xi }\right) =ad_{\xi }^{\ast
}\left( \frac{\delta E}{\delta \xi }\right) -ad_{\frac{\delta E}{\delta \nu }%
}^{\ast }\nu ,\text{ \ \ }\frac{d}{dt}\left( \frac{\delta E}{\delta \nu }%
\right) =-ad_{\xi }\frac{\delta E}{\delta \nu }.  \label{EPgg*}
\end{equation}
\end{proposition}

When the Lagrangian density $E$ in the trivialized Euler-Lagrange equations (%
\ref{UnEPTT*G}) is independent of the group variable $g\in G$, we arrive at
the Euler-Lagrange equations on $\mathfrak{g}_{1}^{\ast }\circledS \left(
\mathfrak{g}_{2}\times \mathfrak{g}_{3}^{\ast }\right) $. In addition, if
the Lagrangian $E$ depends only on fiber coordinates $E=E\left( \xi ,\nu
\right) ,$ we obtain the Euler-Poincaré equations (\ref{EPgg*}). If,
moreover, $E=E\left( \xi \right) ,$ the Euler-Poincaré equations (\ref{EPEq}%
) on $\mathfrak{g}_{2}$ results. This procedure is called the reduction by
stages \cite{CeMaRa01, HoMaRa97}.

Alternatively, the Lagrangian density $E$ in trivialized Euler-Lagrange
equations (\ref{UnEPTT*G}) can be independent of $\mu \in \mathfrak{g}%
_{1}^{\ast }$, that is, $E$ is invariant under the action of $\mathfrak{g}%
_{1}^{\ast }$ on $^{1}TT^{\ast }G$, then we have Euler-Lagrange equations on
$G\circledS \left( \mathfrak{g}_{2}\times \mathfrak{g}_{3}^{\ast }\right) $.
When $E=E\left( g,\xi \right) $, we have trivialized Euler-Lagrange
equations on $G\circledS \mathfrak{g}_{2}$. The following diagram summarizes
the discussion.

\begin{eqnarray}
\xymatrix{ & {\begin{array}{c}(G\circledS \mathfrak{g}_{1}^{\ast })\circledS
(\mathfrak{g}_{2}\circledS \mathfrak{g}_{3}^{\ast })\\\text{EL in
(\ref{TT*G1}-\ref{TT*G})}\end{array}} \ar[ddl]|-{\text{L.R. by }G}
\ar[ddr]|-{\text{L.R. by }\mathfrak{g}_{1}^{\ast }} \ar[dd]|-{\text{EP.R. by
}G\circledS \mathfrak{g}_{1}^{\ast }} \\\\
{\begin{array}{c}(\mathfrak{g}_{1}^{\ast }\times
\mathfrak{g}_{2})\circledS\mathfrak{g}_{3}^{\ast }\\\text{EL in
(\ref{g*gg*-})}\end{array}} &
{\begin{array}{c}(\mathfrak{g}_{2}\circledS\mathfrak{g}_{3}^{\ast
})\\\text{EP in (\ref{gg*-})}\end{array}} & {\begin{array}{c} G\circledS (
\mathfrak{g}_{2}\times\mathfrak{g}_{3}^{\ast}) \\ \text{EL in
(\ref{Ggg*-})}\end{array}} \\\\ & {\begin{array}{c}\mathfrak{g}_{2}
\\\text{EP in (\ref{g-})}\end{array}} \ar@{^{(}->}[uu]|-{\txt{canonical \\
immersion}} \ar@{^{(}->}[uul]|-{\txt{canonical \\ immersion}}
\ar@{^{(}->}[uur]|-{\txt{canonical \\ immersion}} & {\begin{array}{c}
G\circledS\mathfrak{g}_{2}\\\text{EL in (\ref{Gg-})} \end{array}}
\ar@{^{(}->}[uu]|-{\txt{canonical \\ immersion} }}  \label{TT*Gl} \\
--------------------------  \notag \\
\frac{d}{dt}\left( \frac{\delta E}{\delta \xi }\right) =T_{e}^{\ast
}R_{g}\left( \frac{\delta E}{\delta g}\right) -ad_{\frac{\delta E}{\delta
\mu }}^{\ast }\mu +ad_{\xi }^{\ast }\left( \frac{\delta E}{\delta \xi }%
\right) -ad_{\frac{\delta E}{\delta \nu }}^{\ast }\nu  \label{TT*G1} \\
\frac{d}{dt}\left( \frac{\delta E}{\delta \nu }\right) =\frac{\delta E}{%
\delta \mu }-ad_{\xi }\frac{\delta E}{\delta \nu }  \label{TT*G} \\
\frac{d}{dt}\left( \frac{\delta E}{\delta \xi }\right) =T_{e}^{\ast
}R_{g}\left( \frac{\delta E}{\delta g}\right) +ad_{\xi }^{\ast }\left( \frac{%
\delta E}{\delta \xi }\right) -ad_{\frac{\delta E}{\delta \nu }}^{\ast }\nu ,%
\text{ \ \ }\frac{d}{dt}\left( \frac{\delta E}{\delta \nu }\right) =-ad_{\xi
}\frac{\delta E}{\delta \nu }  \label{Ggg*-} \\
\frac{d}{dt}\left( \frac{\delta E}{\delta \xi }\right) =T_{e}^{\ast
}R_{g}\left( \frac{\delta E}{\delta g}\right) +ad_{\xi }^{\ast }\left( \frac{%
\delta E}{\delta \xi }\right)  \label{Gg-} \\
\frac{d}{dt}\left( \frac{\delta E}{\delta \xi }\right) =-ad_{\frac{\delta E}{%
\delta \mu }}^{\ast }\mu +ad_{\xi }^{\ast }\left( \frac{\delta E}{\delta \xi
}\right) -ad_{\frac{\delta E}{\delta \nu }}^{\ast }\nu ,\text{ \ \ }\frac{d}{%
dt}\left( \frac{\delta E}{\delta \nu }\right) =\frac{\delta E}{\delta \mu }%
-ad_{\xi }\frac{\delta E}{\delta \nu }  \label{g*gg*-} \\
\frac{d}{dt}\left( \frac{\delta E}{\delta \xi }\right) =ad_{\xi }^{\ast
}\left( \frac{\delta E}{\delta \xi }\right) -ad_{\frac{\delta E}{\delta \nu }%
}^{\ast }\nu ,\text{ \ \ }\frac{d}{dt}\left( \frac{\delta E}{\delta \nu }%
\right) =-ad_{\xi }\frac{\delta E}{\delta \nu }  \label{gg*-} \\
\frac{d}{dt}\left( \frac{\delta E}{\delta \xi }\right) =ad_{\xi }^{\ast
}\left( \frac{\delta E}{\delta \xi }\right)  \label{g-}
\end{eqnarray}
\newpage
\section{Summary, Discussions and Prospectives}

On the tangent bundle $^{1}TTG=\left( G\circledS \mathfrak{g}_{1}\right)
\circledS \left( \mathfrak{g}_{2}\circledS \mathfrak{g}_{3}\right) $, the
trivialized Euler-Lagrange equations are derived. The reduction of this
dynamics under the actions of the subgroups $G,$ $\mathfrak{g}_{1}$, $%
G\circledS \mathfrak{g}_{1}$ give the trivialized Euler-Lagrange equations
on $\mathfrak{g}_{1}\circledS \left( \mathfrak{g}_{2}\times \mathfrak{g}%
_{3}\right) ,$ $G\circledS \left( \mathfrak{g}_{2}\times \mathfrak{g}%
_{3}\right) ,$ $\mathfrak{g}_{1}\circledS \mathfrak{g}_{3},$ respectively.
The reductions due to the actions of $G\circledS \mathfrak{g}_{1}$ and $%
\left( G\circledS \mathfrak{g}_{1}\right) \circledS \mathfrak{g}_{2}$ result
with Euler-Poincaré equations on $\mathfrak{g}_{2}\circledS \mathfrak{g}_{3}$
and $\mathfrak{g}_{2}$, respectively. All these are summarized in the
diagram (\ref{TTGd}).

For both of the cotangent bundles $^{1}T^{\ast }TG=\left( G\circledS
\mathfrak{g}_{1}\right) \circledS \left( \mathfrak{g}_{2}^{\ast }\times
\mathfrak{g}_{3}^{\ast }\right) $ and $^{1}T^{\ast }T^{\ast }G=\left(
G\circledS \mathfrak{g}_{1}^{\ast }\right) \circledS \left( \mathfrak{g}%
_{2}^{\ast }\times \mathfrak{g}_{3}\right) ,$ the Hamilton's equations are
written. The symplectic and Poisson reductions of $^{1}T^{\ast }TG$ are
performed under the actions of $G,$ $\mathfrak{g}_{1}$ and $G\circledS
\mathfrak{g}_{1}$ and diagramed in (\ref{T*TG}). On $^{1}T^{\ast }T^{\ast }G$%
, the reductions are performed under the actions of $G,$ $\mathfrak{g}%
_{1}^{\ast }$ and $G\circledS \mathfrak{g}_{1}^{\ast }$. For this we refer
to the diagram in (\ref{T*T*G}).

On the Tulczyjew's symplectic space $^{1}TT^{\ast }G=\left( G\circledS
\mathfrak{g}_{1}^{\ast }\right) \circledS \left( \mathfrak{g}_{2}\circledS
\mathfrak{g}_{3}^{\ast }\right) $ is a tangent bundle carrying Tulczyjew's
symplectic two-form. Both of the Hamilton's and the Euler-Lagrange equations
on $^{1}TT^{\ast }G$ are computed. For the Hamilton's equations, the
symplectic and the Poisson reductions are performed with the actions of $G,$
$\mathfrak{g}_{1}^{\ast }$, $\mathfrak{g}_{2}$, $G\circledS \mathfrak{g}_{2}$
and $G\circledS \mathfrak{g}_{1}^{\ast }$, and summarized in the diagram (%
\ref{TT*Gd}). On the other hand, the Lagrangian reductions are performed
with the subgroup actions $G$, $\mathfrak{g}_{1}^{\ast }$, $G\circledS
\mathfrak{g}_{1}^{\ast }$, $G\circledS \left( \mathfrak{g}_{1}^{\ast }\times
\mathfrak{g}_{2}\right) $. For the Lagrangian reductions of $^{1}TT^{\ast }G$%
, we refer the diagram (\ref{TT*Gl}).

After the Hamiltonian reductions of the Tulczyjew's symplectic space $%
^{1}TT^{\ast }G$ are achieved, the next question could be generalization of
this to the symplectic reduction of tangent bundle of a symplectic manifold
with the lifted symplectic structure. This may be the first step toward more
general studies on the reduction of the special symplectic structures and
the reduction of Tulczyjew's triplet with configuration manifold $\mathcal{Q}
$.

\end{document}